\documentclass[preprint]{elsarticle}
\usepackage{amssymb}
\usepackage{amssymb,amsmath,mathrsfs,bm,xcolor}
\usepackage{cases}
\usepackage{array}
\usepackage{wasysym,latexsym}
\usepackage{graphicx,color,algorithm,algpseudocode}
\usepackage{booktabs,makecell,multirow}
\usepackage{epstopdf,epsfig}
\usepackage{diagbox}
\usepackage{float}
\usepackage{mathtools}
\usepackage{setspace}
\usepackage{amsthm}

\newtheorem{thm}{Theorem}
\newtheorem{lemma}{Lemma}
\newtheorem{remark}{Remark}

\newtheorem{assumption}{Assumption}
\usepackage{setspace}
\usepackage{graphicx}
\usepackage{enumerate}
\usepackage{algorithm,algpseudocode}
\usepackage{algpseudocode}   
\usepackage{caption}
\usepackage{subcaption}
\usepackage{adjustbox}
\usepackage[absolute,overlay]{textpos}
\usepackage{geometry}
\usepackage{url}
\usepackage{xcolor}
\usepackage[title]{appendix}

\usepackage[colorlinks,linkcolor=black,anchorcolor=black,linktocpage=true,urlcolor=black,citecolor=black]{hyperref}

\usepackage{enumitem} 
\usepackage{hyperref} 

\newcommand{\msc}[1]{\textbf{MSC Code}#1}

\journal{ }
\graphicspath{{figures/}}
\begin{document}

	\begin{frontmatter}
		\title{A New Primal-Dual Algorithm with Convex Combination and Extrapolation for Convex-Concave Saddle point Problems with Nonlinear Coupling Term}
		\author[1]{Jialong Li}
		\ead{ailijia0417@163.com}
		
		\author[1]{Zexian Liu\corref{cor1}}
		\ead{liuzexian2008@163.com}

        \author[3]{Xiaokai Chang}
		\ead{xkchang@lut.edu.cn}
		
		\address[1]{School of Mathematics and Statistics, Guizhou University, Guiyang 550025, P.R. China}
        \address[3]{School of Science, Lanzhou University of Technology, Lanzhou 730050, P.R. China}
		
		\cortext[cor1]{Corresponding author}
		\begin{abstract}
            Convex-concave saddle point problems with a nonlinear coupling term arise in signal processing, machine learning, robust optimization, and generative modeling, among other areas. In the bilinear setting, the evaluation of the mapping at the extrapolated point in the dual update coincides with the extrapolation of its evaluations at the current and previous primal points. For nonlinear mappings, this property generally fails. This discrepancy, together with the mismatch between proximal reference points and coupling-gradient evaluation points in some accelerated primal-dual methods, may complicate algorithmic design and significantly reduce the practical effectiveness of primal–dual methods when such methods are extended from bilinear to nonlinearly coupled problems. To address these difficulties, we propose a new primal-dual algorithm, termed PDAce, which combines a convex combination strategy with extrapolation in the nonlinear mapping space. Specifically, the primal update evaluates the coupling gradient at a convex combination point rather than the current primal point, while the dual update incorporates an extrapolated mapping that exploits the latest primal information through a three-point extrapolation of the nonlinear mapping values. We establish the global pointwise convergence of PDAce under standard convexity assumptions, without relying on a global Lipschitz continuity or a priori boundedness of the iterate sequence. We further derive an $\mathcal{O}(1/N)$ ergodic convergence rate in terms of the restricted primal-dual merit function over bounded comparison sets. Under additional strong convexity assumptions on both the primal and dual components, we establish linear convergence of PDAce. Furthermore, when the primal component is strongly convex, we develop an accelerated variant, termed aPDAce, and prove an $\mathcal{O}(1/N^2)$ ergodic convergence rate. Numerical experiments on quadratically constrained quadratic programming, convex-concave minimax games, and exponential coupling problems demonstrate the numerical stability and computational efficiency of the proposed algorithms compared with representative existing methods on the tested problems.
		\end{abstract}
	\end{frontmatter}
    \msc{: 49M29, 65K10, 65Y20, 90C25}

	\section{Introduction}\label{section1}
    We consider the following convex–concave saddle point problem with a nonlinear coupling term $\langle g(x), y \rangle$:
	\begin{equation} \label{prob1}
		\min_{x\in\mathcal{X}} \max_{y\in\mathcal{Y}}  \mathcal{L}(x, y) := f(x) + h(x) + \langle g(x), y \rangle - H^*(y) \buildrel \Delta \over = F(x) +\Phi(x,y) - H^*(y) ,
	\end{equation}
	where $x \in \mathcal{X} \subseteq \mathbb{R}^n$ is the primal variable and $y \in \mathcal{Y}\subseteq \mathbb{R}^m$ is the dual variable.
    Here, $\mathbb{R}^n$ and $\mathbb{R}^m$ are finite--dimensional Euclidean spaces equipped with the inner product $\langle \cdot, \cdot \rangle$ and induced norm $\|\cdot\| = \sqrt{\langle \cdot, \cdot \rangle}$; $H : \mathbb{R}^m \to (-\infty, +\infty]$, $f: \mathbb{R}^n \to (-\infty, +\infty]$ and $h : \mathbb{R}^n \to (-\infty, +\infty]$ are all extended-real-valued, proper, closed, and convex functions; \textcolor{black}{the Legendre--Fenchel conjugate $H^*$ associated with $H$ is defined for every $y \in \mathbb{R}^m$ as $H^*(y) = \sup_{u \in \mathbb{R}^m} \big\{ \langle y, u \rangle - H(u) \big\}$; the mapping $\Phi : \operatorname{dom}(F) \times \operatorname{dom}(H^*) \to \mathbb{R}$ is assumed to be continuous and satisfy mild differentiability assumptions; moreover, $\Phi$ is convex with respect to (w.r.t.) its first variable and affine with respect to the second variable, where $\operatorname{dom}(F)$ denotes the effective domain of the functional $F$.}
    The optimization problem \eqref{prob1} is closely related to structured convex optimization and saddle point formulations arising in signal processing, image reconstruction, compressive sensing, statistical estimation, and machine learning \cite{w17,w1,w17_1,w19,w18}. Related optimization structures also arise in kernel matrix learning for support vector machines \cite{w14} and maximum-margin clustering \cite{w28}. In particular, kernel matrix learning admits minimax formulations involving the kernel matrix and classifier dual variables. However, whether these formulations fit the specific convex-linear coupling structure of \eqref{prob1} depends on the choice of variables and the corresponding structural assumptions.
    In robust optimization, certain minimax formulations with convex dependence on the decision variable and affine dependence on the uncertainty variable can be represented within the framework of \eqref{prob1} \cite{w21}. Minimax formulations also arise in Wasserstein generative adversarial networks (WGANs), whose underlying framework is closely related to Wasserstein distances and optimal transport \cite{w20,w22}. These applications motivate the study of structured saddle point problems, although general neural-network parameterizations do not necessarily satisfy the convex-concave assumptions imposed in the present work.
    Related primal-dual, extragradient, and first-order optimization methods have been extensively investigated in convex optimization, variational inequalities, and imaging problems \cite{w25,w26,w27}. Proximal point methods and approximation techniques have also been developed for constrained nonconvex optimization and minimax problems \cite{w29,w30,w31}. These studies provide further motivation for developing efficient primal-dual algorithms, while the convergence analysis in the present paper focuses on convex-concave saddle point problems with nonlinear convex-linear coupling.
    
    A noteworthy special case arises if the coupling term takes the bilinear form $\Phi(x,y)=\langle Kx,y\rangle$ for a fixed matrix $K\in\mathbb{R}^{m\times n}$. In this setting, the generalized model \eqref{prob1} reduces to the standard convex-concave saddle point formulation with bilinear coupling.
        
	\subsection{Notation}
    Let $\phi: \mathbb{R}^m \to \mathbb{R} \cup \{+\infty\}$ be proper, closed, and convex on $\mathbb{R}^m$. For any $x \in \mathbb{R}^m$, we denote the gradient and subdifferential of $\phi$ at $x$ by $\nabla \phi(x)$ and $\partial \phi(x)$, respectively, where the gradient is defined only at differentiable points. The subdifferential of $\phi$ at the point $x$ is given by the following definition:
    \[
    \partial \phi(x) := \left\{ \xi \in \mathbb{R}^m : \phi(y) \ge \phi(x) + \langle \xi, y - x \rangle, \ \forall y \in \operatorname{dom}(\phi) \right\}.
    \]
    
    For any positive parameter $\lambda > 0$, the proximal mapping associated with $\lambda \phi$ is given by
    \[
    \operatorname{Prox}_{\lambda \phi}(x) = \operatorname*{argmin}_{y\in\mathbb{R}^m} \left\{ \phi(y) + \frac{1}{2\lambda} \| y - x \|^2 \right\}.
    \]
    The strong convexity of the quadratic regularization term guarantees that this proximal mapping is uniquely defined for all $x \in \mathbb{R}^m$.
    
    We further introduce the standard smoothness condition for $\phi$. If $\phi$ is differentiable on $\operatorname{int}(\operatorname{dom}(\phi))$ and that there exists a constant $L_{\phi} \ge 0$ such that
    \[
    \| \nabla \phi(x) - \nabla \phi(y) \| \le L_{\phi} \| x - y \|, \quad \forall x,y \in \operatorname{int}(\operatorname{dom}(\phi)),
    \]
    then $\phi$ is said to be $L_{\phi}$-smooth. In addition, the gradient $\nabla\phi$ is Lipschitz continuous on bounded subsets if, for every bounded subset $\mathcal{B}\subset \operatorname{int}(\operatorname{dom}(\phi))$, there exists a constant $L_{\mathcal{B}}$ such that the corresponding inequality holds on $\mathcal{B}$. By convention, we set $1/0 = \infty$ whenever this occurs.
	
	\subsection{Related algorithms}
    We now review some closely related literature. The Arrow–Hurwicz method~\cite{w4} was originally proposed to handle the bilinearly coupled case of our target minimax model \eqref{prob1}, in which the coupling term takes the bilinear form $\Phi(x,y)=\langle Kx,y\rangle$. This algorithm alternates between minimizing the objective w.r.t. $x$ and maximizing it w.r.t. $y$; proximal regularization is incorporated into each update to exploit the most recent iterates at every iteration. Unfortunately, convergence of the vanilla Arrow–Hurwicz method cannot be guaranteed for general saddle point problems, and rigorous counterexamples are provided in~\cite{w32,w33}. Two distinct fixes were subsequently developed to address this limitation: Korpelevich~\cite{w25} incorporated extrapolation steps into the base iteration, whereas Popov~\cite{w35} used optimistic gradient correction as an alternative modification strategy. A seminal extrapolation-based primal-dual framework named FOPDA was proposed by Chambolle and Pock in~\cite{w1,w2}. Later, Condat~\cite{w3} extended the primal-dual splitting framework to accommodate an additional smooth convex function with a Lipschitz-continuous gradient. By setting the relaxation parameter to one and neglecting computational errors, the resulting iterative scheme can be expressed as follows:
	\begin{equation}\label{a1}
		\left\{
		\begin{aligned}
			x^k &= \mathrm{Prox}_{\tau f}\left(x^{k-1} - \tau \big(\nabla h(x^{k-1}) + K^\top y^{k-1}\big)\right), \\
            \bar{x}^k &= x^k + \theta\left(x^k - x^{k-1}\right), \\
            y^k &= \mathrm{Prox}_{\sigma H^*}\left(y^{k-1} + \sigma K \bar{x}^k\right),
		\end{aligned}
		\right.
	\end{equation}
    where $\theta = 1$. It is worth noting that the formulations in \eqref{a1} involve a bilinear coupling function. Nevertheless, many practical functions encountered in real-world applications are nonlinear rather than bilinear. Based on the FOPDA, Zhu et al.~\cite{w13} developed efficient algorithms for solving problem \eqref{prob1}, and established convergence rate analysis of this method. Recently, Hamedani and Aybat~\cite{w11} proposed the following iterative scheme for a more general form of problem \eqref{prob1}:
	\begin{equation}\label{a2}
		\left\{
		\begin{aligned}
			&s^{k} = (1+\theta_k)\nabla_y \Phi(x^{k}, y^{k}) - \theta_k \nabla_y \Phi(x^{k-1}, y^{k-1}), \\
            &y^{k+1} = \mathrm{Prox}_{\sigma_k H^*}\left( y^k + \sigma_k s^k \right), \\
            &x^{k+1} = \mathrm{Prox}_{\tau_k f}\left( x^k - \tau_k (\nabla h(x^k) + \nabla_x \Phi(x^{k}, y^{k+1})) \right),
		\end{aligned}
		\right.
	\end{equation}
	where $\theta_k \in(0,1]$. Clearly, the iterative scheme \eqref{a2} can be regarded as a natural extension of the classical Chambolle--Pock primal-dual algorithm from bilinear settings to general nonlinear coupling scenarios. In addition, Hamedani and Aybat  proposed an accelerated primal-dual algorithm (APD) and its backtracking variant (APDB) \cite{w11}. Both algorithms are covered by their generic accelerated primal-dual framework (GAPD). When the strong convexity parameter $\mu=0$, the APD parameter update reduces to the nonaccelerated case. In their theoretical analysis, the $\mathcal{O}(1/N^2)$ convergence rate is guaranteed when the primal component is strongly convex and the coupling term is affine with respect to the dual variable.
	
	In another research direction, Chang and Yang~\cite{chang2021golden} proposed the GRPDA for bilinear coupling problems by incorporating Malitsky’s Golden Ratio Algorithm (GRA)~\cite{w7} into the classical FOPDA framework to replace the extrapolation points with convex combination points. Related acceleration mechanisms have also been investigated in the context of linearized alternating direction methods of multipliers \cite{w8}. Zhou et al.~\cite{w12} generalized this method to the extended golden-ratio primal-dual algorithm (E-GRPDA). The corresponding iterative scheme is stated as follows:
	\begin{equation}\label{a3}
		\left\{
		\begin{aligned}
			z^k &= \frac{\psi - 1}{\psi} x^{k-1} + \frac{1}{\psi} z^{k-1}, \\
            x^k &= \mathrm{Prox}_{\tau f}\left(z^k - \tau \nabla h(x^{k-1}) - \tau K^\top y^{k-1}\right), \\
            y^k &= \mathrm{Prox}_{\sigma H^*}\left(y^{k-1} + \sigma K x^k\right),
		\end{aligned}
		\right.
	\end{equation}
	where $\psi \in \left(1, \frac{1+\sqrt{5}}{2}\right]$. Convergence of this algorithm is guaranteed on an enlarged parameter set, where the reciprocal of the convex combination parameter is constrained from above by the golden ratio. The algorithm achieves an ergodic convergence rate of $\mathcal{O}(1/N)$, and it can be further accelerated to $\mathcal{O}(1/N^2)$ under strong‑convexity conditions. Subsequently, Chang and Yang developed a linesearch version in \cite{w6}, referred to as GRPDAL, to resolve the parameter selection problem.
	
	Recently, Nie and Long \cite{w10} extended this method to general coupling settings and established its global convergence along with a sublinear ergodic convergence rate of $\mathcal{O}(1/N)$ in general convex-concave scenarios. They adopt the same notation GRPDA for the generalized fixed‑step scheme as in \cite{chang2021golden}, and the corresponding iterative updates are stated as follows:
	\begin{equation}\label{a4}
		\left\{
		\begin{aligned}
			z^k &= \alpha z^{k-1} + (1-\alpha) x^{k-1},\ \ (\alpha = 1/\psi), \\
            x^k &= \mathrm{Prox}_{\tau f}\big(z^k - \tau \nabla h(x^{k-1}) - \tau \nabla_x \Phi(x^{k-1}, y^{k-1})\big), \\
            y^k &= \mathrm{Prox}_{\sigma H^*}\big(y^{k-1} + \sigma \nabla_y \Phi(x^k, y^{k-1})\big),
		\end{aligned}
		\right.
	\end{equation}
	where $\alpha \in \left[\frac{\sqrt{5}-1}{2}, 1\right)$. Within this framework, Chang et al.~\cite{w5} proposed a linesearch algorithm based on GRPDA (denoted as PDAc-L) for solving general saddle point problems. Based only on the local Lipschitz assumption of the gradient, PDAc-L attains an $\mathcal{O}(1/N)$ convergence rate under general convex-concave conditions. Furthermore, an accelerated variant of PDAc-L is proposed under  an additional strong convexity assumption, which achieves an $\mathcal{O}(1/N^2)$ ergodic convergence rate.
	
	The above developments demonstrate the effectiveness of extrapolation and convex combination strategies in primal-dual algorithms for convex-concave saddle point problems. Nevertheless, the presence of nonlinear coupling introduces additional challenges in algorithmic design and convergence analysis, since operations performed in the primal variable space do not generally preserve the corresponding structure in the nonlinear mapping space. This observation motivates the investigation of primal-dual schemes that coordinate convex combination and extrapolation mechanisms to better accommodate nonlinear coupling structures.
	
	\subsection{Motivations and Contributions}\label{Section1.3}
    \textcolor{black}{Although numerous acceleration strategies have been successfully integrated into primal-dual algorithms, several challenges remain when addressing convex-concave saddle point problems with a nonlinear coupling term, which are summarized as follows:}
    \begin{enumerate}
    	\item[(i)] Many classical primal-dual algorithms were originally developed for convex-concave saddle point problems with a bilinear coupling term, where the linearity of the coupling operator provides useful algebraic identities for analyzing extrapolation and convex combination steps. For nonlinear coupling of the form $\Phi(x,y)=\langle g(x),y\rangle$, however, these identities generally no longer hold, since a convex combination in the primal variable space does not necessarily yield the corresponding convex combination in the nonlinear mapping space. This discrepancy poses challenges when directly extending these primal-dual algorithms to the nonlinear case, and it necessitates the development of primal-dual schemes to overcome this obstacle.

    	\item[(ii)] When dealing with a highly nonlinear mapping $g(x)$, employing a naive alternating dual update and directly linearizing the mapping at the previous iterate $x^{k-1}$ introduces a structural lag.  The primal update takes the following form:
        \begin{equation}\label{eq:x}
            x^k = \mathrm{Prox}_{\tau f}\big(z^k - \tau \nabla h(x^{k-1}) - \tau \nabla_x \Phi(x^{k-1}, y^{k-1})\big),
        \end{equation}
        where $z^k$ denotes the convex combination point. Evidently, the reference point $z^k$ in the proximal update differs from the linearization point $x^{k-1}$. Therefore, this conventional gradient evaluation strategy does not fully reflect the local geometry of the dynamically generated convex combination points and may lead to additional approximation errors in the convergence analysis.
    
        \item[(iii)] Both APDB and GRPDA employ acceleration mechanisms on one side of the primal–dual updates. Specifically, APDB only introduces extrapolation operations in the update of the dual variable sequence $\{y^k\}$, which corresponds to the term $s^k$ in \eqref{a2}, i.e.,
        \[
        y^{k+1} = \mathrm{Prox}_{\sigma_k H^*}\left( y^k + \sigma_k s^k \right),
        \]
        where $s^k$ represents the extrapolation point. When $\Phi(x,y) = \langle g(x),y \rangle$, the vector $s^k$ is given by $s^k = (1 + \theta_k)g(x^k) - \theta_k g(x^{k-1})$. In contrast, GRPDA applies convex combination points to update the primal variable sequence $\{x^k\}$, which corresponds to $z^k$ in \eqref{a4}, namely
        \[
        z^k = \frac{\psi - 1}{\psi} x^{k-1} + \frac{1}{\psi} z^{k-1}.
        \]
        Subsequently, $x^k$ is updated by using $z^k$ as shown in \eqref{eq:x}. These observations suggest that, for nonlinear coupling problems, it is useful to coordinate the convex combination mechanism in the primal variable space with an extrapolation mechanism directly in the nonlinear mapping space.
    \end{enumerate}
    
    The nonlinear coupling term in convex-concave saddle point problems has an important impact on both algorithmic design and convergence analysis.
    Motivated by the above observations, this paper investigates primal-dual strategies tailored to the nonlinear coupling term in convex-concave saddle point problems. In particular, we propose a new primal-dual algorithm with convex combination and extrapolation (PDAce), which combines linearization at convex combination points with extrapolation in the nonlinear mapping space. The main contributions of this work are summarized as follows:
    
    \begin{enumerate}
    	\item[(i)] 
        To alleviate the geometric mismatch and delayed information exchange of existing methods, we use the convex combination point $\bar{x}^k = \alpha\bar{x}^{k-1} + (1-\alpha)x^{k }$ as the reference point for the proximal update and the evaluation point of $J_g$, namely,
        \[
        x^{k+1}
        = \text{Prox}_{\tau f}\left(\bar{x}^{k}
        - \tau\big(\nabla h(x^k) + J_g(\bar{x}^{k})^\top y^k\big)\right).
        \]
        In addition, by placing the convex combination step immediately after the primal update,  we extend extrapolation to the nonlinear mapping space through the auxiliary variable
    	\[
    	v^{k+1}
        = g(\bar{x}^{k+1})
        - \alpha\big(g(\bar{x}^{k}) - g(x^{k+1})\big),
    	\]
        which is subsequently used in the dual update
        \[
        y^{k+1}
        = \text{Prox}_{\sigma H^*}\left(y^k + \sigma v^{k+1}\right).
        \]
        This construction reorganizes the main nonlinear cross terms arising from the coupling function into a form that can be controlled by local regularity estimates and incorporated into a Lyapunov descent analysis.
    
    	\item[(ii)]
        These primal and dual updates make the convergence analysis more involved than in the bilinear case. By combining linearization at convex combination points with mapping-space extrapolation, we construct a new Lyapunov potential function and establish the global convergence of PDAce under standard convexity assumptions and local Lipschitz continuity of $J_g$ on bounded sets. More precisely, by constructing a bounded invariant region containing the generated trajectory, the required local regularity constants remain valid throughout the iterations. Based on this Lyapunov framework, we further establish an $\mathcal{O}(1/N)$ ergodic convergence rate for PDAce in terms of the restricted primal-dual merit function. Furthermore, when both the primal and dual functions are strongly convex, we prove that the iterative sequence generated by PDAce converges linearly. 
    
    	\item[(iii)]
        We develop an accelerated version (aPDAce) of PDAce under the strong convexity assumption of the primal function. By designing a noncircular parameter update strategy and establishing the corresponding weighted Lyapunov descent, we prove that aPDAce achieves an $\mathcal{O}(1/N^2)$ ergodic convergence rate.
    
    	\item[(iv)] 
        Numerical comparisons with several state-of-the-art methods, including APD \cite{w11}, GRPDA \cite{w10}, and PDAc-L \cite{w5}, are performed on three types of convex-concave saddle point problems with a nonlinear coupling term, confirming the effectiveness of the proposed methods.
    \end{enumerate}
	
	\subsection{Organization}
	The remainder of this paper is organized as follows. 
    Section \ref{section2} introduces the standing assumptions, basic notation, and several auxiliary identities used in the convergence analysis.
    Section \ref{section3} presents PDAce and establishes its fundamental one-step estimate, invariant-region property, global pointwise convergence, and $\mathcal{O}(1/N)$ ergodic convergence rate. The same section further proves linear convergence under strong convexity of both $f$ and $H^*$. 
    Section \ref{section4} develops the accelerated variant aPDAce under strong convexity of the primal function and establishes its $\mathcal{O}(1/N^2)$ ergodic convergence rate. 
    Section \ref{section5} reports numerical experiments on quadratically constrained quadratic programming, convex-concave minimax games, and problems with exponential coupling. 
    Finally, Section \ref{section6} concludes the paper and discusses possible future research directions.
	
	\section{Preliminaries}\label{section2}
	\subsection{Fundamental assumptions} 
    The algorithms developed in this paper are based on the following three assumptions.
	\begin{assumption}\label{ass1}
        There exists a saddle point
        $(x^*,y^*)\in\operatorname{dom}(F)\times\operatorname{dom}(H^*)$
        satisfying
        \[
            \mathcal{L}(x^*,y) \leq \mathcal{L}(x^*,y^*) \leq \mathcal{L}(x,y^*),\ \ \forall (x,y)\in\operatorname{dom}(F) \times \operatorname{dom}(H^*).
        \]
    \end{assumption}
    
    Assumption \ref{ass1} guarantees the existence of at least one finite-valued saddle point, which is a standard assumption in the convergence analysis of primal-dual methods \cite{w5,w9,w10,w11,w13}.
    In the subsequent analysis, a saddle point will serve as a reference for constructing bounded invariant regions and establishing the convergence of the proposed algorithms.
    
    \begin{assumption}\label{ass2}
    \textcolor{black}{The following conditions hold:}
        \begin{enumerate}[label=(\alph*), leftmargin=*]
            \item The functions $f$ and $H$ are proper, closed, and convex.
            The function $h:\mathbb{R}^n\to(-\infty,+\infty]$ is convex and
            continuously differentiable, with an $L_h$-Lipschitz
            continuous gradient \textup{\cite{w10, w13}}.
    
            \item The coupling function is given by
            $\Phi(x,y)=\langle g(x),y\rangle$, where
            $g:\mathbb{R}^n\to\mathbb{R}^m$ is continuously differentiable.
            For every $y\in\operatorname{dom}(H^*)$, the function
            $x\mapsto\Phi(x,y)$ is convex \textup{\cite{w5, w9}}.
    
            \item The Jacobian $J_g$ is Lipschitz continuous on every
            bounded convex subset of $\mathbb{R}^n$ \textup{\cite{w5, w13}}. More precisely,
            for any bounded convex set $\mathcal{B}\subset\mathbb{R}^n$,
            there exists a constant $\ell_{\mathcal{B}}\geq 0$ such that
            \[
            \|J_g(u)-J_g(v)\| \leq \ell_{\mathcal{B}}\|u-v\|, \quad u,v\in\mathcal{B}.
            \]
        \end{enumerate}
    \end{assumption}
    
    \textcolor{black}{Assumption~\ref{ass2} imposes only local regularity on the nonlinear coupling. In particular, it does not require $g$ or
    $J_g$ to be globally Lipschitz continuous. To see the implications
    of this assumption, let $\mathcal{B}\subset\mathbb{R}^n$ be an
    arbitrary bounded convex set. Since $J_g$ is continuous, the
    constant}
    \[
    M_{\mathcal{B}}:=\sup_{z\in\mathcal{B}}\|J_g(z)\|
    \]
    is finite. Consequently, for any $u,v\in\mathcal{B}$, $\|g(u)-g(v)\|\leq M_{\mathcal{B}}\|u-v\|$. Moreover, using
    $\nabla_x\Phi(x,y)=J_g(x)^\top y$, we obtain
    \begin{equation}\label{local_lipschitz}
        \begin{aligned}
            \|\nabla_x\Phi(u,y)
            -\nabla_x\Phi(v,\hat y)\|\leq
            \ell_{\mathcal{B}}\|y\|\|u-v\|
            +M_{\mathcal{B}}\|y-\hat y\|,
        \end{aligned}
    \end{equation}
    \textcolor{black}{for all $u,v\in\mathcal{B}$ and
    $y,\hat y\in\operatorname{dom}(H^*)$.
    Thus, the partial gradient of $\Phi$ admits uniform Lipschitz
    bounds whenever the primal and dual variables are restricted to
    bounded sets. Results analogous to \eqref{local_lipschitz} can be found in \cite{w13}. In addition, the convexity of $\Phi(\cdot,y)$ and \eqref{local_lipschitz} imply}
    \begin{equation}\label{lipschitz}
        \begin{aligned}
            0 \leq \Phi(u,y)-\Phi(v,y)-\langle\nabla_x\Phi(v,y),u-v\rangle \leq \frac{\ell_{\mathcal{B}}\|y\|}{2}\|u-v\|^2,
        \end{aligned}
    \end{equation}
    for all $u,v\in\mathcal{B}$ and
    $y\in\operatorname{dom}(H^*)$.
    
    \textcolor{black}{These properties will be used in the convergence analysis. In particular, the bounded regions on which the local constants $\ell_{\mathcal{B}}$ and $M_{\mathcal{B}}$ are evaluated will be specified through the invariant-region argument, rather than
    assumed in advance.}
    
    Furthermore, we introduce the following assumptions on $f$ and $H$, which hold for a large number of practically relevant problems; see, e.g., \cite[Chapter 6]{w15}.
    \begin{assumption}\label{ass3}
    Suppose that the proximal mappings of $f$ and $H^*$ either have closed-form solutions or can be computed efficiently via numerical methods.
    \end{assumption}
    	
	\subsection{Some fundamental lemmas}
	In this section, we present several fundamental lemmas that play a key role in our convergence analysis.
	\begin{lemma}[\cite{chang2021golden}, Fact 2.1]\label{lem1}
		Let $\phi: \mathbb{R}^n \rightarrow(-\infty,+\infty]$ be a proper, closed, extended-real-valued, $\mu$-strongly convex function with modulus $\mu \geq 0$. Then for any $\tau > 0 $ and $x \in \mathbb{R}^n$, it holds that $z = \operatorname{Prox}_{\tau \phi}(x)$ if and only if $\phi(y) \geq \phi(z) + \frac{1}{\tau} \langle x - z, y - z\rangle +\frac{\mu }{2}\| y - z\| ^2$ for all $y \in \mathbb{R}^n$. Moreover, for any $u,v\in\mathbb{R}^n$ and $\tau>0$, the proximal mapping satisfies
        \[
        \|\operatorname{Prox}_{\tau\phi}(u) -  \operatorname{Prox}_{\tau\phi}(v)\| \leq \frac{1}{1+\mu\tau}\|u-v\| \leq\|u-v\|.
        \]
	\end{lemma}
	
	\begin{lemma}[\cite{w5}, Fact 2.3]\label{lem2}
		For any $x, y, z \in \mathbb{R}^n$ and $\alpha \in \mathbb{R}$, there hold
        \begin{equation}\label{eq3_}
        \|x+y+z\|^2 \leq 3\|x\|^2 + 3\|y\|^2 + 3\|z\|^2
        \end{equation}
    	\begin{equation}\label{eq2}
    		2 \langle x - y, x - z \rangle = \| x - y \|^2 + \| x - z \|^2 - \| y - z \|^2,
    	\end{equation}
    	\begin{equation}\label{eq3}
    		\| \alpha x + (1 - \alpha) y \|^2 = \alpha \| x \|^2 + (1 - \alpha) \| y \|^2 - \alpha(1 - \alpha) \| x - y \|^2.
    	\end{equation}
	\end{lemma}
	
	\subsection{First-order optimality conditions and the restricted primal-dual merit function}\label{section2_3}
    Under Assumptions~\ref{ass1} and~\ref{ass2}, the saddle point
    set of problem~\eqref{prob1} is nonempty. A pair
    $(x^*,y^*)\in\operatorname{dom}(F)\times\operatorname{dom}(H^*)$
    is a saddle point if and only if it satisfies
    \begin{align*}
        0 \in \partial f(x^*)+\nabla h(x^*) + J_g(x^*)^\top y^*,\quad 0 \in \partial H^*(y^*)-g(x^*).
    \end{align*}
    Accordingly, the saddle point set is given by
    \begin{equation}\label{eq4}
        \begin{aligned}
        \Omega:=\big\{(x^*,y^*)\in \operatorname{dom}(F) \times \operatorname{dom}(H^*) \mid -\nabla h(x^*)-J_g(x^*)^\top y^*\in\partial f(x^*),\ g(x^*)\in\partial H^*(y^*) \big\}.
        \end{aligned}
    \end{equation}
    
    To measure the performance of the proposed primal-dual algorithms, let $\mathcal{S}$ be a nonempty bounded subset of $\operatorname{dom}(F)\times\operatorname{dom}(H^*)$ containing at least one saddle point, i.e., $\mathcal{S}\cap \Omega \neq \emptyset$. Define the restricted primal-dual merit function \cite{w5} by
    \begin{equation}\label{pdg}
        G_{\mathcal{S}}(x,y):=\sup_{(u,v)\in\mathcal{S}}\left\{\mathcal{L}(x,v)-\mathcal{L}(u,y)\right\}, \quad (x,y)\in\operatorname{dom}(F)\times\operatorname{dom}(H^*).
    \end{equation}
    \textcolor{black}{By the saddle point inequality in Assumption \ref{ass1}, $G_{\mathcal{S}}(x,y)\geq 0$ for every feasible pair $(x,y)$, and $G_{\mathcal{S}}(x^*,y^*)=0$ for every $(x^*,y^*)\in\Omega\cap\mathcal{S}$.}
    
    \textcolor{black}{\textcolor{black}{Building on the work of Nesterov \cite{N2007}, Lemma 3.5 of \cite{w5} justifies the choice of the function in \eqref{pdg} as a merit function. Results of this type were first established by Nesterov in Lemma 1 of~\cite{N2007} and later extended by Malitsky in Lemma 3 of~\cite{w7} to monotone variational inequality (MVI) problems.}
    Here, $(x_c,y_c)\in \operatorname{dom}(F)\times \operatorname{dom}(H^*)$ denotes a reference primal‑dual point serving as the center of the closed ball $\overline{B}((x_c,y_c),r)$. It is worth emphasizing that $(x_c,y_c)$ is only a theoretical auxiliary point for convergence analysis and is not required to be a saddle point or a solution of the considered saddle‑point problem. In addition, this point does not need to be computed in practical implementations. 
    Lemma 3.5 of \cite{w5} further states that, when $\mathcal{S}=\left(\operatorname{dom}(F)\times\operatorname{dom}(H^*)\right)\cap\overline{B}((x_c,y_c),r)$ for some $r>0$, where $\mathcal{S}$ contains at least one saddle point, the restricted gap also characterizes saddle points in the interior of the ball. Specifically, for any feasible pair $(x,y)$ satisfying $\|(x,y)-(x_c,y_c)\|<r$,  it follows that}
    \[
    G_{\mathcal{S}}(x,y)=0 \quad\Longleftrightarrow\quad (x,y)\in\Omega.
    \]
    Therefore, the restricted primal-dual gap is a suitable merit function for establishing the ergodic convergence rates of the corresponding algorithms for convex-concave saddle point problems.

	\section{The proposed algorithms and convergence analysis} \label{section3}
    In this section, we propose a primal-dual algorithm incorporating convex combination and extrapolation strategies for solving problem~\eqref{prob1}. By constructing a suitable Lyapunov potential function and establishing the boundedness of the generated iterates, we prove the global convergence of the proposed algorithm and derive an $\mathcal{O}(1/N)$ ergodic convergence rate in terms of the restricted primal-dual merit. Furthermore, under additional strong convexity and strong concavity assumptions, we establish the linear convergence of the iterates generated by PDAce.
    
    \subsection{Algorithm framework}
    
    Motivated by the discussions in Section~\ref{section1}, we propose a primal-dual algorithm that combines convex combination in the primal variable space with extrapolation in the nonlinear mapping space. The resulting algorithm, termed PDAce, is presented in Algorithm~\ref{alg:PDAce}.
    
	\begin{algorithm}[H]
    \caption{Primal-Dual Algorithm with Convex Combination and Extrapolation (PDAce)}
    \label{alg:PDAce}
    \begin{algorithmic}[1]
        \State \textbf{Initialization:} Choose parameters $\alpha\in[\frac{\sqrt{5}-1}{2},1)$, $\beta>0$, and $\tau>0$. Set $\sigma=\beta\tau$. Choose initial points $x^0\in\operatorname{dom}(F)$ and $y^0\in\operatorname{dom}(H^*)$. Set $\bar{x}^0=x^0$ and $k=0$.
        \State \textbf{Main Iteration:}
        \For{$k=0,1,2,\dots$}
            \State Step 1. Compute
            \begin{align}
                x^{k+1} &= \operatorname{Prox}_{\tau f}\left(\bar{x}^k-\tau\big(\nabla h(x^k)+J_g(\bar{x}^k)^\top y^k\big)\right). \label{eq:primal}
            \end{align}
            \State Step 2. Compute
            \begin{align}
                \bar{x}^{k+1} &= \alpha\bar{x}^k+(1-\alpha)x^{k+1}, \label{eq10}\\
                v^{k+1} &= g(\bar{x}^{k+1})-\alpha\big(g(\bar{x}^k)-g(x^{k+1})\big), \label{eq:vupdate}\\
                y^{k+1} &= \operatorname{Prox}_{\sigma H^*}\left(y^k+\sigma v^{k+1}\right). \label{eq:dualupdate}
            \end{align}
        \EndFor
    \end{algorithmic}
    \end{algorithm}

    \begin{remark}\label{rm2}
        Under Assumptions~\ref{ass1} and~\ref{ass2}, for any fixed $\alpha\in[\frac{\sqrt{5}-1}{2},1)$ and $\beta>0$, there exists some sufficiently small positive   stepsize $\tau$ such that the convergence conditions of PDAce are satisfied, with $\sigma=\beta\tau$. Although the Jacobian $J_g$ is assumed to be Lipschitz continuous only on bounded sets, the subsequent analysis establishes the existence of admissible stepsizes by first constructing bounded candidate regions independently of the final choice of $\tau$, and then deriving a strictly positive stepsize bound using the corresponding local regularity constants. This approach avoids imposing global Lipschitz conditions or assuming boundedness of the iterates in advance. The explicit sufficient stepsize conditions and the invariance of the constructed regions will be established in the convergence analysis \textcolor{black}{ (see \eqref{stepsize_PDAce} in Lemma \ref{lem5})}. 
    \end{remark}

    \begin{remark}
      The update rules in Algorithm \ref{alg:PDAce} address the issues discussed in the three motivations in Section~\ref{Section1.3}. To account for the nonlinear discrepancy between convex combinations in the primal variable space and their corresponding mapping values \textcolor{black}{(Motivation (i))}, the updates in \eqref{eq10} and \eqref{eq:vupdate} incorporate the nonlinear mapping $g$ directly into the construction of the auxiliary variable $v^{k+1}$. \textcolor{black}{It is noted that the extrapolation design for $v^{k+1}$ in \eqref{eq:vupdate} is based on three points $g(\bar x^{k+1}), g(\bar x^{k})$ and $g(x^{k+1})$, while its counterpart $s^{k+1}$ in the APDB scheme of \cite{w11} only uses two points $g(x^{k})$ and $g(x^{k+1})$.} In \eqref{eq:primal}, the convex combination point $\bar{x}^k$ serves as both the reference point of the proximal update and the evaluation point of the coupling gradient, addressing the inconsistency discussed in Motivation (ii). Meanwhile, the updates in \eqref{eq10}--\eqref{eq:dualupdate} combine convex combination in the primal variable space with extrapolation in the nonlinear mapping space, consistent with the idea in Motivation (iii).
    \end{remark}

    \begin{remark}
        When $\Phi(x,y)=\langle Kx,y\rangle$, i.e., $g(x)=Kx$, and $h(x)=0$, Algorithm~\ref{alg:PDAce} reduces to the following form:
        \[
        \left\{
        \begin{aligned}
            x^{k+1} &= \mathrm{Prox}_{\tau f}\big(\overline{x}^k-\tau K^\top y^k\big),\\
            \overline{x}^{k+1} &= \alpha\overline{x}^k+(1-\alpha)x^{k+1},\\
            v^{k+1} &= K\big(\overline{x}^{k+1}+\alpha(x^{k+1}-\overline{x}^k)\big),\\
            y^{k+1} &= \mathrm{Prox}_{\sigma H^*}\big(y^k+\sigma v^{k+1}\big).
        \end{aligned}
        \right.
        \]
        Using $\overline{x}^{k+1}=\alpha\overline{x}^k+(1-\alpha)x^{k+1}$, we obtain
        \[
        v^{k+1}=K\big(\overline{x}^{k+1}+\alpha(x^{k+1}-\overline{x}^k)\big)=K\big(\alpha\overline{x}^k+(1-\alpha)x^{k+1}
        +\alpha(x^{k+1}-\overline{x}^k)\big)=Kx^{k+1}.
        \]
        Consequently, the dual update becomes
        \[
        y^{k+1}=\mathrm{Prox}_{\sigma H^*}\big(y^k+\sigma Kx^{k+1}\big).
        \]
        Therefore, the resulting iterative scheme coincides with~\textup{(\ref{a3})}, and hence PDAce reduces to GRPDA in the bilinear case with $h=0$. Accordingly, the existing convergence theory for GRPDA can be directly applied to this special case.  In particular, the classical GRPDA stepsize condition
        \[
        \tau\sigma\|K\|^2\leq\frac{1}{\alpha}\leq\frac{\sqrt{5}+1}{2}
        \]
        is admissible for the reduced scheme. The relationship  between the sufficient stepsize conditions derived for the general nonlinear setting and the classical GRPDA condition will be further discussed in Remark~\ref{rm5}.
    \end{remark}
    
	\begin{remark}
        The auxiliary variable $v^{k+1}$ defined in~\eqref{eq:vupdate} plays an important role in the convergence analysis of PDAce. Although related convex combination and extrapolation strategies have been considered in~\cite{Chang2024ce}, the particular combination of mapping values at $\bar{x}^{k+1}$, $\bar{x}^{k}$, and $x^{k+1}$ in~\eqref{eq:vupdate} facilitates the estimation of the nonlinear remainder term $A_k$ in Lemma~\ref{lem4}. Specifically, the resulting nonlinear terms can be controlled using the local regularity of $J_g$ and the convexity of $\Phi(\cdot,y)$, which contributes to the derivation of the fundamental Lyapunov inequality.
    \end{remark}
    
	\subsection{Basic properties and related lemmas of PDAce}
    We first introduce several quantities that will be used throughout the convergence analysis. For any $(x,y)\in\operatorname{dom}(F)\times\operatorname{dom}(H^*)$, define the primal-dual residual
    \begin{equation}\label{Qk}
    Q_k(x,y):=\mathcal{L}(x^k,y)-\mathcal{L}(x,y^k).
    \end{equation}
    In particular, when $(x,y)=(x^*,y^*)\in \Omega$, it follows from Assumption \ref{ass1} that $Q_k(x^*,y^*)\ge0$. For $k\geq1$, we further define the Lyapunov potential \textcolor{black}{function}
    \begin{equation}\label{Ek}
    \mathcal{E}_k(x,y):=\frac{1}{1-\alpha}\|\bar{x}^k-x\|^2+\frac{1}{\beta}\|y^{k-1}-y\|^2+\tau L_h\|x^k-x^{k-1}\|^2.
    \end{equation}
    Here, $\beta=\sigma/\tau>0$ as specified in Algorithm~\ref{alg:PDAce}. The quantity $Q_k(x,y)$ measures the primal-dual residual with respect to an arbitrary comparison pair $(x,y)$, while $\mathcal{E}_k(x,y)$ will serve as the basic Lyapunov potential \textcolor{black}{function} in the subsequent analysis.

    From the convex combination update $\bar{x}^k=\alpha\bar{x}^{k-1}+(1-\alpha)x^k$, we obtain, for every $k\geq1$,
    \begin{equation}\label{basic_relation1}
    x^k-\bar{x}^k=\alpha(x^k-\bar{x}^{k-1}),
    \end{equation}
  \textcolor{black}{which immediately yields}
    \begin{equation}\label{basic_relation2}
    x^k-\bar{x}^{k-1}=\frac{1}{\alpha}(x^k-\bar{x}^k),
    \qquad
    \bar{x}^k-\bar{x}^{k-1}
    =\frac{1-\alpha}{\alpha}(x^k-\bar{x}^k).
    \end{equation}

    We next establish a fundamental one-step estimate for PDAce. Importantly, the following result is stated for an arbitrary comparison pair $(x,y)$ and does not require any a priori boundedness assumption on the generated iterates.
    
	\begin{lemma}\label{lem4}
        Let $\{(x^k,\bar{x}^k,y^k)\}_{k\geq0}$ be the sequence generated by Algorithm~\ref{alg:PDAce}. For any $k\geq 1$ and any $(x,y)\in\operatorname{dom}(F)\times\operatorname{dom}(H^*)$, let $\mathcal{B}$ be a bounded convex set containing $\bar{x}^{k-1}$, $x^k$, and $\bar{x}^k$, and denote $\ell:=\ell_{\mathcal{B}}$ and $M:=M_{\mathcal{B}}$. Then
        \begin{equation}\label{fundamental_ineq}
            \begin{split}
                & \mathcal{E}_{k+1}(x,y) + 2\tau Q_k(x,y)\\
                \leq& \mathcal{E}_k(x,y) - \left(\frac{1}{\alpha} - \tau\left(\ell\|y^{k-1}\|\left(\alpha + \frac{(1-\alpha)^2}{\alpha^2}\right) + 2M + \ell\|y\|\frac{1-\alpha}{\alpha}\right)\right)\|x^k - \bar{x}^k\|^2\\
                & - \left(\frac{1}{\beta} - 3\tau M\right)\|y^k - y^{k-1}\|^2 - \left(1 + \alpha - \frac{1}{\alpha}\right)\|x^{k+1} - \bar{x}^k\|^2\\
                & - \left(\frac{1}{\alpha} - \tau\left(\ell\|y^{k-1}\| + M + 2L_h\right)\right)\|x^{k+1} - x^k\|^2.
            \end{split}
        \end{equation}

    \end{lemma}
    
	\begin{proof}
		By Lemma \ref{lem1}, the following three relations hold.
		\begin{equation}\label{eq24}
        	\begin{split}
        		f(x^{k+1}) - f(x) \leq \left\langle \nabla h(x^k) + J_g(\bar{x}^{k})^\top y^k,\ x - x^{k+1} \right\rangle
        		 + \frac{1}{\tau} \left\langle x^{k+1} - \bar{x}^{k},\ x - x^{k+1} \right\rangle,
        	\end{split}
        \end{equation}
        \begin{equation}\label{eq25}
        	\begin{split}
        		f(x^k) - f(x^{k+1}) \leq \left\langle \nabla h(x^{k-1}) + J_g(\bar{x}^{k-1})^\top y^{k-1},\ x^{k+1} - x^k \right\rangle 
        		 + \frac{1}{\tau} \left\langle x^k - \bar{x}^{k-1},\ x^{k+1} - x^k \right\rangle,
        	\end{split}
        \end{equation}
        \begin{equation}\label{eq23}
        	\begin{split}
        		H^*(y^k) - H^*(y) \leq \left\langle g(\bar{x}^{k}) - \alpha\left(g(\bar{x}^{k-1}) - g(x^k)\right),\ y^k - y \right\rangle + \frac{1}{\sigma} \left\langle y^k - y^{k-1},\ y - y^k \right\rangle.
        	\end{split}
        \end{equation}
        By the convexity of the function $h(x)$, it follows that
        \begin{equation}\label{eq26}
        	\begin{split}
        		h(x^k) - h(x) \leq \left\langle \nabla h(x^k),\ x^k - x \right\rangle.
        	\end{split}
        \end{equation}

        Summing inequalities (\ref{eq24})--(\ref{eq26}), we obtain
    	\[
    	\begin{aligned}
    		& f(x^k) - f(x) + h(x^k) - h(x) + H^*(y^k) - H^*(y) \\
    		\leq & \left\langle \nabla h(x^k) + J_g(\bar{x}^{k})^\top y^k,\ x - x^{k+1} \right\rangle + \left\langle \nabla h(x^{k-1}) + J_g(\bar{x}^{k-1})^\top y^{k-1},\ x^{k+1} - x^k \right\rangle \\
    		& + \frac{1}{\tau} \left\langle x^{k+1} - \bar{x}^{k},\ x - x^{k+1} \right\rangle + \left\langle g(\bar{x}^{k}),\ y^k - y \right\rangle- \alpha \left\langle g(\bar{x}^{k-1}) - g(x^k),\ y^k - y \right\rangle \\
    		& + \frac{1}{\tau} \left\langle x^k - \bar{x}^{k-1},\ x^{k+1} - x^k \right\rangle + \frac{1}{\sigma} \left\langle y^k - y^{k-1},\ y - y^k \right\rangle \\
    		& + \left\langle \nabla h(x^k),\ x^k - x \right\rangle.
    	\end{aligned}
    	\]
        Adding $\Phi(x^k,y) - \Phi(x,y^k) = \left\langle g(x^k), y \right\rangle - \left\langle g(x), y^k \right\rangle$ to both sides of the above inequality, we obtain
    	\[
    	\begin{aligned}
    		& f(x^k) - f(x) + h(x^k) - h(x) + H^*(y^k) - H^*(y) + \left\langle g(x^k),\ y \right\rangle 
    		- \left\langle g(x),\ y^k \right\rangle\\
    		\leq & \left\langle \nabla h(x^k) + J_g(\bar{x}^{k})^\top y^k,\ x - x^{k+1} \right\rangle + \left\langle \nabla h(x^{k-1}) + J_g(\bar{x}^{k-1})^\top y^{k-1},\ x^{k+1} - x^k \right\rangle \\
    		& + \frac{1}{\tau} \left\langle x^{k+1} - \bar{x}^{k},\ x - x^{k+1} \right\rangle + \left\langle g(\bar{x}^{k}),\ y^k - y \right\rangle- \alpha \left\langle g(\bar{x}^{k-1}) - g(x^k),\ y^k - y \right\rangle \\
    		& + \frac{1}{\tau} \left\langle x^k - \bar{x}^{k-1},\ x^{k+1} - x^k \right\rangle + \frac{1}{\sigma} \left\langle y^k - y^{k-1},\ y - y^k \right\rangle \\
    		& + \left\langle \nabla h(x^k),\ x^k - x \right\rangle + \left\langle g(x^k),\ y \right\rangle 
    		- \left\langle g(x),\ y^k \right\rangle.
    	\end{aligned}
    	\]
        Using the definition of the primal-dual residual $Q_k(x,y)$ in (\ref{Qk}), the above inequality can be rewritten in the form
        \begin{equation}\label{eq21}
        	\begin{split}
        		&Q_k(x,y)\\
        		\leq & \frac{1}{\tau} \left\langle x^{k+1} - \bar{x}^{k},\ x - x^{k+1} \right\rangle + \frac{1}{\tau} \left\langle x^k - \bar{x}^{k-1},\ x^{k+1} - x^k  \right\rangle + \frac{1}{\sigma} \left\langle y^k - y^{k-1},\ y - y^k \right\rangle \\
        		& + \left\langle J_g(\bar{x}^{k})^\top y^k,\ x - x^{k+1} \right\rangle 
        		+ \left\langle J_g(\bar{x}^{k-1})^\top y^{k-1},\ x^{k+1} - x^k \right\rangle \\
        		& + \left\langle g(\bar{x}^{k}),\ y^k - y \right\rangle 
        		+ \left\langle g(x^k),\ y \right\rangle 
        		- \left\langle g(x),\ y^k \right\rangle - \alpha \left\langle g(\bar{x}^{k-1}) - g(x^k),\ y^k - y \right\rangle \\
        		& + \left\langle \nabla h(x^k),\ x - x^{k+1} \right\rangle + \left\langle \nabla h(x^{k-1}),\ x^{k+1} - x^k \right\rangle + \left\langle \nabla h(x^k),\ x^k - x \right\rangle.
        	\end{split}
        \end{equation}
        Define 
        \[
        \begin{aligned}
        	A_k(x,y) ={}& \left\langle J_g(\bar{x}^{k})^\top y^k,\ x - x^{k+1} \right\rangle 
        	+ \left\langle J_g(\bar{x}^{k-1})^\top y^{k-1},\ x^{k+1} - x^k \right\rangle \\
        	&+ \left\langle g(\bar{x}^{k}),\ y^k - y \right\rangle 
        	+ \left\langle g(x^k),\ y \right\rangle 
        	- \left\langle g(x),\ y^k \right\rangle \\
        	&- \alpha \left\langle g(\bar{x}^{k-1}) - g(x^k),\ y^k - y \right\rangle,\\
        	B_k(x,y) ={}& \left\langle \nabla h(x^k),\ x - x^{k+1} \right\rangle + \left\langle \nabla h(x^{k-1}),\ x^{k+1} - x^k \right\rangle + \left\langle \nabla h(x^k),\ x^k - x \right\rangle.
        \end{aligned}
        \]
        Multiplying both sides of inequality (\ref{eq21}) by $\tau$, and using $x^k - \bar{x}^{k-1} = \frac{1}{\alpha}(x^k - \bar{x}^{k})$ together with $\beta = \frac{\sigma}{\tau}$, we obtain
        \begin{equation}\label{eq22}
        	\begin{split}
        		&\tau Q_k(x,y)\\
        		\leq & \left\langle x^{k+1} - \bar{x}^{k},\ x - x^{k+1} \right\rangle + \frac{1}{\alpha} \left\langle x^k - \bar{x}^{k},\ x^{k+1} - x^k  \right\rangle + \frac{1}{\beta} \left\langle y^k - y^{k-1},\ y - y^k \right\rangle \\
        		& + \tau(A_k(x,y) + B_k(x,y)).
        	\end{split}
        \end{equation}
        To improve the intuition and readability of the proof, we rewrite $\langle g(x),y\rangle$ as $\Phi(x,y)$. Then, $A_k(x,y)$ can be written as:
        \[
        \begin{aligned}
            A_k(x,y) =& \left\langle \nabla_x \Phi(\bar{x}^k,y^k), x - x^{k+1} \right\rangle + \left\langle \nabla_x \Phi(\bar{x}^{k-1},y^{k-1}), x^{k+1} - x^k \right\rangle \\
            & + \Phi(\bar{x}^k,y^k) - \Phi(x,y^k) - \alpha \Phi(\bar{x}^{k-1},y^k) + \alpha \Phi(x^k,y^k)\\ & + \alpha \Phi(\bar{x}^{k-1},y) + (1-\alpha)\Phi(x^k,y) - \Phi(\bar{x}^k,y).
        \end{aligned}
        \]
        Here we employ the fact that $\Phi(x,y)$ is linear with respect to $y$.

        First, we derive an upper bound for $A_k(x,y)$ as follows. 
		\[
        \begin{aligned}
            A_k(x,y) ={}&
            \left\langle \nabla_x\Phi(\bar{x}^k,y^k), x - x^{k+1} \right\rangle
            + \left\langle \nabla_x\Phi(\bar{x}^{k-1},y^{k-1}), x^{k+1} - x^k \right\rangle \\
            &+ \Phi(\bar{x}^k,y^k) - \Phi(x,y^k)
            - \alpha\Phi(\bar{x}^{k-1},y^k) + \alpha\Phi(x^k,y^k) \\
            &+ \alpha\Phi(\bar{x}^{k-1},y) + (1 - \alpha)\Phi(x^k,y) - \Phi(\bar{x}^k,y) \\
            \leq{}&
            \left\langle \nabla_x\Phi(\bar{x}^k,y^k), \bar{x}^k - x^{k+1} \right\rangle
            + \left\langle \nabla_x\Phi(\bar{x}^{k-1},y^{k-1}), x^{k+1} - x^k \right\rangle \\
            &+ \alpha\big(\Phi(x^k,y^k) - \Phi(\bar{x}^{k-1},y^k)\big)
            + \alpha\Phi(\bar{x}^{k-1},y) + (1 - \alpha)\Phi(x^k,y) - \Phi(\bar{x}^k,y) \\
            ={}&
            -\left\langle \nabla_x\Phi(\bar{x}^k,y^{k-1}), x^k - \bar{x}^k \right\rangle
            + \alpha\big(\Phi(x^k,y^{k-1}) - \Phi(\bar{x}^{k-1},y^{k-1})\big) \\
            &+ \left\langle \nabla_x\Phi(\bar{x}^{k-1},y^{k-1}) - \nabla_x\Phi(\bar{x}^k,y^{k-1}), x^{k+1} - x^k \right\rangle \\
            &+ \left\langle \alpha\big(g(x^k) - g(\bar{x}^{k-1})\big) - J_g(\bar{x}^k)(x^k - \bar{x}^k), y^k - y^{k-1} \right\rangle \\
            &- \left\langle J_g(\bar{x}^k)^\top(y^k - y^{k-1}), x^{k+1} - x^k \right\rangle \\
            &+ \alpha\Phi(\bar{x}^{k-1},y) + (1 - \alpha)\Phi(x^k,y) - \Phi(\bar{x}^k,y) \\
            \leq{}&
            \frac{\alpha\ell\|y^{k-1}\|}{2}\|x^k - \bar{x}^k\|^2
            + \ell\|y^{k-1}\|\|\bar{x}^k - \bar{x}^{k-1}\|\|x^{k+1} - x^k\| \\
            &+ \left(\alpha M\|x^k - \bar{x}^{k-1}\| + M\|x^k - \bar{x}^k\|\right)\|y^k - y^{k-1}\| \\
            &+ M\|y^k - y^{k-1}\|\|x^{k+1} - x^k\|
            + \frac{\ell\|y\|(1 - \alpha)}{2\alpha}\|x^k - \bar{x}^k\|^2 \\
            ={}&
            \frac{\alpha\ell\|y^{k-1}\|}{2}\|x^k - \bar{x}^k\|^2
            + \ell\|y^{k-1}\|\frac{1 - \alpha}{\alpha}\|x^k - \bar{x}^k\|\|x^{k+1} - x^k\| \\
            &+ 2M\|x^k - \bar{x}^k\|\|y^k - y^{k-1}\|
            + M\|y^k - y^{k-1}\|\|x^{k+1} - x^k\| \\
            &+ \frac{\ell\|y\|(1 - \alpha)}{2\alpha}\|x^k - \bar{x}^k\|^2 \\
            \leq{}&
            \frac{\alpha\ell\|y^{k-1}\|}{2}\|x^k - \bar{x}^k\|^2
            + \frac{\ell\|y^{k-1}\|}{2}\frac{(1 - \alpha)^2}{\alpha^2}\|x^k - \bar{x}^k\|^2
            + \frac{\ell\|y^{k-1}\|}{2}\|x^{k+1} - x^k\|^2 \\
            &+ M\|x^k - \bar{x}^k\|^2
            + M\|y^k - y^{k-1}\|^2
            + \frac{M}{2}\|y^k - y^{k-1}\|^2
            + \frac{M}{2}\|x^{k+1} - x^k\|^2 \\
            &+ \frac{\ell\|y\|(1 - \alpha)}{2\alpha}\|x^k - \bar{x}^k\|^2 \\
            ={}&
            \frac{1}{2}\left( \ell\|y^{k-1}\|\left(\alpha + \frac{(1 - \alpha)^2}{\alpha^2}\right) + 2M + \ell\|y\|\frac{1 - \alpha}{\alpha}\right)\|x^k - \bar{x}^k\|^2 \\
            &+ \frac{3M}{2}\|y^k - y^{k-1}\|^2
            + \frac{\ell\|y^{k-1}\| + M}{2}\|x^{k+1} - x^k\|^2.
        \end{aligned}
        \]
		Here, the first inequality follows from the convexity of $\Phi(\cdot,y^k)$, while the subsequent estimates use the local Lipschitz bounds in~\eqref{local_lipschitz} and $\|g(u)-g(v)\| \leq M_{\mathcal B}\|u-v\|$. The term $\alpha\Phi(\bar{x}^{k-1},y) + (1-\alpha)\Phi(x^k,y) - \Phi(\bar{x}^k,y)$ is bounded by the local quadratic estimate~\eqref{lipschitz} together with~\eqref{basic_relation2}. The remaining cross terms are estimated by Young's inequality, and the final equality follows by collecting the corresponding coefficients. It should be emphasized that no properties of the point $x$ are used in the estimation of $A_k(x,y)$.\\
		
		Next, we estimate $B_k(x,y)$. By decomposition and Lipschitz continuity of the gradient of $h$, we have
		\[
		\begin{aligned}
			B_k(x,y) ={}& \left\langle \nabla h(x^k),\ x - x^{k+1} \right\rangle + \left\langle \nabla h(x^{k-1}),\ x^{k+1} - x^k \right\rangle + \left\langle \nabla h(x^k),\ x^k - x \right\rangle \\
            ={}& \left\langle \nabla h(x^k),\ x - x^k \right\rangle
			+ \left\langle \nabla h(x^k),\ x^k - x \right\rangle
			+ \left\langle \nabla h(x^k) - \nabla h(x^{k-1}),\ x^k - x^{k+1} \right\rangle\\
			\leq {}& \frac{L_h}{2}\left\| x^k - x^{k-1} \right\|^2 + \frac{L_h}{2} \left\| x^{k+1} - x^k \right\|^2.
		\end{aligned}
		\]
		This inequality is derived from Young's inequality.
        Then substituting the bounds of $A_k(x,y)$ and $B_k(x,y)$ back into \eqref{eq22} and multiplying both sides by 2 yields
        \[
        \begin{aligned}
            & 2\tau Q_k(x,y)\\
            \leq & 2\left\langle x^{k+1} - \bar{x}^{k},\ x - x^{k+1} \right\rangle + \frac{2}{\alpha} \left\langle x^k - \bar{x}^{k},\ x^{k+1} - x^k  \right\rangle + \frac{2}{\beta} \left\langle y^k - y^{k-1},\ y - y^k \right\rangle \\
            &+ \tau \left(\ell\|y^{k-1}\|\left(\alpha+\frac{(1-\alpha)^2}{\alpha^2}\right)+2M+\ell\|y\|\frac{1-\alpha}{\alpha}\right)\|x^k-\bar{x}^k\|^2 \\
            & + \tau\left( 3M\|y^k-y^{k-1}\|^2 + \big(\ell\|y^{k-1}\|+M\big)\|x^{k+1}-x^k\|^2 \right)\\
            &+ \tau\left( L_h\|x^k-x^{k-1}\|^2 + L_h\|x^{k+1}-x^k\|^2 \right).
        \end{aligned}
        \]
         Applying Lemma \ref{lem2} and relation \eqref{basic_relation2} to deal with the inner product terms on the right-hand side of the above inequality gives
        \[
        \begin{aligned}
            & 2\tau Q_k(x,y)\\
            \leq & \frac{1}{1-\alpha}\left(\|\bar{x}^k - x\|^2 - \|\bar{x}^{k+1} - x\|^2\right) - (1+\alpha)\|x^{k+1}-\bar{x}^k\|^2 \\
            & + \frac{1}{\beta}\left(\|y^{k-1}-y\|^2 - \|y^k-y\|^2 - \|y^k-y^{k-1}\|^2\right) \\
            & + \frac{1}{\alpha}\left(\|x^{k+1}-\bar{x}^k\|^2 - \|x^k-\bar{x}^k\|^2 - \|x^{k+1}-x^k\|^2\right)\\
            &+ \tau \left(\ell\|y^{k-1}\|\left(\alpha+\frac{(1-\alpha)^2}{\alpha^2}\right)+2M+\ell\|y\|\frac{1-\alpha}{\alpha}\right)\|x^k-\bar{x}^k\|^2 \\
            & + \tau\left( 3M\|y^k-y^{k-1}\|^2 + \big(\ell\|y^{k-1}\|+M\big)\|x^{k+1}-x^k\|^2 \right)\\
            &+ \tau\left( L_h\|x^k-x^{k-1}\|^2 + L_h\|x^{k+1}-x^k\|^2 \right).
        \end{aligned}
        \]
        Combining and rearranging terms, we deduce that
        \[
        \begin{aligned}
            & 2\tau Q_k(x,y)
            + \frac{1}{1-\alpha}\|\bar{x}^{k+1}-x\|^2
            + \frac{1}{\beta}\|y^k-y\|^2
            + \tau L_h\|x^{k+1}-x^k\|^2 \\
            \leq & \frac{1}{1-\alpha}\|\bar{x}^k-x\|^2
            + \frac{1}{\beta}\|y^{k-1}-y\|^2
            + \tau L_h\|x^k-x^{k-1}\|^2 \\
            & -\left( \frac{1}{\alpha} - \tau\left(\ell\|y^{k-1}\|\left(\alpha+\frac{(1-\alpha)^2}{\alpha^2}\right)+2M+\ell\|y\|\frac{1-\alpha}{\alpha}\right) \right)\|x^k-\bar{x}^k\|^2 \\
            & -\left( \frac{1}{\beta} - 3\tau M \right)\|y^k-y^{k-1}\|^2 \\
            & -\left( 1+\alpha-\frac{1}{\alpha} \right)\|x^{k+1}-\bar{x}^k\|^2 \\
            & -\left( \frac{1}{\alpha} - \tau\left(\ell\|y^{k-1}\|+M+2L_h\right) \right)\|x^{k+1}-x^k\|^2.
        \end{aligned}
        \]
        By the definition of $\mathcal{E}_k(x,y)$, we conclude the proof of this lemma.

	\end{proof}
	
	\begin{remark}
        The estimate of $A_k(x,y)$ relies essentially on three structural properties of the coupling function. First, the convexity of $\Phi(\cdot,y)$ and the local quadratic estimate in~\eqref{lipschitz} are used to control the nonlinear variation with respect to the primal variable. Second, the linearity of $\Phi(x,\cdot)=\langle g(x),\cdot\rangle$ with respect to the dual variable allows the terms evaluated at $y^k$, $y^{k-1}$, and the arbitrary test point $y$ to be rearranged so that the relevant cross-terms can be expressed through $y^k-y^{k-1}$ and differences of the nonlinear mapping $g$. Together with the local Lipschitz estimates in \eqref{local_lipschitz} and the identities in~\eqref{basic_relation2}, these terms can then be controlled by Young's inequality.
        
        The particular construction of $v^{k+1}$ in the mapping space is crucial to this rearrangement. It aligns the extrapolated nonlinear mapping terms with the convex combination structure of the primal iterates, leading to the structured bound for $A_k(x,y)$ used in Lemma~\ref{lem4}. Thus, rather than eliminating all nonlinear residuals, the proposed extrapolation makes them compatible with the Lyapunov estimate and allows them to be controlled under the local regularity assumptions.
    \end{remark}
	
	Next we construct bounded regions on which the local regularity constants in Assumption~\ref{ass2} can be uniformly controlled. Fix an arbitrary saddle point $(x^*,y^*)\in\Omega$ and a preliminary constant $\tau_0>0$. Define $d_0:=\|x^0-x^*\|$ and $
    c_0:=\|\nabla h(x^0)-\nabla h(x^*)+J_g(x^0)^\top y^0-J_g(x^*)^\top y^*\|$. Set
    $$
    \bar{\mathcal{E}}:=\frac{\left(d_0+(1-\alpha)\tau_0c_0\right)^2}{1-\alpha}+\frac{1}{\beta}\|y^0-y^*\|^2+\tau_0L_h\left[2d_0+\tau_0c_0\right]^2.
    $$
    Let $R_x:=\frac{1+\alpha}{1-\alpha}\sqrt{(1-\alpha)\bar{\mathcal{E}}}$ and $R_y:=\sqrt{\beta\bar{\mathcal{E}}}$, and define
    \begin{align}
        \mathcal{X}:=\overline{B}(x^*,R_x), \qquad
        \mathcal{Y}:=\operatorname{dom}(H^*)\cap\overline{B}(y^*,R_y).\label{localset}
    \end{align}
    We further set $B:=\|y^*\|+R_y$, $\ell:=\ell_{\mathcal{X}}$, $M:=M_{\mathcal{X}}$, and $\Lambda := \ell B\left(\alpha+\frac{1-\alpha}{\alpha^2}\right)+2M$.

    The above bounded sets are constructed as candidate regions for the primal and dual iterates. \textcolor{purple}{Notice that they are determined by the fixed saddle point $(x^*,y^*)$, the initial point, and the preliminary bound $\tau_0$, rather than by the eventual choice of the stepsize $\tau$.} The following lemma shows that, under a suitable stepsize condition, these candidate regions are invariant under the PDAce iteration. Consequently, the local constants $\ell$ and $M$ remain valid along the entire sequence, and the associated Lyapunov potential is nonincreasing.

    \begin{lemma}\label{lem5}
        Suppose that $\alpha\in[\frac{\sqrt{5}-1}{2},1)$, $\beta>0$, $\sigma=\beta\tau$, and
        \begin{equation}\label{stepsize_PDAce}
            0<\tau<\min\left\{\tau_0,\frac{1}{\alpha\Lambda},\frac{1}{3\beta M},\frac{1}{\alpha(\ell B+M+2L_h)}\right\}.
        \end{equation} Then the sequence generated by Algorithm~\ref{alg:PDAce} satisfies $x^k,\bar{x}^k\in\mathcal{X}$ and $y^k\in\mathcal{Y}$ for all $k\geq0$. Moreover, for every $k\geq1$,
        \begin{equation}\label{saddle_descent}
            \mathcal{E}_{k+1}(x^*,y^*) \leq \mathcal{E}_k(x^*,y^*) \leq \mathcal{E}_1(x^*,y^*)\leq\bar{\mathcal{E}}.
        \end{equation}
        Consequently,
        $$
        \sum_{k=1}^{\infty}\|x^k-\bar{x}^k\|^2<+\infty, \qquad
        \sum_{k=1}^{\infty}\|y^k-y^{k-1}\|^2<+\infty, \qquad
        \sum_{k=1}^{\infty}\|x^{k+1}-x^k\|^2<+\infty.
        $$
    \end{lemma}
	\begin{proof}
		We prove the lemma by mathematical induction, showing that the sequence of iterates $\{(x^k,\bar{x}^k,y^k)\}$ generated by Algorithm \ref{alg:PDAce} is bounded, i.e., $\{x^k,\bar{x}^k,y^k\} \subset \mathcal{X} \times \mathcal{X} \times \mathcal{Y}$. First, we consider the case $k=1$ and verify that $\mathcal{E}_1\le\bar{\mathcal{E}}$.
        Since $(x^*, y^*)$ is a saddle point, we obtain
        \[
        0\in \partial f(x^*)+\nabla h(x^*) + J_g(x^*)^\top y^*.
        \]
        This is equivalent to
        \[
        x^* = \mathrm{Prox}_{\tau f}\big(x^* - \tau\big(\nabla h(x^*) + J_g(x^*)^\top y^*\big)\big).
        \]
        By the nonexpansiveness of the proximal mapping in Lemma \ref{lem1}, it holds that
        \[
        \begin{aligned}
            \|x^1 - x^*\|
            &\le \Big\|x^0 - x^* - \tau\big(\nabla h(x^0)-\nabla h(x^*) + J_g(x^0)^\top y^0 - J_g(x^*)^\top y^*\big)\Big\| \le d_0 + \tau c_0 \le d_0 + \tau_0 c_0.
        \end{aligned}
        \]
        Furthermore, from $\bar{x}^1 = \alpha \bar{x}^0 + (1-\alpha) x^1 = \alpha x^0 + (1-\alpha) x^1$, we obtain
        \[
        \begin{aligned}
        \|\bar{x}^1 - x^*\| &\le \alpha \|x^0 - x^*\| + (1-\alpha)\|x^1 - x^*\| \\
        &\le \alpha d_0 + (1-\alpha)(d_0 + \tau_0 c_0) \\
        &= d_0 + (1-\alpha)\tau_0 c_0.
        \end{aligned}
        \]
        Meanwhile, using the triangle inequality, we obtain
        $$\|x^1 - x^0\| \le \|x^1 - x^*\| + \|x^0 - x^*\| \le 2d_0 + \tau_0 c_0.$$
        Therefore, we conclude that
        \begin{equation}\label{eq401}
            \mathcal{E}_1(x^*, y^*) = \frac{1}{1-\alpha}\|\bar{x}^1 - x^*\|^2 + \frac{1}{\beta}\|y^0 - y^*\|^2 + \tau L_h\|x^1 - x^0\|^2 \le \bar{\mathcal{E}}.
        \end{equation}
        By the definition of $\bar{\mathcal{E}}$, we have $\sqrt{(1-\alpha)\bar{\mathcal{E}}} \ge d_0 + (1-\alpha)\tau_0 c_0 \ge d_0$. Thus $x^0,\bar{x}^0,\bar{x}^1\in\mathcal{X}$. Moreover, from $x^1=\frac{1}{1-\alpha}(\bar{x}^1-\alpha\bar{x}^0)$, we get
        \begin{equation}\label{eq402}
            \|x^1 - x^*\| \le \frac{1}{1-\alpha}\big(\|\bar{x}^1 - x^*\| + \alpha\|\bar{x}^0 - x^*\|\big) \le \frac{1+\alpha}{1-\alpha}\sqrt{(1-\alpha)\bar{\mathcal{E}}} = R_x.
        \end{equation}
        Inequality \eqref{eq401} yields $\frac{1}{\beta}\|y^0-y^*\|^2 \le \bar{\mathcal{E}}$, which implies 
        \begin{equation}\label{eq403}
            \|y^0-y^*\|\le \sqrt{\beta\bar{\mathcal{E}}} = R_y.
        \end{equation}
        Then from \eqref{eq402} and \eqref{eq403}, we conclude that $x^1\in\mathcal{X}$ and $y^0 \in \mathcal{Y}$. Assume that for some $k\ge 1$, 
        \begin{equation}\label{math_induction_1}
            \mathcal{E}_j(x^*,y^*)\le \bar{\mathcal{E}},\quad 1\le j\le k.
        \end{equation}
        Moreover, all primal sequence points required up to the $k$-th iteration lie in $\mathcal{X}$. From the inequality \eqref{math_induction_1}, we obtain $\frac{1}{1-\alpha}\|\bar{x}^k - x^*\|^2 \le \bar{\mathcal{E}}$, which implies
        \[
        \|\bar{x}^k - x^*\| \le \sqrt{(1-\alpha)\bar{\mathcal{E}}}.
        \]
        We also have $\frac{1}{\beta}\|y^{k-1}-y^*\|^2\le \bar{\mathcal{E}}$, which yields 
        \begin{equation}\label{407}
            \|y^{k-1}-y^*\|\le\sqrt{\beta\bar{\mathcal{E}}} = R_y.
        \end{equation}
        Moreover, $x^k=\frac{1}{1-\alpha}(\bar{x}^k-\alpha\bar{x}^{k-1})$, which gives
        \begin{equation}\label{408}
            \begin{aligned}
                \|x^k - x^*\| \le \frac{1}{1-\alpha}\big(\|\bar{x}^k - x^*\| + \alpha\|\bar{x}^{k-1} - x^*\|\big) \le \frac{1+\alpha}{1-\alpha}\sqrt{(1-\alpha)\bar{\mathcal{E}}} = R_x.
            \end{aligned}
        \end{equation}
        Inequalities \eqref{407} and \eqref{408} imply that Lemma \ref{lem4} can indeed be applied with the same $\mathcal{X}$, $\mathcal{Y}$, $\ell$, $M$ at the $k$-th step. Set $(x,y)=(x^*,y^*)$ in Lemma \ref{lem4}. Since $\|y^{k-1}\|\le B$ and $\|y^*\|\le B$, the coefficient of $\|x^k-\bar{x}^k\|^2$ in \eqref{fundamental_ineq} satisfies
        $$
        \begin{aligned}
            \ell\|y^{k-1}\|\left(\alpha+\frac{(1-\alpha)^2}{\alpha^2}\right)+2M+\ell\|y^*\|\frac{1-\alpha}{\alpha}
            &\le \ell B\left(\alpha+\frac{(1-\alpha)^2}{\alpha^2}+\frac{1-\alpha}{\alpha}\right)+2M \\
            &= \ell B\left(\alpha+\frac{1-\alpha}{\alpha^2}\right)+2M = \Lambda.
        \end{aligned}
        $$
        By Lemma \ref{lem4}, we obtain
        \begin{equation}\label{eq404}
            \begin{aligned}
                &\mathcal{E}_{k+1}(x^*,y^*) + 2\tau Q_k(x^*,y^*)\\
                \leq & \mathcal{E}_k(x^*,y^*) -\left(\frac{1}{\alpha}-\tau\Lambda\right)\|x^k-\bar{x}^k\|^2 -\left(\frac{1}{\beta}-3\tau M\right)\|y^k-y^{k-1}\|^2 \\
                & -\left(1+\alpha-\frac{1}{\alpha}\right)\|x^{k+1}-\bar{x}^k\|^2 -\left(\frac{1}{\alpha}-\tau(\ell B+M+2L_h)\right)\|x^{k+1}-x^k\|^2.
            \end{aligned}
        \end{equation}
        From \eqref{stepsize_PDAce} and $\alpha\in[\frac{\sqrt{5}-1}{2},1)$, it can be derived that 
        \[
        \begin{aligned}
            \frac{1}{\alpha}-\tau\Lambda> 0,\quad \frac{1}{\beta}-3\tau M > 0,\quad 1+\alpha-\frac{1}{\alpha}\geq 0,\quad \frac{1}{\alpha}-\tau(\ell B+M+2L_h)> 0.
        \end{aligned}
        \]
        Since $(x^*,y^*)$ is a saddle point, this yields $Q_k(x^*,y^*) = \mathcal{L}(x^k,y^*) - \mathcal{L}(x^*,y^k)\ge 0$.
        Thus, from \eqref{eq404}, we can obtain
        \begin{equation}
            \mathcal{E}_{k+1}(x^*,y^*) \leq \mathcal{E}_k(x^*,y^*) \leq \mathcal{E}_1(x^*,y^*)\leq\bar{\mathcal{E}}.
        \end{equation}
        This gives $\|\bar{x}^{k+1}-x^*\|\le \sqrt{(1-\alpha)\bar{\mathcal{E}}} \le R_x$ and
        \begin{equation}\label{eq405}
            \|y^k-y^*\|\le\sqrt{\beta\bar{\mathcal{E}}} = R_y.
        \end{equation}
        Moreover, since $x^{k+1}=\frac{1}{1-\alpha}\big(\bar{x}^{k+1}-\alpha\bar{x}^k\big)$, we have
        \begin{equation}\label{eq406}
            \|x^{k+1}-x^*\|\le\frac{1+\alpha}{1-\alpha}\sqrt{(1-\alpha)\bar{\mathcal{E}}} = R_x.
        \end{equation}
        Therefore, from \eqref{eq405} and \eqref{eq406}, we conclude that $x^{k+1}\in\mathcal{X}$ and $y^k\in\mathcal{Y}$. Then, by mathematical induction, 
        \[
        \begin{aligned}
            x^k,\ \bar{x}^k\in\mathcal{X}\ \text{and}\  y^k \in \mathcal{Y},\ \forall  k\geq0.
        \end{aligned}
        \]

        Now sum inequality \eqref{eq404} from $k=1$ to $N$. We obtain
        $$
        \begin{aligned}
            &\quad\mathcal{E}_{N+1}(x^*,y^*) + 2\tau \sum_{k=1}^NQ_k(x^*,y^*) +\left(\frac{1}{\alpha}-\tau\Lambda\right)\sum_{k=1}^N\|x^k-\bar{x}^k\|^2 \\
            &+\left(\frac{1}{\beta}-3\tau M\right)\sum_{k=1}^N\|y^k-y^{k-1}\|^2 +\left(1+\alpha-\frac{1}{\alpha}\right)\sum_{k=1}^N\|x^{k+1}-\bar{x}^k\|^2 \\
            &+\left(\frac{1}{\alpha}-\tau(\ell B+M+2L_h)\right)\sum_{k=1}^N\|x^{k+1}-x^k\|^2
            \le \mathcal{E}_1(x^*,y^*)\le \bar{\mathcal{E}}.
        \end{aligned}
        $$
        Letting $N\to+\infty$ on the left-hand side of the inequality yields
        \begin{align}
            \sum_{k=1}^{\infty}\|x^k-\bar{x}^k\|^2 &\le \frac{1}{\frac{1}{\alpha}-\tau\Lambda}\mathcal{E}_1(x^*,y^*)<\infty,\label{x_k}\\
            \sum_{k=1}^{\infty}\|y^k-y^{k-1}\|^2 &\le \frac{1}{\frac{1}{\beta}-3\tau M}\mathcal{E}_1(x^*,y^*)<\infty,\\
            \sum_{k=1}^{\infty}\|x^{k+1}-x^k\|^2 &\le \frac{1}{\frac{1}{\alpha}-\tau(\ell B+M+2L_h)}\mathcal{E}_1(x^*,y^*)<\infty.
        \end{align}
        This completes the proof of the lemma.
        
	\end{proof}

    \begin{remark}\label{rm5}
        We further discuss the relation between the sufficient stepsize conditions in Lemma~\ref{lem5} and the classical GRPDA condition in the bilinear case \cite{chang2021golden}. Suppose that $g(x)=Kx$ and $h(x)\equiv0$. Then $J_g(x)=K$, and hence $\ell=0$, $M=\|K\|$, and $\Lambda=2\|K\|$. Therefore, the stepsize condition~\eqref{stepsize_PDAce} reduces to
        $$
        \tau < \min\left\{\tau_0,\frac{1}{2\alpha\|K\|},\frac{1}{3\beta\|K\|}\right\},
        $$
        where $\sigma=\beta\tau$. The last two inequalities imply
        $$
        \tau\sigma\|K\|^2 = \beta\tau^2\|K\|^2 < \frac{1}{6\alpha} < \frac{1}{\alpha}.
        $$
        Since $\alpha\in[\frac{\sqrt{5}-1}{2},1)$, it follows that
        $$
        \tau\sigma\|K\|^2 < \frac{1}{\alpha} \leq \frac{\sqrt{5}+1}{2},
        $$
        which shows that the sufficient conditions derived from our general nonlinear analysis imply the classical GRPDA stepsize condition in~\cite{chang2021golden}.
        
        The converse is generally not true, since the sufficient conditions in Lemma~\ref{lem5} are more restrictive. This conservatism arises from the unified estimates used to control nonlinear coupling terms. In the bilinear case with $h\equiv0$, PDAce coincides with GRPDA, and hence the less restrictive classical GRPDA stepsize condition can be directly used for this special case.
    \end{remark}

    \subsection{Global convergence and convergence rate}
    The boundedness and summability properties established in Lemma~\ref{lem5} allow us to characterize the cluster points of the generated sequence and further prove the convergence of the whole sequence.
    Then, building upon Lemmas \ref{lem4} and \ref{lem5}, we analyze the convergence properties of PDAce, including its global pointwise convergence as well as an ergodic convergence rate of $\mathcal{O}(1/N)$.

    \begin{thm}[Global pointwise convergence]\label{thm1}
    Suppose that Assumptions~\ref{ass1} and~\ref{ass2} hold, and let $\{(x^k,\bar{x}^k,y^k)\}_{k\geq0}$ be the sequence generated by Algorithm~\ref{alg:PDAce}. If the parameters satisfy the conditions of Lemma~\ref{lem5}, then there exists a saddle point $(\hat{x},\hat{y})\in\Omega$ such that
    $$
    x^k \to \hat{x}, \qquad \bar{x}^k \to \hat{x}, \qquad y^k \to \hat{y}.
    $$
    In particular, the whole sequence generated by PDAce converges pointwise to a saddle point of problem~\eqref{prob1}.
    \end{thm}
    
	\begin{proof}
		From Lemma \ref{lem5}, we have $\sum_{k=1}^{\infty}\|x^k-\bar{x}^k\|^2<\infty$. It follows that
        \begin{equation}\label{eq409}
            \lim_{k\to\infty}(x^k-\bar{x}^k)=0.
        \end{equation}
        Similarly, $\lim_{k\to\infty}(x^{k+1}-x^k)=0$ and $\lim_{k\to\infty}(y^k-y^{k-1})=0$.
        Thus
        \begin{equation}\label{eq410}
            \lim_{k\to\infty}(x^{k+1}-\bar{x}^k)=\lim_{k\to\infty}\big((x^{k+1}-x^k)+(x^k-\bar{x}^k)\big)=0.
        \end{equation}
        The entire sequence is bounded, so there exists a subsequence $\{(x^{k_j},y^{k_j})\}$ and some $(\hat{x},\hat{y})$ such that
        \begin{equation}\label{eq411}
            x^{k_j}\to \hat{x},\quad y^{k_j}\to \hat{y}.
        \end{equation}
        Using \eqref{eq409}, \eqref{eq410} and \eqref{eq411}, we obtain
        \begin{equation}\label{eq412}
            \bar{x}^{k_j}\to \hat{x},\quad x^{k_j+1}\to \hat{x},\quad y^{k_j-1}\to \hat{y}. 
        \end{equation}
        Moreover, from $x^k-\bar{x}^{k-1}=\frac{1}{\alpha}(x^k-\bar{x}^k)$, we get $x^k-\bar{x}^{k-1}=\frac{1}{\alpha}(x^k-\bar{x}^k)\to 0$. Hence
        \begin{equation}\label{eq413}
            \bar{x}^{k_j-1}\to \hat{x}.
        \end{equation}
        Considering the optimality condition of the update of $x^{k+1}$, it follows that
        $$
        \frac{1}{\tau}(\bar{x}^k-x^{k+1})-\nabla h(x^k)-J_g(\bar{x}^k)^\top y^k\in \partial f(x^{k+1}).
        $$
        Passing to the limit along the subsequence $k=k_j$. Since  $\lim_{k\to+\infty}(x^{k+1}-x^k)\to 0$, we have $\lim_{j\to\infty}\frac{1}{\tau}(\bar{x}^{k_j}-x^{k_j+1})=0$.
        Then by continuity, $\lim_{j\to\infty}\nabla h(x^{k_j})=\nabla h(\hat{x})$ and $\lim_{j\to\infty}J_g(\bar{x}^{k_j})^\top y^{k_j}=J_g(\hat{x})^\top \hat{y}$.
        Since $\lim_{j\to\infty}x^{k_j+1}=\hat{x}$, and the subdifferential of a proper closed convex function has a closed graph, we get
        \begin{equation}\label{eq414}
        -\nabla h(\hat{x})-J_g(\hat{x})^\top \hat{y}\in \partial f(\hat{x}).
        \end{equation}
        We now turn to the dual sequence.
        From the update rule for $v^k$, we have $v^{k_j}=g(\bar{x}^{k_j})-\alpha\big(g(\bar{x}^{k_j-1})-g(x^{k_j})\big)$.
        Using \eqref{eq410}-\eqref{eq413} and the continuity of $g$, we obtain
        $$
        g(\bar{x}^{k_j})\to g(\hat{x}),\quad g(\bar{x}^{k_j-1})\to g(\hat{x}),\quad g(x^{k_j})\to g(\hat{x}).
        $$
        Therefore, we get
        \begin{equation}\label{eq415}
            v^{k_j}\to g(\hat{x}).
        \end{equation}
        The optimality condition for the dual update is
        $$
        v^k+\frac{1}{\sigma}(y^{k-1}-y^k)\in\partial H^*(y^k).
        $$
        By $\lim_{k\to\infty}(y^k-y^{k-1})=0$, we have $\lim_{k\to\infty}\frac{1}{\sigma}(y^{k-1}-y^k)=0$.
        Combining this with \eqref{eq415} and $y^{k_j}\to\hat{y}$, and again using the closed graph property of the subdifferential, we obtain
        \begin{equation}\label{eq416}
            g(\hat{x})\in\partial H^*(\hat{y}).
        \end{equation}
        Equations \eqref{eq414} and \eqref{eq416} are equivalent to
        $$
        \begin{aligned}
            0\in\partial f(\hat{x})+\nabla h(\hat{x})+J_g(\hat{x})^\top\hat{y},\quad
            0\in\partial H^*(\hat{y})-g(\hat{x}).
        \end{aligned}
        $$
        Hence $(\hat{x},\hat{y})\in\Omega$. That is, any limit point of the sequence $\{(x^k,y^k)\}$ generated by Algorithm \ref{alg:PDAce} is a saddle point of problem \eqref{prob1}. We next show that the entire sequence converges to this limit point. \\
        
        The dual sequence satisfies $\|y^k\|\le B$ for all $k\ge 0$.
         By continuity of the norm and $y^{k_j}\to\hat y$, we have $\|\hat y\|\le B$. Take $(x,y)=(\hat{x},\hat{y})$ in Lemma \ref{lem4}. Since $(\hat{x},\hat{y})$ is a saddle point, $Q_k(\hat{x},\hat{y}) = \mathcal{L}(x^k,\hat{y}) - \mathcal{L}(\hat{x},y^k)\ge 0$.
        Moreover, $\|y^{k-1}\|\le B$ and $\|\hat{y}\|\le B$.
        Then from Lemma \ref{lem5}, we obtain
        $$
        \mathcal{E}_{k+1}(\hat{x},\hat{y})\le \mathcal{E}_k(\hat{x},\hat{y}).
        $$
        That is, $\{\mathcal{E}_k(\hat{x},\hat{y})\}$ is nonincreasing and bounded below, so it converges.
        Taking the limit along the subsequence $k_j$, we derive that
        $$
        \bar{x}^{k_j}\to \hat{x},\quad y^{k_j-1}\to \hat{y}.
        $$
        From \eqref{eq410}, $x^{k_j}-x^{k_j-1}\to 0$.
        Then, by the definition of $\mathcal{E}_k(x,y)$,
        $$
        \mathcal{E}_{k_j}(\hat{x},\hat{y})\to 0.
        $$
        Since the whole sequence $\{\mathcal{E}_k(\hat{x},\hat{y})\}$ is nonincreasing and nonnegative, and one of its subsequences converges to $0$, this implies that
        $$
        \mathcal{E}_k(\hat{x},\hat{y})\to 0.
        $$
        By the definition of $\mathcal{E}_k(\hat{x},\hat{y})$, this yields
        $$
        \|\bar{x}^k-\hat{x}\|\to 0,\quad \|y^{k-1}-\hat{y}\|\to 0.
        $$
        Combined with $x^k-\bar{x}^k\to 0$ and $y^k-y^{k-1}\to 0$, it follows that
        $$
        x^k\to \hat{x},\quad y^k\to \hat{y}.
        $$
        This completes the proof of the theorem.
        
	\end{proof}

    We next return to an arbitrary comparison pair $(x,y)$. By combining the fundamental estimate in Lemma~\ref{lem4} with the invariant-region property established in Lemma~\ref{lem5}, we obtain a one-step inequality with a summable error term.

    \begin{lemma}\label{lem6}
        Suppose that the conditions of Lemma~\ref{lem5} hold. Then, for any $(x,y)\in\operatorname{dom}(F)\times\operatorname{dom}(H^*)$ and every $k\geq1$,
        \begin{equation}\label{general_descent}
            \mathcal{E}_{k+1}(x,y) + 2\tau Q_k(x,y) \leq \mathcal{E}_k(x,y)
            + \tau\ell\frac{1-\alpha}{\alpha}(\|y\|-B)_+\|x^k-\bar{x}^k\|^2,
        \end{equation}
        where $(t)_+:=\max\{t,0\}$.
    \end{lemma}
    
    \begin{proof}
        Lemma \ref{lem5} yields $\|y^{k-1}\|\le B$ for all $k\ge 1$.
        Then the coefficient of $\|x^k-\bar{x}^k\|^2$ in Lemma \ref{lem4} is at least
        \begin{equation}\label{eq417}
            \frac{1}{\alpha}-\tau\left(\ell B\left(\alpha+\frac{(1-\alpha)^2}{\alpha^2}\right)+2M+\ell\|y\|\frac{1-\alpha}{\alpha}\right).
        \end{equation}
        By the definition of $\Lambda$, \eqref{eq417} can be rewritten as
        $$
        \frac{1}{\alpha}-\tau\Lambda-\tau\ell\frac{1-\alpha}{\alpha}\big(\|y\|-B\big).
        $$
        For the above relation, we proceed by case analysis. When $\|y\|\le B$, this yields that $\|y\|-B\le 0$.
        Thus the coefficient of $\|x^k-\bar{x}^k\|^2$ is at least $\frac{1}{\alpha}-\tau\Lambda$.
        When $\|y\|\ge B$, we get $\|y\|-B\ge 0$, and  the coefficient is bounded below by
        $$
        \frac{1}{\alpha}-\tau\Lambda-\tau\ell\frac{1-\alpha}{\alpha}\big(\|y\|-B\big).
        $$
        These two cases can be unified as
        $$
        \begin{aligned}
            \frac{1}{\alpha}-\tau\left(\ell\|y^{k-1}\|\left(\alpha+\frac{(1-\alpha)^2}{\alpha^2}\right)+2M+\ell\|y\|\frac{1-\alpha}{\alpha}\right)
            &\ge \frac{1}{\alpha}-\tau\Lambda-\tau\ell\frac{1-\alpha}{\alpha}\big(\|y\|-B\big)_+.
        \end{aligned}
        $$
        Therefore,
        $$
        \begin{aligned}
            &-\left(\frac{1}{\alpha}-\tau\left(\ell\|y^{k-1}\|\left(\alpha+\frac{(1-\alpha)^2}{\alpha^2}\right)+2M+\ell\|y\|\frac{1-\alpha}{\alpha}\right)\right)\|x^k-\bar{x}^k\|^2\\
            \le & -\left(\frac{1}{\alpha}-\tau\Lambda\right)\|x^k-\bar{x}^k\|^2+\tau\ell\frac{1-\alpha}{\alpha}\big(\|y\|-B\big)_+\|x^k-\bar{x}^k\|^2 \\
            \le & \tau\ell\frac{1-\alpha}{\alpha}\big(\|y\|-B\big)_+\|x^k-\bar{x}^k\|^2.
        \end{aligned}
        $$
        The coefficients of the remaining three terms in Lemma \ref{lem4} satisfy
        $$
        \frac{1}{\beta}-3\tau M>0,\quad 1+\alpha-\frac{1}{\alpha}\ge 0,
        $$
        $$
        \frac{1}{\alpha}-\tau\big(\ell\|y^{k-1}\|+M+2L_h\big)\ge \frac{1}{\alpha}-\tau\big(\ell B+M+2L_h\big)>0.
        $$
        Hence all these negative terms can be discarded. This completes the proof of the lemma.
        
    \end{proof}
        
	\begin{thm}[Ergodic sublinear convergence rate]\label{thm2}
    Let $\{(x^k,y^k,\bar{x}^k)\}_{k\geq0}$ be the sequence generated by Algorithm~\ref{alg:PDAce}, and suppose that the conditions of Lemma \ref{lem5} hold. For any integer $N\geq1$, define the averaged iterates
        $$
        X_N:=\frac{1}{N}\sum_{k=1}^{N}x^k, \qquad Y_N:=\frac{1}{N}\sum_{k=1}^{N}y^k.
        $$
        Let $\mathcal{S}$ be a bounded comparison set as defined in Section \ref{section2_3}. We conclude that
        \begin{equation}\label{eq418}
            G_{\mathcal{S}}(X_N,Y_N)\le \frac{1}{2\tau N}\left(\sup_{(x,y)\in \mathcal{S}}\mathcal{E}_1(x,y)+\frac{\tau\ell(1-\alpha)}{\alpha(\tfrac1\alpha-\tau\Lambda)}\left(\sup_{(x,y)\in \mathcal{S}}\|y\|-B\right)_+\mathcal{E}_1(x^*,y^*)\right).
        \end{equation}
        In particular, if $\sup_{(x,y)\in \mathcal{S}}\|y\|\le B$, then
        \begin{equation}\label{eq419}
            G_{\mathcal{S}}(X_N,Y_N)\le \frac{1}{2\tau N}\sup_{(x,y)\in \mathcal{S}}\mathcal{E}_1(x,y).
        \end{equation}
    \end{thm}
    
	\begin{proof}
		Summing \eqref{general_descent} over $k=1,...,N$ gives
        $$
        \mathcal{E}_{N+1}(x,y)+2\tau\sum_{k=1}^N Q_k(x,y)
        \le \mathcal{E}_1(x,y)+\tau\ell\frac{1-\alpha}{\alpha}\big(\|y\|-B\big)_+\sum_{k=1}^N\|x^k-\bar{x}^k\|^2.
        $$
        Since $\mathcal{E}_{N+1}(x,y)\ge 0$, it follows that
        $$
        2\tau\sum_{k=1}^N Q_k(x,y)\le \mathcal{E}_1(x,y)+\tau\ell\frac{1-\alpha}{\alpha}\big(\|y\|-B\big)_+\sum_{k=1}^N\|x^k-\bar{x}^k\|^2.
        $$
        By inequality \eqref{x_k} in the proof of Lemma \ref{lem5},
        $$
        \sum_{k=1}^{\infty}\|x^k-\bar{x}^k\|^2\le \frac{1}{\frac{1}{\alpha}-\tau\Lambda}\mathcal{E}_1(x^*,y^*).
        $$
        Therefore,
        $$
        2\tau\sum_{k=1}^N Q_k(x,y)\le \mathcal{E}_1(x,y)+\frac{\tau\ell(1-\alpha)}{\alpha\left(\frac{1}{\alpha}-\tau\Lambda\right)}\big(\|y\|-B\big)_+\mathcal{E}_1(x^*,y^*).
        $$
        Moreover, $\mathcal{L}(\cdot,y)$ is convex and $\mathcal{L}(x,\cdot)$ is concave. Hence
        $$
        \mathcal{L}(X_N,y)\le \frac{1}{N}\sum_{k=1}^N\mathcal{L}(x^k,y),\quad
        \mathcal{L}(x,Y_N)\ge \frac{1}{N}\sum_{k=1}^N\mathcal{L}(x,y^k).
        $$
        Subtract the second inequality from the first one. This yields
        $$
        \mathcal{L}(X_N,y)-\mathcal{L}(x,Y_N)\le \frac{1}{N}\sum_{k=1}^N\big(\mathcal{L}(x^k,y)-\mathcal{L}(x,y^k)\big)=\frac{1}{N}\sum_{k=1}^N Q_k(x,y).
        $$
        Combined with the above inequality, we get
        $$
        \mathcal{L}(X_N,y)-\mathcal{L}(x,Y_N)
        \le \frac{1}{2\tau N}\left(\mathcal{E}_1(x,y)+\frac{\tau\ell(1-\alpha)}{\alpha\left(\frac{1}{\alpha}-\tau\Lambda\right)}\big(\|y\|-B\big)_+\mathcal{E}_1(x^*,y^*)\right).
        $$
        Taking the supremum over $(x,y)\in \mathcal{S}$ and using the definition of $G_{\mathcal{S}}(x,y)$, we obtain
        \begin{equation}\label{eq420}
            G_{\mathcal{S}}(X_N,Y_N) \le \frac{1}{2\tau N}\bigg(\sup_{(x,y)\in \mathcal{S}}\mathcal{E}_1(x,y)
            +\frac{\tau\ell(1-\alpha)}{\alpha(\tfrac1\alpha-\tau\Lambda)}\sup_{(x,y)\in \mathcal{S}}\big(\|y\|-B\big)_+
            \mathcal{E}_1(x^*,y^*)\bigg).
        \end{equation}
        Since $\mathcal{S}$ is bounded, we have
        $\sup_{(x,y)\in\mathcal{S}}\mathcal{E}_1(x,y)<\infty$ and $\sup_{(x,y)\in\mathcal{S}}\|y\|<\infty$.
        Thus, the right-hand side of~\eqref{eq420} is bounded by $C_{\mathcal{S}}/N$ for some finite constant $C_{\mathcal{S}}$ independent of $N$, which establishes~\eqref{eq418}. In particular, if $\sup_{(x,y)\in\mathcal{S}}\|y\|\le B$, the second term on the right-hand side of~\eqref{eq420} vanishes, yielding~\eqref{eq419}. Consequently, $G_{\mathcal{S}}(X_N,Y_N)=\mathcal{O}(1/N)$, which completes the proof.

    \end{proof}
    
    \subsection{Linear Convergence}

    We now establish the linear convergence of PDAce under additional strong convexity assumptions. In this subsection, we impose the following condition.
    
    \begin{assumption}\label{ass4}
    The function $f$ is $\mu_f$-strongly convex with $\mu_f>0$, and $H^*$ is $\mu_{H^*}$-strongly convex with $\mu_{H^*}>0$. That is, for any $u,x\in\operatorname{dom}(f)$ and $p\in\partial f(x)$,
    $$
    f(u) \geq f(x) + \langle p,u-x\rangle + \frac{\mu_f}{2}\|u-x\|^2,
    $$
    and for any $v,y\in\operatorname{dom}(H^*)$ and $q\in\partial H^*(y)$,
    $$
    H^*(v) \geq H^*(y) + \langle q,v-y\rangle + \frac{\mu_{H^*}}{2}\|v-y\|^2.
    $$
    \end{assumption}

    \begin{lemma}\label{lem7}
        Suppose that Assumptions~\ref{ass1}, \ref{ass2}, and~\ref{ass4} hold, and let $(x^*,y^*)\in\Omega$. Then, for every $k\geq1$,
        \begin{equation}\label{eq423}
        Q_k(x^*,y^*)\ge \frac{\mu_f}{2}\|x^k-x^*\|^2+\frac{\mu_{H^*}}{2}\|y^k-y^*\|^2.
        \end{equation}
        Moreover, the saddle point $(x^*,y^*)$ is unique.
    \end{lemma}

    \begin{proof}
        Since $(x^*,y^*)\in \Omega$, $x^*$ is the minimizer of the function $x\mapsto \mathcal{L}(x,y^*)$.
        Moreover, $f$ is $\mu_f$-strongly convex, so $x\mapsto \mathcal{L}(x,y^*)$ is also $\mu_f$-strongly convex.
        Since $x^*$ is its minimizer, the strong convexity yields
        $$
        \mathcal{L}(x,y^*)-\mathcal{L}(x^*,y^*)\ge \frac{\mu_f}{2}\|x-x^*\|^2,\quad \forall x\in\operatorname{dom}(f).
        $$
        In particular, taking $x=x^k$, we obtain
        \begin{equation}\label{eq421}
            \mathcal{L}(x^k,y^*)-\mathcal{L}(x^*,y^*)\ge \frac{\mu_f}{2}\|x^k-x^*\|^2.
        \end{equation}
        By the strong convexity of the dual component, we similarly obtain
        \begin{equation}\label{eq422}
            \mathcal{L}(x^*,y^*)-\mathcal{L}(x^*,y^k)\ge \frac{\mu_{H^*}}{2}\|y^k-y^*\|^2.
        \end{equation}
        Summing \eqref{eq421} and \eqref{eq422} yields
        $$
        \begin{aligned}
        Q_k(x^*,y^*)&=\mathcal{L}(x^k,y^*)-\mathcal{L}(x^*,y^k)\\
        &=\mathcal{L}(x^k,y^*)-\mathcal{L}(x^*,y^*)+\mathcal{L}(x^*,y^*)-\mathcal{L}(x^*,y^k)\\
        &\ge \frac{\mu_f}{2}\|x^k-x^*\|^2+\frac{\mu_{H^*}}{2}\|y^k-y^*\|^2.
        \end{aligned}
        $$
        We now establish the uniqueness of the saddle point.
        Suppose there exists another saddle point $(\hat{x},\hat{y})$. Since both $(x^*,y^*)$ and $(\hat{x},\hat{y})$ are saddle points, we have
        $$
        \mathcal{L}(x^*,\hat{y})\le \mathcal{L}(x^*,y^*)\le \mathcal{L}(\hat{x},y^*)\ \ \text{and} \ \ 
        \mathcal{L}(\hat{x},y^*)\le \mathcal{L}(\hat{x},\hat{y})\le \mathcal{L}(x^*,\hat{y}).
        $$
        From these two inequalities, all these quantities must be equal. Thus,
        $$
        \mathcal{L}(\hat{x},y^*)=\mathcal{L}(x^*,y^*).
        $$
        Since $x^*$ is the unique minimizer of the strongly convex function $\mathcal{L}(\cdot,y^*)$ and $y^*$ is the unique maximizer of the strongly concave function $\mathcal{L}(x^*,\cdot)$, we obtain
        $$
        \hat{x}=x^*,\ \ \hat{y}=y^*.
        $$
        Therefore, the saddle point is unique.

    \end{proof}
	
    \begin{thm}[Linear convergence of PDAce]\label{thm3}
        Under Assumptions \ref{ass1}, \ref{ass2}, \ref{ass3}, \ref{ass4}, and the stepsize conditions of PDAce in Lemma \ref{lem5}, define
        \begin{equation}\label{eq424}
            \gamma:=\min\left\{\frac{(1-\alpha)\tau\mu_f}{3},\;\frac{1-\alpha}{3\alpha^2}\left(\frac1\alpha-\tau\Lambda\right),\;\beta\tau\mu_{H^*},\;\frac{\frac1\alpha-\tau(\ell B+M+2L_h)}{3(1-\alpha)+\tau L_h}\right\}.
        \end{equation}
        Then \(\gamma>0\), and for all \(k\ge 1\),
        \begin{equation}\label{eq425}
            \mathcal{E}_k(x^*,y^*)\le \frac{1}{(1+\gamma)^{k-1}}\mathcal{E}_1(x^*,y^*).
        \end{equation}
        Furthermore, the iterative sequence $\{(x^k,\bar{x}^k,y^k)\}$ generated by Algorithm \ref{alg:PDAce} converges linearly. More precisely, the following bounds hold:
        \begin{align*}
            \|x^k - x^*\|  \le \frac{\sqrt{D_1}}{(1+\gamma)^{\frac{k-1}{2}}},   \qquad
            \|y^k - y^*\|  \le \frac{\sqrt{D_2}}{(1+\gamma)^{\frac{k-1}{2}}}, \qquad 
            \|\bar{x}^k - x^*\|  \le \frac{\sqrt{D_3}}{(1+\gamma)^{\frac{k-1}{2}}},
        \end{align*} 
        where \(D_1 = \frac{\mathcal{E}_1(x^*,y^*)}{\tau\mu_f}\), \(D_2 = \frac{\mathcal{E}_1(x^*,y^*)}{\tau\mu_{H^*}}\), and \(D_3 = (1-\alpha)\mathcal{E}_1(x^*,y^*)\). 
    \end{thm}

    \begin{proof}
		From Lemma \ref{lem5}, we obtain
        \begin{equation}\label{eq427}
            \begin{aligned}
                &\mathcal{E}_{k+1}(x^*,y^*) + 2\tau Q_k(x^*,y^*)\\
                \le & \mathcal{E}_k(x^*,y^*) -\left(\frac1\alpha-\tau\Lambda\right)\|x^k-\bar{x}^k\|^2 -\left(\frac1\beta-3\tau M\right)\|y^k-y^{k-1}\|^2 \\
                & -\left(1+\alpha-\frac1\alpha\right)\|x^{k+1}-\bar{x}^k\|^2 -\left(\frac1\alpha-\tau(\ell B+M+2L_h)\right)\|x^{k+1}-x^k\|^2.
            \end{aligned}
        \end{equation}
        By Lemma \ref{lem7},
        \begin{equation}\label{eq428}
        Q_k(x^*,y^*)\ge \frac{\mu_f}{2}\|x^k-x^*\|^2+\frac{\mu_{H^*}}{2}\|y^k-y^*\|^2.
        \end{equation}
        Substitute \eqref{eq428} into \eqref{eq427}. Since $1+\alpha-\frac1\alpha$ and $\frac1\beta-3\tau M$ are nonnegative, substituting \eqref{eq428} into \eqref{eq427} yields
        \begin{equation}\label{eq429}
            \begin{aligned}
                &\quad \mathcal{E}_{k+1}(x^*,y^*) + \tau\mu_f\|x^k-x^*\|^2+\tau\mu_{H^*}\|y^k-y^*\|^2 \\
                &+\left( \frac1\alpha-\tau\Lambda \right)\|x^{k}-\bar{x}^k\|^2 \\
                &+\left(\frac1\alpha-\tau(\ell B+M+2L_h)\right)\|x^{k+1}-x^k\|^2 
                \le  \mathcal{E}_k(x^*,y^*).
            \end{aligned}
        \end{equation}
        By the definition of $\mathcal{E}_{k}(x^*,y^*)$, we have
        \begin{equation}\label{eq430}
            \mathcal{E}_{k+1}(x^*,y^*)=\frac{1}{1-\alpha}\|\bar{x}^{k+1}-x^*\|^2+\frac{1}{\beta}\|y^k-y^*\|^2+\tau L_h\|x^{k+1}-x^k\|^2.
        \end{equation}
        Moreover, since $\bar{x}^{k+1}=\alpha\bar{x}^k+(1-\alpha)x^{k+1}$, it follows that $\bar{x}^{k+1}-x^* = x^k - x^* + \alpha(\bar{x}^k - x^k)+(1-\alpha)(x^{k+1}-x^k)$. From Lemma \ref{lem2}, we obtain
        \begin{equation*}
            \frac{1}{1-\alpha}\|\bar{x}^{k+1}-x^*\|^2 \le \frac{3}{1-\alpha}\|x^k-x^*\|^2 + \frac{3\alpha^2}{1-\alpha}\|x^k-\bar{x}^k\|^2 + 3(1-\alpha)\|x^{k+1}-x^k\|^2.
        \end{equation*}
        Substituting into \eqref{eq430}, we obtain
        \begin{equation}\label{eq433}
            \begin{aligned}
                \mathcal{E}_{k+1}(x^*,y^*) &\le \frac{3}{1-\alpha}\|x^k-x^*\|^2 + \frac{3\alpha^2}{1-\alpha}\|x^k-\bar{x}^k\|^2 + \frac1\beta\|y^k-y^*\|^2 \\
                &\quad + \big(3(1-\alpha)+\tau L_h\big)\|x^{k+1}-x^k\|^2.
            \end{aligned}
        \end{equation}
        By the definition of $\gamma$, we have
        \begin{equation}\label{eq434}
            \gamma \frac{3}{1-\alpha}\|x^k-x^*\|^2 \le \tau\mu_f\|x^k-x^*\|^2,
        \end{equation}
        \begin{equation}\label{eq435}
            \gamma \frac{3\alpha^2}{1-\alpha}\|x^k-\bar{x}^k\|^2 \le \left(\frac1\alpha-\tau\Lambda\right)\|x^k-\bar{x}^k\|^2,
        \end{equation}
        \begin{equation}\label{eq436}
            \gamma \frac1\beta\|y^k-y^*\|^2 \le \tau\mu_{H^*}\|y^k-y^*\|^2,
        \end{equation}
        \begin{equation}\label{eq437}
            \gamma \big(3(1-\alpha)+\tau L_h\big)\|x^{k+1}-x^k\|^2 \le \left(\frac1\alpha-\tau(\ell B+M+2L_h)\right)\|x^{k+1}-x^k\|^2.
        \end{equation}
        Combining \eqref{eq433}--\eqref{eq437}, we obtain
        \begin{equation}\label{eq438}
            \begin{aligned}
                \gamma \mathcal{E}_{k+1}(x^*,y^*)
                &\le \tau\mu_f\|x^k-x^*\|^2+\tau\mu_{H^*}\|y^k-y^*\|^2 \\
                &\quad + \left(\frac1\alpha-\tau\Lambda\right)\|x^{k}-\bar{x}^k\|^2 \\
                &\quad + \left(\frac1\alpha-\tau(\ell B+M+2L_h)\right)\|x^{k+1}-x^k\|^2.
            \end{aligned}
        \end{equation}
        Using \eqref{eq429}, we then obtain
        \begin{equation*}
            (1+\gamma)\mathcal{E}_{k+1}(x^*,y^*) \le \mathcal{E}_k(x^*,y^*).
        \end{equation*}
        Since $\gamma>0$, we have
        \begin{equation*}
            \mathcal{E}_{k+1}(x^*,y^*)\le \frac{1}{1+\gamma}\mathcal{E}_k(x^*,y^*).
        \end{equation*}
        By recursion, we obtain
        \begin{equation}\label{eq440}
            \begin{aligned}
                \mathcal{E}_{k}(x^*,y^*)&\le \frac{1}{1+\gamma}\mathcal{E}_{k-1}(x^*,y^*)\\
                &\le \cdots \le (1+\gamma)^{-(k-1)}\mathcal{E}_1(x^*,y^*).
            \end{aligned}
        \end{equation}
        From~\eqref{eq429}, we obtain
        \[
        \tau\mu_f\|x^k - x^*\|^2 \le \mathcal{E}_k(x^*,y^*), \quad
        \tau\mu_{H^*}\|y^k - y^*\|^2 \le \mathcal{E}_k(x^*,y^*).
        \]
        Moreover, by the definition of $\mathcal{E}_k(x^*,y^*)$, we have
        \[
        \|\bar{x}^k - x^*\|^2 \le (1-\alpha)\mathcal{E}_k(x^*,y^*).
        \]
        Combining these estimates with~\eqref{eq440} yields the linear convergence bounds for $x^k$, $y^k$, and $\bar{x}^k$ stated in Theorem \ref{thm3}. This completes the proof.
        
    \end{proof}

	\section{Accelerated convergence rate}\label{section4}

    In this section, we develop an accelerated variant of PDAce, termed aPDAce, under the additional assumption that $f$ is $\mu_f$-strongly convex with $\mu_f>0$. We establish an $\mathcal{O}(1/N^2)$ ergodic convergence rate by combining an adaptive stepsize sequence with a weighted Lyapunov analysis.
    
    We first construct candidate bounded regions and specify   sufficient choices   of the acceleration parameters under the local regularity assumptions. Fix $\alpha\in(\frac{\sqrt{5}-1}{2},1)$, $1<\varphi<\min\{\frac{1}{\alpha},\alpha(1+\alpha)\}$, and let  $\widehat{\tau}>0$  and $\beta_0>0$ be given constants. Let $(x^*,y^*)\in\Omega$ be an arbitrary saddle point. Define $d_0:=\|x^0-x^*\|$ and $c_0:=\|\nabla h(x^0)-\nabla h(x^*)+J_g(x^0)^\top y^0-J_g(x^*)^\top y^*\|$. Set      \begin{equation}\label{apdace_initial_bound}
        \begin{split}
            \overline{E} := \frac{\beta_0(1 + \alpha\mu_f \varphi\widehat{\tau})} {\alpha\varphi(1-\alpha)}
            (d_0 + (1-\alpha)\widehat{\tau}c_0)^2 + \|y^0-y^*\|^2
            + \frac{\beta_0L_h\widehat{\tau}}{\alpha} (2d_0 + \widehat{\tau}c_0)^2,
        \end{split}
    \end{equation}
    and   $R_0 := \max\left\{d_0,\sqrt{\frac{(1-\alpha)\overline{E}}{\beta_0}}\right\}$.
    Choose positive radii $R_X$ and $R_Y$ satisfying
    \begin{equation}\label{apdace_region}
        R_X\geq\frac{1+\alpha}{1-\alpha}R_0,
        \qquad
        R_Y\geq\sqrt{\overline{E}},
    \end{equation}
    and define
    \[
    \mathcal{X}:=\overline{B}(x^*,R_X), \qquad
    \mathcal{Y}:=\operatorname{dom}(H^*)\cap\overline{B}(y^*,R_Y).
    \]
    Let $\ell:=\ell_{\mathcal{X}}$ and choose a finite constant $M>0$ satisfying $M\geq M_{\mathcal{X}}:=\sup_{z\in\mathcal{X}}\|J_g(z)\|$, and set $B:=\|y^*\|+R_Y$ and $L:=\ell B$.     Define
    \[
    C_1 :=L\left(\alpha+\frac{(1-\alpha)^2}{\alpha^2} +\frac{1-\alpha}{\alpha} \right),
    \qquad
    C_2 :=\left(L+L_h\frac{1+\alpha}{\alpha}-\mu_f \right)_+.
    \]
    We use the convention that \(1/0=+\infty\). Choose
    \(\tau_{\max}\) such that
    \[
    0<\tau_{\max} < \min\left\{\widehat{\tau}, \frac{1}{2C_1}, \frac{1}{2C_2} \right\},
    \]
    and subsequently choose \(\xi\) satisfying
    \begin{equation}\label{apdace_parameters}
    0<\xi < \min\left\{\frac{1}{10\alpha M}, \beta_0M\tau_{\max}^2 \right\}.
    \end{equation}
    These choices imply
    \begin{equation}\label{apdace_stability}
        d:=1-5\alpha\xi M-\tau_{\max}C_1>0,
        \quad
        e:=1-5\alpha\xi M-\tau_{\max}C_2>0.
    \end{equation}

    In this section, we adopt the same notation for quantities with the same meaning as in Section~\ref{section3}. Recall that $M>0$ is an upper bound for the Jacobian norm $\|J_g(z)\|$ on the candidate bounded primal region $\mathcal{X}$, as specified above. This parameter enters both the initial stepsize and the adaptive stepsize update. With the candidate bounded regions and admissible parameters specified above, we present the accelerated algorithm aPDAce in Algorithm 2.
    
	\begin{algorithm}[H]
    \caption{Accelerated PDA with Convex Combination and Extrapolation (aPDAce)}
    \label{alg:aPDAce}
        \begin{algorithmic}[1]
            \State \textbf{Initialization:} Choose parameters $\alpha\in(\frac{\sqrt{5}-1}{2},1)$, $1<\varphi<\min\{\frac{1}{\alpha},\alpha(1+\alpha)\}$, $\beta_0>0$, and positive constants $M$, $\tau_{\max}$, and $\xi$ satisfying \eqref{apdace_parameters}. Set $\tau_0=\sqrt{\frac{\xi}{\beta_0M}}$. Choose initial points $x^0\in\operatorname{dom}(F)$ and $y^0\in\operatorname{dom}(H^*)$. Set $\bar{x}^0=x^0$ and $k=0$.
            \State \textbf{Main Iteration:}
            \For{$k=0,1,2,\dots$}
                \State Step 1. Compute
                \begin{align}
                    x^{k+1}
                    &= \operatorname{Prox}_{\tau_k f}\left(\bar{x}^k-\tau_k\big(\nabla h(x^k)+J_g(\bar{x}^k)^\top y^k\big)\right).
                    \label{eq40}
                \end{align}
                \State Step 2. Compute
                \begin{align}
                    \bar{x}^{k+1}
                    &= \alpha\bar{x}^k+(1-\alpha)x^{k+1},
                    \label{eq41}\\
                    \rho_{k+1}
                    &= \frac{1-\alpha\varphi}{1+\mu_f\alpha\varphi\tau_k},
                    \label{eq42}\\
                    \beta_{k+1}
                    &= \beta_k\left(1+\rho_{k+1}\mu_f\tau_k\right),
                    \label{eq43}\\
                    \tau_{k+1}
                    &= \min\left\{\varphi\tau_k,\frac{\xi}{\tau_k\beta_{k+1}M},\tau_{\max}\right\},
                    \label{eq44}\\
                    v^{k+1}
                    &= g(\bar{x}^{k+1})-\alpha\big(g(\bar{x}^k)-g(x^{k+1})\big),
                    \label{eq45}\\
                    y^{k+1}
                    &= \operatorname{Prox}_{\beta_{k+1}\tau_{k+1}H^*}\left(y^k+\beta_{k+1}\tau_{k+1}v^{k+1}\right).
                    \label{eq46}
                \end{align}
            \EndFor
        \end{algorithmic}
    \end{algorithm}

    For the convergence analysis, we introduce the following weighted Lyapunov potential for any $(x,y)\in\operatorname{dom}(F)\times\operatorname{dom}(H^*)$ and $k\geq1$:
    \begin{equation}\label{apdace_energy}
        E_k(x,y):=\frac{\beta_k(1+\alpha\mu_f\tau_k)}{1-\alpha}\|\bar{x}^k-x\|^2+\|y^{k-1}-y\|^2+L_h\beta_k\tau_k\|x^k-x^{k-1}\|^2.
    \end{equation}

    \begin{remark}\label{rm6}
        The parameter $\varphi$ \textcolor{black}{in Algorithm \ref{alg:aPDAce}} controls the growth of the primal stepsizes, whereas $\xi$ regulates the coupling between $\tau_k$ and $\beta_k$. The upper bound $\tau_{\max}$ is introduced to maintain the positivity of the coefficients in the weighted Lyapunov estimate. The parameter selection procedure described above establishes the existence of positive constants such that the descent property of the potential function \eqref{apdace_descent} holds (see Lemma \ref{lem10}), with no circular dependence between the candidate regions and the final stepsize parameters.
        
        The bounded regions and the associated local constants are introduced for the theoretical convergence analysis. In particular, their construction uses a reference saddle point and is not intended as a directly computable parameter selection rule. In practical implementations, the parameter $M$ should be supplied as a valid upper bound for the Jacobian norm on a region containing the generated primal iterates. The invariance of the candidate regions will be established in Lemma~\ref{lem10}.
    \end{remark}
    
	\begin{remark}\label{rm4}
        For Algorithm \ref{alg:aPDAce}, we have the following observations.
        \begin{enumerate}
            \item In this algorithm, $\rho_{k+1}$ is used to update $\beta_{k+1}$ and underpins our subsequent convergence analysis. Meanwhile, we balance the growth ratio $\tau_{k+1}/\tau_k$ and the magnitude of $\tau_{k+1}$ in \eqref{eq44}.
            \item \textcolor{black}{If \(f\) is only convex, i.e., the strong convexity parameter satisfies \(\mu_f=0\), then \(\beta_{k+1}\equiv\beta_0\) for all \(k\ge0\). By induction, it is easy to verify that the minimum in \eqref{eq44} is always attained at the second term on the right-hand side, which implies that \(\tau_{k+1}\equiv\tau_0\) for all \(k\ge0\). Consequently, aPDAce (Algorithm \ref{alg:aPDAce}) reduces to PDAce (Algorithm \ref{alg:PDAce}) with fixed stepsizes \(\tau=\tau_0\) and \(\sigma=\sigma_0\), where \(\sigma_0:=\beta_0\tau_0\).}
        \end{enumerate}
    \end{remark}
    
    We now derive several useful properties of the sequences $\{\tau_{k+1}\}$ and $\{\beta_{k+1}\}$ generated by Algorithm \ref{alg:aPDAce}. For notational simplicity, we set $\delta_{k+1} = \frac{\tau_{k+1}}{\tau_k}$ for the rest of this section.
    
	\begin{lemma}\label{lem8}
        (i) For the parameter sequences generated by Algorithm \ref{alg:aPDAce} (aPDAce), the following properties hold:
        \begin{equation}\label{m25}
            \beta_{k+1}>\beta_k>0,\quad \alpha<\frac{\tau_k}{\tau_{k-1}}\le \varphi, \quad \forall k\ge 1.
        \end{equation}
        (ii) There exist constants $c_1,c_2>0$ independent of $k$ such that
        \begin{equation}\label{m28}
        \sqrt{\beta_k}\,\tau_k\ge c_1,\quad \beta_k\ge c_2^2 k^2.
        \end{equation}
        Furthermore,
        \begin{equation}\label{m29}
            \sum_{k=1}^N \beta_k\tau_k \ge \frac{c_1c_2}{2}N(N+1).
        \end{equation}
    \end{lemma}
        	
    \begin{proof}
        (i) From the update formulas of $\rho_{k+1}$ and $\beta_{k+1}$, we have
        \[
        \frac{\beta_{k+1}}{\beta_k} = 1+\frac{(1-\alpha\varphi)\mu_f\tau_k}{1+\mu_f\alpha\varphi\tau_k}
        =\frac{1+\mu_f\alpha\varphi\tau_k+(1-\alpha\varphi)\mu_f\tau_k}{1+\mu_f\alpha\varphi\tau_k}
        =\frac{1+\mu_f\tau_k}{1+\mu_f\alpha\varphi\tau_k}.
        \]
        Since $\alpha\varphi<1$, we have $1+\mu_f\tau_k>1+\mu_f\alpha\varphi\tau_k$, which yields
        \[
        \beta_{k+1}>\beta_k>0.
        \]
        We aim to verify $\frac{1+\mu_f\tau_k}{1+\mu_f\alpha\varphi\tau_k}<\frac{1}{\alpha\varphi}$. Multiplying both sides of the inequality by the positive quantity $\alpha\varphi(1+\mu_f\alpha\varphi\tau_k)$, it suffices to show
        \[
        \alpha\varphi(1+\mu_f\tau_k)<1+\mu_f\alpha\varphi\tau_k \implies \alpha\varphi<1.
        \]
        Hence,
        \begin{equation}\label{m32}
           1<\frac{\beta_{k+1}}{\beta_k}<\frac{1}{\alpha\varphi}.
        \end{equation}
        
        When $k=0$, $\beta_0\tau_0^2=\frac{\xi}{M}<\frac{\varphi\xi}{M}$.
        
        For $k\ge1$, from the update rule of $\tau_{k+1}$, we have $\tau_k\le\varphi\tau_{k-1}$ and $\tau_k\le\frac{\xi}{M\beta_k\tau_{k-1}}$.
        Multiplying these two inequalities together gives
        \begin{equation}\label{m33}
            \beta_k\tau_k^2\le\frac{\varphi \xi}{M}.
        \end{equation}
        From \eqref{m32} and \eqref{m33}, we obtain
        \begin{equation}\label{m34}
            \beta_k\tau_{k-1}^2 = \frac{\beta_k}{\beta_{k-1}} \cdot \beta_{k-1}\cdot\tau_{k-1}^2
            < \frac{1}{\alpha\varphi}\cdot\frac{\varphi \xi}{M} = \frac{\xi}{\alpha M}.
        \end{equation}
        Therefore,
        \[
        \frac{\xi}{M\beta_k\tau_{k-1}^2}>\alpha.
        \]
        Moreover, since $\tau_{k-1}\le\tau_{\max}$, we have
        \begin{equation}\label{m36}
            \frac{\tau_{\max}}{\tau_{k-1}}\ge 1>\alpha,\quad \varphi>1>\alpha.
        \end{equation}
        
        Replacing $k+1$ by $k$ in \eqref{eq44} and dividing by $\tau_{k-1}$ give
        \[
        \frac{\tau_k}{\tau_{k-1}} = \min\left\{ \varphi,\ \frac{\xi}{M\beta_k\tau_{k-1}^2},\ \frac{\tau_{\max}}{\tau_{k-1}} \right\}.
        \]
        All three candidates are strictly greater than $\alpha$. Hence, $\frac{\tau_k}{\tau_{k-1}} > \alpha$, which implies 
        \[
        \alpha < \frac{\tau_k}{\tau_{k-1}} \le \varphi.
        \]
        
        (ii) When $k=0$, $\sqrt{\beta_0}\tau_0=\sqrt{\frac{\xi}{M}}>\sqrt{\frac{\alpha \xi}{M}}$. We proceed by mathematical induction. Suppose that $\sqrt{\beta_k}\tau_k\ge\sqrt{\frac{\alpha \xi}{M}}$ for $k\ge1$. We consider three cases for $\tau_{k+1}$.
        
        \begin{enumerate}[label=]
            \item \textbf{Case 1.} If $\tau_{k+1}=\varphi\tau_k$, then since $\beta_{k+1}>\beta_k$, we have
            \begin{equation}\label{m37}
                \sqrt{\beta_{k+1}}\tau_{k+1}=\varphi\sqrt{\beta_{k+1}}\tau_k \ge \varphi\sqrt{\beta_k}\tau_k > \sqrt{\frac{\alpha \xi}{M}}.
            \end{equation}
        
            \item \textbf{Case 2.} If $\tau_{k+1}=\frac{\xi}{M\beta_{k+1}\tau_k}$, then from \eqref{m32} and \eqref{m33},
            \[
            \beta_{k+1}\tau_k^2=\frac{\beta_{k+1}}{\beta_k}\cdot \beta_k\tau_k^2 < \frac{1}{\alpha\varphi}\cdot\frac{\varphi \xi}{M}=\frac{\xi}{\alpha M}.
            \]
            Thus $\sqrt{\beta_{k+1}}\tau_k<\sqrt{\frac{\xi}{\alpha M}}$. It follows that
            \begin{equation}\label{m38}
                \sqrt{\beta_{k+1}}\tau_{k+1}=\frac{\xi}{M\sqrt{\beta_{k+1}}\tau_k}>\frac{\xi/M}{\sqrt{\xi/(\alpha M)}}=\sqrt{\frac{\alpha \xi}{M}}.
            \end{equation}
            
            \item \textbf{Case 3.} If $\tau_{k+1}=\tau_{\max}$, then $\beta_{k+1}\ge\beta_0$ and $\tau_{\max}>\tau_0$. Hence
            \begin{equation}\label{m39}
                \sqrt{\beta_{k+1}}\tau_{k+1} = \sqrt{\beta_{k+1}} \tau_{\max}\ge\sqrt{\beta_0} \tau_0 = \sqrt{\frac{\xi}{M}}>\sqrt{\frac{\alpha \xi}{M}}.
            \end{equation}
        \end{enumerate}
        From \eqref{m37}, \eqref{m38}, and \eqref{m39}, we obtain $\sqrt{\beta_{k+1}}\tau_{k+1}>\sqrt{\frac{\alpha \xi}{M}}$. Therefore the statement holds for $k+1$.
        By mathematical induction, for all $k\ge0$,
        \begin{equation}\label{m40}
            \sqrt{\beta_k}\tau_k \ge c_1:=\sqrt{\frac{\alpha \xi }{M}}.
        \end{equation}
        
        By the update rule of $\beta_{k+1}$, we have
        \begin{equation}\label{m41}
            \begin{aligned}
                \beta_{k+1}-\beta_k &= \frac{(1-\alpha\varphi)\mu_f\tau_k}{1+\mu_f\alpha\varphi\tau_k}\beta_k
                \ge \frac{(1-\alpha\varphi)\mu_f}{1+\mu_f\alpha\varphi\tau_{\max}}\beta_k\tau_k \\
                &= \frac{(1-\alpha\varphi)\mu_f}{1+\mu_f\alpha\varphi\tau_{\max}}\sqrt{\beta_k}\big(\sqrt{\beta_k}\tau_k\big)
                \ge \frac{(1-\alpha\varphi)\mu_f c_1}{1+\mu_f\alpha\varphi\tau_{\max}}\sqrt{\beta_k}.
            \end{aligned}
        \end{equation}
        Moreover, from \eqref{m32}, we have $\sqrt{\beta_{k+1}}<\frac{1}{\sqrt{\alpha\varphi}}\sqrt{\beta_k}$.
        
        \begin{equation}\label{m42}
            \sqrt{\beta_{k+1}}-\sqrt{\beta_k}
            =\frac{\beta_{k+1}-\beta_k}{\sqrt{\beta_{k+1}}+\sqrt{\beta_k}}
            \ge \frac{(1-\alpha\varphi)\mu_f c_1}{\big(1+\mu_f\alpha\varphi\tau_{\max}\big)\big(1+\frac{1}{\sqrt{\alpha\varphi}}\big)}.
        \end{equation}
        
        Set $c_2:=\frac{(1-\alpha\varphi)\mu_f c_1}{\big(1+\mu_f\alpha\varphi\tau_{\max}\big)\big(1+\frac{1}{\sqrt{\alpha\varphi}}\big)}$, 
        then $\sqrt{\beta_{k+1}} \ge \sqrt{\beta_k}+c_2$.
        Summing recursively, we obtain
        \[
        \sqrt{\beta_k}\ge \sqrt{\beta_0}+c_2 k \ge c_2 k \implies \beta_k\ge c_2^2 k^2.
        \]
        Moreover,
        \[
        \beta_k\tau_k=\sqrt{\beta_k}\big(\sqrt{\beta_k}\tau_k\big)\ge c_1c_2 k,
        \]
        which yields
        \begin{equation}\label{m43}
            \sum_{k=1}^N \beta_k\tau_k \ge c_1c_2\sum_{k=1}^N k=\frac{c_1c_2}{2}N(N+1).
        \end{equation}
        This completes the proof of the lemma.
	\end{proof}

    \begin{lemma}\label{lem9}
        Assume that $\bar{x}^{k-1}, \bar{x}^k, x^k\in \mathcal{X}$ and $y^{k-1}\in \mathcal{Y}$. Then for any fixed $(x,y)\in \operatorname{dom}f\times \operatorname{dom}H^*$, the following estimate holds: 
        \[
        \begin{aligned}
            &E_{k+1}(x,y)+2\beta_k\tau_k Q_k(x,y)\\
            \le & E_k(x,y) - \left(1-\frac{M\beta_k\tau_k\tau_{k-1}}{\xi}\right)\|y^k-y^{k-1}\|^2  - \beta_k\left(1+\alpha+\alpha\mu_f\tau_k-\frac{\delta_k}{\alpha}\right)\|x^{k+1}-\bar{x}^k\|^2 \\
            & - \beta_k\left(\delta_k\left(\frac1\alpha-5\xi M\right)-\tau_k L\left(\alpha+\frac{(1-\alpha)^2}{\alpha^2}\right)-\tau_k\ell\|y\|\frac{1-\alpha}{\alpha}\right)\|x^k-\bar{x}^k\|^2 \\
            & - \beta_k\left(\delta_k\left(\frac1\alpha-5\xi M\right)+\mu_f\tau_k-\tau_k\left(L+L_h\right)-L_h\tau_{k+1}\frac{\beta_{k+1}}{\beta_k}\right)\|x^{k+1}-x^k\|^2.
        \end{aligned}
        \]
    \end{lemma}
    The detailed proof of Lemma \ref{lem9} is deferred to Appendix \ref{app_lem9}.
    
    \begin{lemma}\label{lem10}
        Suppose that Assumptions~\ref{ass1} and~\ref{ass2} hold  and that $f$ is $\mu_f$-strongly convex with $\mu_f>0$. Let $(x^*,y^*)\in\Omega$ and $E_1^*:=E_1(x^*,y^*)$. Then the sequence generated by Algorithm~\ref{alg:aPDAce} satisfies $x^k,\bar{x}^k\in\mathcal{X}$ and $y^k\in\mathcal{Y}$ for all $k\geq0$. Moreover,
        \begin{equation}\label{apdace_descent}
            E_{k+1}(x^*,y^*)\leq E_k(x^*,y^*)\leq E_1^*\leq\overline E, \qquad k\geq1.
        \end{equation}
        Furthermore,
        \begin{equation}\label{apdace_summability}
            \sum_{k=1}^{\infty}\beta_k\|x^k-\bar{x}^k\|^2\leq\frac{E_1^*}{d}<+\infty.
        \end{equation}
    \end{lemma}
    The detailed proof of Lemma \ref{lem10} is deferred to Appendix \ref{app_lem10}.

    \begin{lemma}\label{lem11}
        Under the assumptions of Lemma \ref{lem10}, for every $k\ge1$ and any fixed $(x,y)\in \operatorname{dom}f\times \operatorname{dom}H^*$, there holds
        \[
        E_{k+1}(x,y)+2\beta_k\tau_k Q_k(x,y)
        \le E_k(x,y)+\beta_k\tau_k\ell\frac{1-\alpha}{\alpha}\big(\|y\|-B\big)_+\|x^k-\bar{x}^k\|^2.
        \]
    \end{lemma}
    The detailed proof of Lemma \ref{lem11} is deferred to Appendix \ref{app_lem11}.
    
    \begin{thm}\label{thm4}
        Suppose that the assumptions of Lemma \ref{lem10} hold, and let the sequence $\{(x^k,\bar{x}^k,y^k)\}$ be generated by Algorithm \ref{alg:aPDAce}. For any $N\geq 1$, define $W_N=\sum_{k=1}^N \beta_k\tau_k$, and the ergodic sequences
        \[
        \bar{X}_N=\frac{1}{W_N}\sum_{k=1}^N \beta_k\tau_k x^k,\quad
        \bar{Y}_N=\frac{1}{W_N}\sum_{k=1}^N \beta_k\tau_k y^k.
        \]
        Let $\mathcal{S}$ be a bounded comparison set as defined in Section \ref{section2_3}. Then the following bound holds:
        \begin{equation}\label{m105}
            G_{\mathcal{S}}(\bar{X}_N,\bar{Y}_N)\le \frac{1}{2W_N}\left( \sup_{(x,y)\in \mathcal{S}}E_1(x,y)+\frac{\tau_{\max}\ell(1-\alpha)}{\alpha d}\big(\sup_{(x,y)\in \mathcal{S}}\|y\|-B\big)_+E_1^* \right).
        \end{equation}
        In particular, if $\sup_{(x,y)\in \mathcal{S}}\|y\|\le B$, then
        \begin{equation}\label{m106}
            G_{\mathcal{S}}(\bar{X}_N,\bar{Y}_N)\le \frac{1}{2W_N}\sup_{(x,y)\in \mathcal{S}}E_1(x,y)
            \le \frac{1}{c_1c_2N(N+1)}\sup_{(x,y)\in \mathcal{S}}E_1(x,y).
        \end{equation}
        Moreover, by Lemma \ref{lem10}, we obtain 
        \begin{equation}
            \big\|\bar{x}^{N+1}-x^*\big\|^2 \le \frac{1}{c_2^2(N+1)^2}(1-\alpha)E_1^*.
        \end{equation}
    \end{thm}
    The detailed proof of Theorem \ref{thm4} is deferred to Appendix \ref{app_thm4}.
	
	\section{Numerical Experiments}\label{section5}
    In this section, we first examine the numerical improvement induced by two key innovations in the proposed PDAce (Algorithm \ref{alg:PDAce}), and then evaluate the numerical performance of PDAce and its accelerated variant aPDAce (Algorithm \ref{alg:aPDAce}) on three representative classes of convex–concave saddle point problems: quadratically constrained quadratic programming problems, convex-concave minimax game problems, and convex-concave minimax problems with an exponential coupling term.
    All experiments are conducted in \textsc{Matlab} R2023b on a 64-bit Windows 10 PC equipped with an Intel(R) Core(TM) i5-10500 CPU @ 3.10 GHz and 8 GB of RAM. 

    Unless otherwise specified, the stepsizes and other algorithmic parameters of GRPDA \cite{w10}, PDAce and aPDAce in the numerical experiments are chosen empirically according to the characteristics of each test problem. The theoretical stepsize conditions established in the convergence analysis are sufficient conditions and are not imposed on the numerical implementations. For competing methods such as APD \cite{w11} and PDAc-L \cite{w5}, we conduct the experiments following the parameter settings reported in their corresponding references.

    \subsection{Impact of Key Components: Linearization at \texorpdfstring{$\bar{x}^{k}$}{\bar{x}(k)} and Update of \texorpdfstring{$v^{k+1}$}{v(k+1)}} \label{sec5.1}
	We examine the numerical improvement induced by two key innovations in PDAce. All numerical tests in this subsection and the following comparison experiments are implemented on convex quadratically constrained quadratic programming (QCQP) problems with the same experimental settings. The QCQP test problem \cite{w11} is formulated as
    \begin{equation}\label{eq:qcqp}
    	\begin{cases}
    		h_{\text{opt}} := \displaystyle\min_{x\in \mathcal{X}} h(x) := \frac{1}{2}x^\top A_0 x + b_0^\top x, \\[1.5ex]
    		\quad \text{s.t.} \quad g_j(x) := \frac{1}{2}x^\top A_j x + b_j^\top x - c_j \le 0, \quad j \in \{1, \dots, m\},
    	\end{cases}
    \end{equation}
    with $\mathcal{X} := [-10, 10]^n$. The synthetic data are generated according to the following procedure: the vectors $\{b_j\}_{j=0}^m \subset \mathbb{R}^n$ are drawn independently from the standard normal distribution, the scalars $\{c_j\}_{j=1}^m \subset \mathbb{R}$ are sampled uniformly from $[0,1]$, and the matrices are set to $A_j = \Lambda_j^\top S_j \Lambda_j$ for $j=0,1,\dots,m$. In this construction, each $\Lambda_j \in \mathbb{R}^{n\times n}$ is a random orthogonal matrix, while each $S_j \in \mathbb{R}_+^{n\times n}$ is diagonal with entries drawn uniformly from $[0,100]$, where zero diagonal entries are permitted. Defining $g(x) := (g_1(x), \dots, g_m(x))^\top$ and $\Phi(x, y) := \langle y, g(x) \rangle$, problem \eqref{eq:qcqp} admits the primal-dual reformulation $\min_x \max_y f(x) + h(x) + \Phi(x, y) - H^*(y)$, with $f(x) = \iota_{\mathcal{X}}(x)$ being the indicator of $\mathcal{X}$ and $H^*(y) = \iota_+(y)$ being the indicator of the nonnegative orthant. Observe that $\langle y, g(x) \rangle$ is affine in $y$. Denoting by $J_g(x)$ the Jacobian of $g$, one readily obtains $\nabla_x \Phi(x, y) = J_g(x)^\top y$ and $\nabla_y \Phi(x, y) = g(x)$.
	
	For all algorithms tested in QCQP, the stopping criterion is
	\[
	\max\left\{e_{\text{obj}}(x_n),\ e_{\text{con}}(x_n)\right\} \leq \varepsilon,
	\]
	where
	\begin{equation}
		e_{\text{obj}}(x) := |h(x) - h_{\text{opt}}| / |h_{\text{opt}}|
		\quad \text{and} \quad
		e_{\text{con}}(x) := \frac{1}{m} \sum_{j=1}^{m} \max\{g_j(x), 0\}.
		\label{eq:termination_errors}
	\end{equation} 
	Here $h_{\text{opt}}$ is computed by MOSEK through CVX\footnote{Downloaded from \url{http://cvxr.com/cvx/}.}. 
	In addition, the iteration stops if the maximum number of iterations $n_{\max}$ is reached. 
	In the experiments, we set $\varepsilon = 10^{-8}$ and $n_{\max} = 1 \times 10^4$.

    \textcolor{black}{ 
    To verify the effectiveness of the linearization at the convex combination $\bar{x}^{k}$ and the dual extrapolation update of $v^{k+1}$ in PDAce, we compare PDAce with GRPDA \cite{w10} and two variants (PDAce\_noe and PDAce\_noc) of PDAce on the QCQP problems to show the contribution of each component. Here, PDAce\_noe is the first variant, which uses $g(x^{k+1})$ instead of $v^{k+1}$ in the $y^{k+1}$ update and only keeps the convex combination strategy for $\bar{x}$; PDAce\_noc is the second variant, which removes the convex combination by using $x^{k}$ instead of $\bar{x}^{k}$ and only keeps the dual extrapolation step.}
    The results are presented in Table \ref{Table1} and Figure~\ref{fig1}. In each figure, the solid curve represents the median value computed from 10 independent random trials, while the shaded band around it indicates the spread (minimum to maximum) of the observed values across these runs.
    
    Before presenting the numerical results, the local regularity constants associated with the QCQP problem \eqref{eq:qcqp} are examined to provide guidance for stepsize selection. Under Assumption~\ref{ass2}, for any bounded convex primal set $\mathcal{X}$ and any $x,\tilde{x}\in\mathcal{X}$, the Jacobian of $g$ satisfies
    \[
    \begin{aligned}
        \|J_g(x)^\top y-J_g(\tilde{x})^\top\tilde{y}\|
        &\leq \|(J_g(x)-J_g(\tilde{x}))^\top y\|
        +\|J_g(\tilde{x})^\top(y-\tilde{y})\|\\
        &\leq \ell_{\mathcal{X}}\|y\|\|x-\tilde{x}\|
        +M_{\mathcal{X}}\|y-\tilde{y}\|.
    \end{aligned}
    \]
    If the dual variables satisfy $\|y\|\leq B$, the corresponding local bound is $L=\ell_{\mathcal{X}}B$.
    
    For the QCQP problem \eqref{eq:qcqp}, the primal feasible set $\mathcal{X}=[-10,10]^n$ is bounded and convex. Since $f=\iota_{\mathcal{X}}$ and $\bar{x}^0=x^0\in \mathcal{X}$, the proximal updates and their convex combinations remain in $\mathcal{X}$. The $j$-th row of $J_g(x)$ is $(A_jx+b_j)^\top$, and hence a valid Lipschitz constant for $J_g$ can be chosen as
    \[
    \ell=\sqrt{\left\|\sum\nolimits_{j=1}^{m}A_j^\top A_j\right\|_2}.
    \]
    Moreover, the Jacobian norm on $X$ admits the explicit upper bound
    \[
    M_{\mathcal{X}}:=\sup_{x\in \mathcal{X}}\|J_g(x)\|_2 \leq \left( \sum\nolimits_{j=1}^{m}
    \left(10\sqrt{n}\|A_j\|_2+\|b_j\|\right)^2 \right)^{1/2}.
    \]
    These estimates establish the required local regularity on the primal feasible set without imposing global boundedness of $J_g$. Nevertheless, even when the regularity conditions are restricted to bounded regions, obtaining sufficiently tight Lipschitz estimates for practical stepsize selection remains challenging. In particular, explicit bounds for the Jacobian norm may be conservative, and the theoretical stepsize conditions additionally require a suitable bound on the relevant dual region.
    
    For a representative QCQP instance with $n=500$ and $m=20$, a high-precision reference solution $(x^*,y^*)$ is obtained using CVX. The smoothness constant of $h$ is $L_h=\|A_0\|_2$, and the curvature of the coupling term at the reference dual point is characterized by
    \[
    L_{xx}^*:=\left\|\sum\nolimits_{j=1}^{m}y_j^*A_j\right\|_2.
    \]
    The Jacobian norm evaluated at the reference primal point is denoted by $\widehat{M}:=\|J_g(x^*)\|_2$. The computed values are
    \[
    L_h\approx99.9874,\qquad L_{xx}^*\approx113.8712,\qquad  \widehat{M}\approx40.2986.
    \]
    These quantities provide approximate information about the local behavior of the problem. In particular, $L_{xx}^*$ measures the coupling curvature at the fixed dual point $y^*$, whereas $\widehat{M}$ is the Jacobian norm at the reference primal point $x^*$. Therefore, neither $L_{xx}^*$ nor $\widehat{M}$ constitutes a certified uniform bound for the corresponding constants required by the convergence analysis.
    
    For PDAce, the theoretical stepsize conditions require
    \[
    0<\tau< \min\left\{ \tau_0,\, \frac{1}{\alpha\Lambda},\, \frac{1}{3\beta M},\, \frac{1}{\alpha(\ell B+M+2L_h)} \right\}, \qquad
    \Lambda=\ell B\left(\alpha+\frac{1-\alpha}{\alpha^2}\right)+2M,
    \]
    where $M$ is a valid upper bound for the Jacobian norm on the relevant primal region and $B$ is a bound for the corresponding dual region. To illustrate the numerical scale of these restrictions, the computed quantities $L_{xx}^*$ and $\widehat{M}$ are used as approximate substitutes for $\ell B$ and $M$, respectively. For $\alpha=0.86$ and $\beta=\sigma/\tau=3.0$, the resulting reference stepsizes are
    \[
    \begin{aligned}
        \widehat{\tau}_{\mathrm{ref}} :=\min\left\{
        \frac{1}{\alpha\widehat{\Lambda}},
        \frac{1}{3\beta\widehat{M}},
        \frac{1}{\alpha(L_{xx}^*+\widehat{M}+2L_h)}
        \right\}
        \approx27.5720\times10^{-4},\quad \widehat{\sigma}_{\mathrm{ref}}
        =\beta\widehat{\tau}_{\mathrm{ref}}
        \approx82.7159\times10^{-4},
    \end{aligned}
    \]
    where $\widehat{\Lambda}:=L_{xx}^*\left(\alpha+\frac{1-\alpha}{\alpha^2}\right)+2\widehat{M}$. These values are intended only as numerical references, since the approximate constants are not certified uniform bounds and the preliminary stepsize restriction $\tau_0$ is not included in this calculation. Consequently, they should not be interpreted as theoretically guaranteed admissible stepsizes. The stepsizes used in the subsequent numerical experiments are selected empirically to investigate the practical performance and parameter sensitivity of the algorithms, without claiming that they satisfy the sufficient theoretical stepsize conditions.
    
    The iteration counts and CPU runtimes of PDAce, PDAce\_noe, PDAce\_noc, and GRPDA are compared under five different empirically selected parameter configurations. The adopted parameter settings are listed as follows:
	\par $(\tau_1,\ \sigma_1,\ \alpha_1) = (20\!\times\!10^{-4},\ 35\!\times\!10^{-4},\ 0.92)$.
	\par $(\tau_2,\ \sigma_2,\ \alpha_2) = (20\!\times\!10^{-4},\ 40\!\times\!10^{-4},\ 0.92)$.
	\par $(\tau_3,\ \sigma_3,\ \alpha_3) = (20\!\times\!10^{-4},\ 20\!\times\!10^{-4},\ 0.92)$.
	\par $(\tau_4,\ \sigma_4,\ \alpha_4) = (20\!\times\!10^{-4},\ 10\!\times\!10^{-4},\ 0.86)$.
	\par $(\tau_5,\ \sigma_5,\ \alpha_5) = (39\!\times\!10^{-4},\ 63\!\times\!10^{-4},\ 0.86)$.
    
	\begin{table}[H]
		\centering
		\caption{Iteration counts and CPU times under identical parameter settings}
		\begin{tabular}{|c|c|c|c|c|c|c|}
			\hline
			\diagbox{\textbf{Algo}}{\textbf{Para}}
			& Metric & $(\tau_1,\sigma_1,\alpha_1)$ & $(\tau_2,\sigma_2,\alpha_2)$ & $(\tau_3,\sigma_3,\alpha_3)$ & $(\tau_4,\sigma_4,\alpha_4)$ & $(\tau_5,\sigma_5,\alpha_5)$ \\ 
			\hline
			\multirow{2}{*}{GRPDA}      
			& Iteration & 1839 & --- & 2089 & 4785 & --- \\ 
			& Time (s)  & 12.3110 & --- & 13.9111 & 31.9517 & --- \\ 
			\hline
			\multirow{2}{*}{PDAce\_noc} 
			& Iteration & 1347 & 1756 & 1853 & 4768 & 688 \\ 
			& Time (s)  & 10.5874 & 13.5741 & 14.3029 & 37.3776 & 6.2644 \\ 
			\hline
			\multirow{2}{*}{PDAce\_noe} 
			& Iteration & 1509 & 1546 & 1761 & 4851 & 705 \\ 
			& Time (s)  & 10.2630 & 12.0166 & 11.9716 & 32.3323 & 5.8733 \\
			\hline
			\multirow{2}{*}{PDAce}      
			& Iteration & 1472 & 1364 & 1702 & 4845 & 515 \\
			& Time (s)  & 11.7342 & 10.2866 & 13.1293 & 37.4661 & 4.7771 \\ 
			\hline
		\end{tabular}
		\vspace{0.3em}
		
		\footnotesize{Note: ``---'' indicates that the algorithm did not meet the stopping criterion within the prescribed iteration limit. ``Algo'' and ``Para'' denote Algorithm and its Parameters, respectively.}
        \label{Table1}
	\end{table}

    The results in Table \ref{Table1} indicate that the algorithms can exhibit stable convergence under empirically selected stepsizes that are not required to satisfy the sufficient theoretical conditions. This observation reflects the conservative nature of the analytical estimates and motivates the use of empirical parameter tuning in practical implementations.
    
    Furthermore, to ensure a fair and rigorous comparison, the hyperparameters of all evaluated algorithms were carefully fine-tuned to achieve their best empirical performance. The specific parameter settings employed in our numerical experiments are summarized as follows.
    \par PDAce\_noe: $\tau=40\times10^{-4},\ \sigma=53\times10^{-4},\ \alpha=0.86$.
    \par PDAce\_noc: $\tau=44\times10^{-4},\ \sigma=64\times10^{-4},\ \alpha=0.74$.
    \par GRPDA: $\tau=14\times10^{-4},\ \sigma=12\times10^{-4},\ \alpha=0.81$.
    \par PDAce: $\tau=69\times10^{-4},\ \sigma=199\times10^{-4},\ \alpha=0.66$.
    
	\begin{figure}[H]
		\centering
		\begin{subfigure}{0.47\textwidth}
			\centering
			\includegraphics[width=\linewidth]{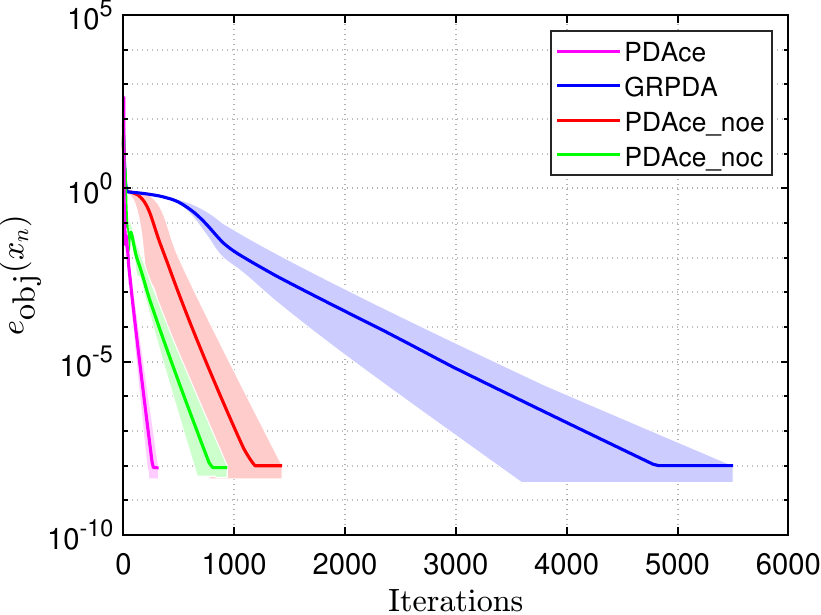}
            \caption{}
		\end{subfigure}
		\hfill
		\begin{subfigure}{0.47\textwidth}
			\centering
			\includegraphics[width=\linewidth]{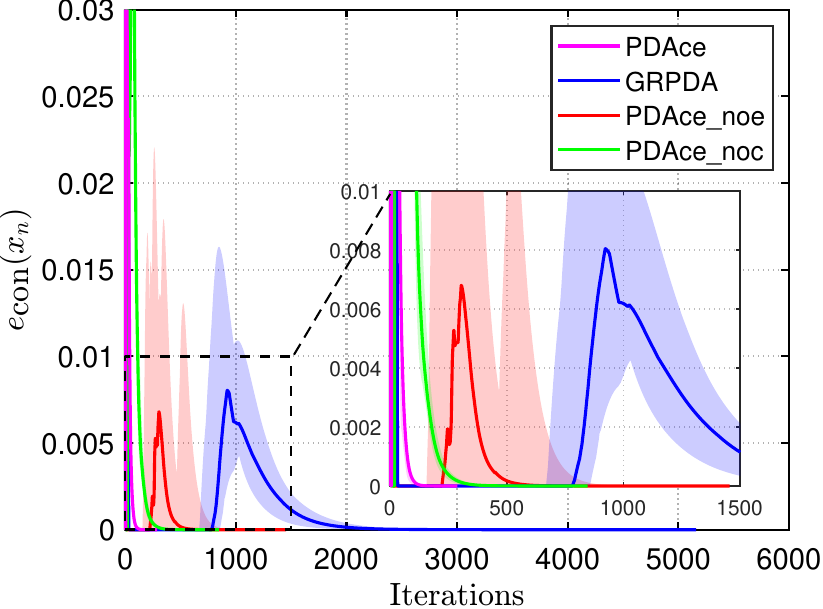}
            \caption{}
		\end{subfigure}
		
		\vspace{1em} 
		
		\begin{subfigure}{0.47\textwidth}
			\centering
			\includegraphics[width=\linewidth]{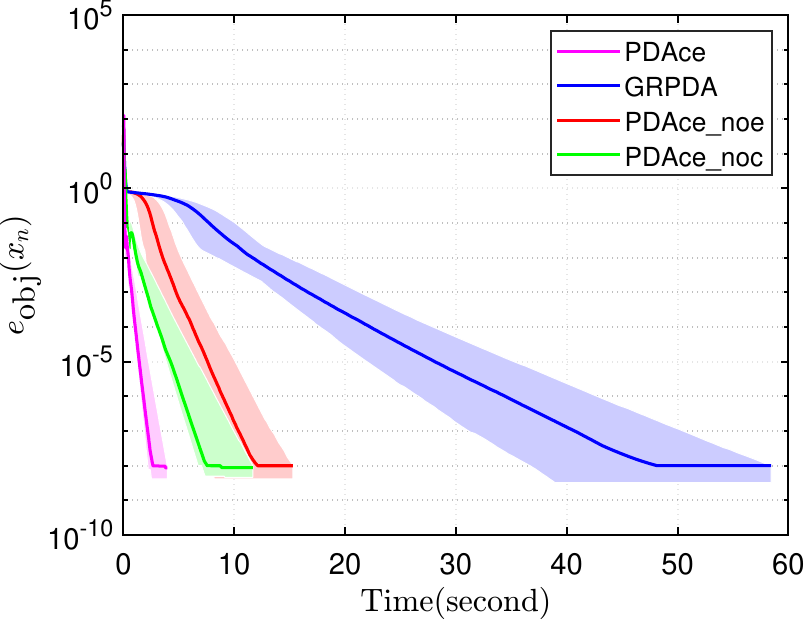}
            \caption{}
		\end{subfigure}
		\hfill
		\begin{subfigure}{0.47\textwidth}
			\centering
			\includegraphics[width=\linewidth]{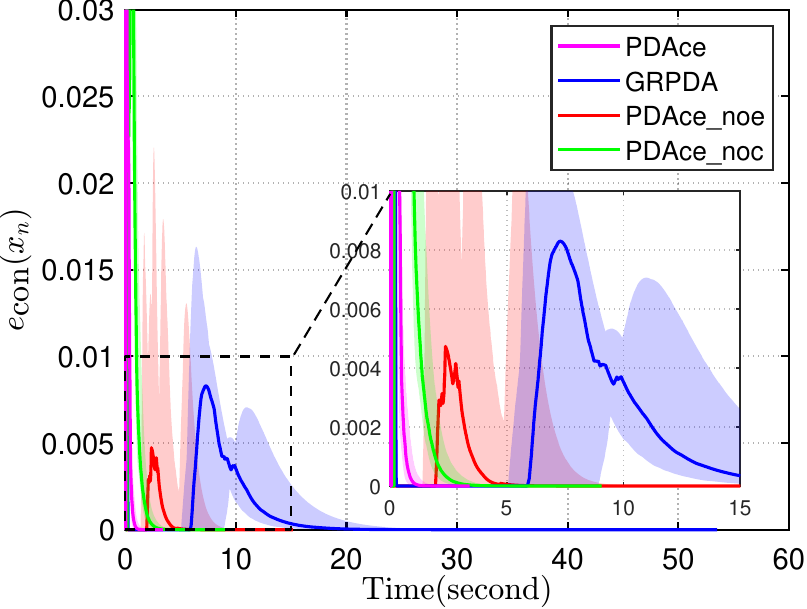}
            \caption{}
		\end{subfigure}
		
		\caption{Comparative performance of $\bar{x}^{k}$ (convex combination) versus $v^{k+1}$ (extrapolation). The subplots display objective residuals (left) and constraint violations (right) across 10 random QCQP instances with $n = 500$ and $m = 20$. Convergence behavior is shown as functions of iteration number (top row) and CPU time (bottom row).}
		\label{fig1}
	\end{figure}

    From Table \ref{Table1}, the PDAce family of algorithms (including PDAce\_noc and PDAce\_noe) achieves better stability and computational efficiency than GRPDA. Specifically, GRPDA is sensitive to parameters and diverges under large stepsizes. The proposed algorithm exhibits superior performance under large stepsize configurations. In contrast, all PDAce algorithms converge stably under all parameter settings, requiring fewer iterations and shorter runtimes in most cases.
    However, due to the update of the auxiliary variable $v^{k+1}$, PDAce and PDAce\_noc need to compute the nonlinear function $g(x)$ twice, which introduces higher computational complexity and longer running time under the same parameters. In addition, since primal-dual algorithms are sensitive to both stepsizes and their ratios, our method does not outperform the competing methods under every parameter configuration.
    
	The experimental results of Figure~\ref{fig1} show that PDAce converges significantly faster than PDAce\_noe (without extrapolation) and PDAce\_noc (without convex combination), and also outperforms the baseline algorithm GRPDA. These results support the effectiveness of combining the two components on the tested QCQP instances.
	
	\subsection{Numerical comparison on the QCQP test problems}
	In this subsection, we test the overall performance of PDAce against three competing methods: the accelerated primal-dual algorithm (APD) \cite{w11}, the golden ratio primal-dual algorithm (GRPDA) \cite{w10}, and the primal-dual algorithm with convex combination, linesearch and adaptive $\beta$ strategy (PDAc-L) \cite{w5}.
	All experiments are performed on the QCQP test problems described in Subsection \ref{sec5.1} with the same settings.

    In all numerical experiments, the parameters of PDAc-L and aPDAc-L  are set as follows: $\psi = 2$, $\varphi = 6/5$, $\nu = 0.9$, $\mu = 0.7$, $\xi = 2/5$, $M = 5$ and $\eta = 0.9$~\cite{w5}.

	\subsubsection{Performance of PDAce, GRPDA, APD\texorpdfstring{$(\mu=0)$}{(mu=0)} and PDAc-L}
	We evaluate the performance of four algorithms, namely PDAce, GRPDA, APD$(\mu=0)$, and PDAc-L, on 10 random QCQP instances with $n=500$ and $m=20$. Figure \ref{fig2} plots the objective error $e_{\text{obj}}(x_n)$ and the constraint violation $e_{\text{con}}(x_n)$ as defined in \eqref{eq:termination_errors}, against both iteration count and CPU time. As in Figure \ref{fig1}, the solid curves represent medians over the 10 instances, while the shaded bands capture the spread of the observed values.

    The parameters of PDAce are selected empirically to investigate its practical acceleration behavior. In particular, the numerical parameter choices are not required to satisfy the sufficient conditions imposed in the convergence analysis, which may be conservative for the tested QCQP instances. Therefore, in the numerical tests, the parameters of the competing algorithms are chosen as follows.
    \par GRPDA: $\tau=14\times10^{-4},\ \sigma=12\times10^{-4},\ \alpha=0.81$.
    \par PDAce: $\tau=69\times10^{-4},\ \sigma=199\times10^{-4},\ \alpha=0.66$.
    \par APD($\mu=0$): $\tau=2\times10^{-4},\ \sigma=14\times10^{-4},\ \mu=0,\ \theta=1$.

    \begin{figure}[H]
		\centering
		\begin{subfigure}{0.47\textwidth}
			\centering
			\includegraphics[width=\linewidth]{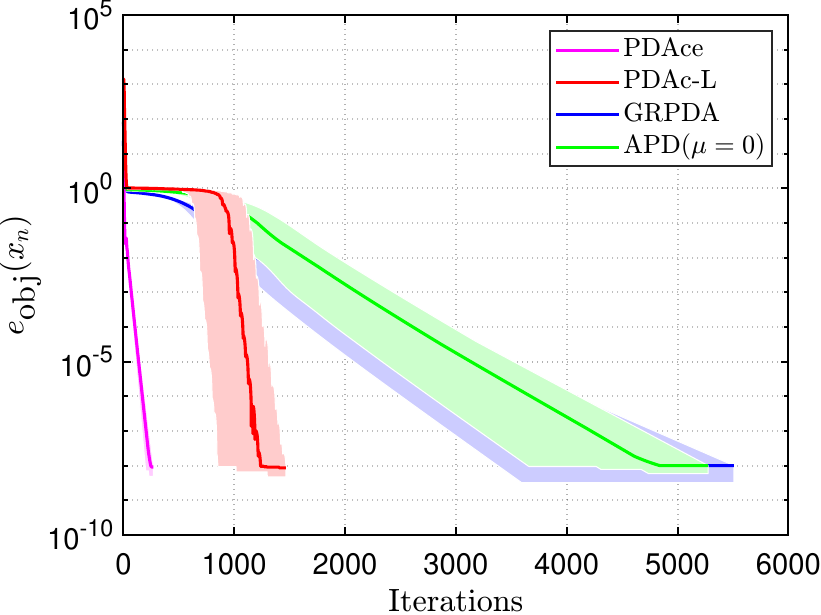}
            \caption{}
		\end{subfigure}
		\hfill
		\begin{subfigure}{0.47\textwidth}
			\centering
			\includegraphics[width=\linewidth]{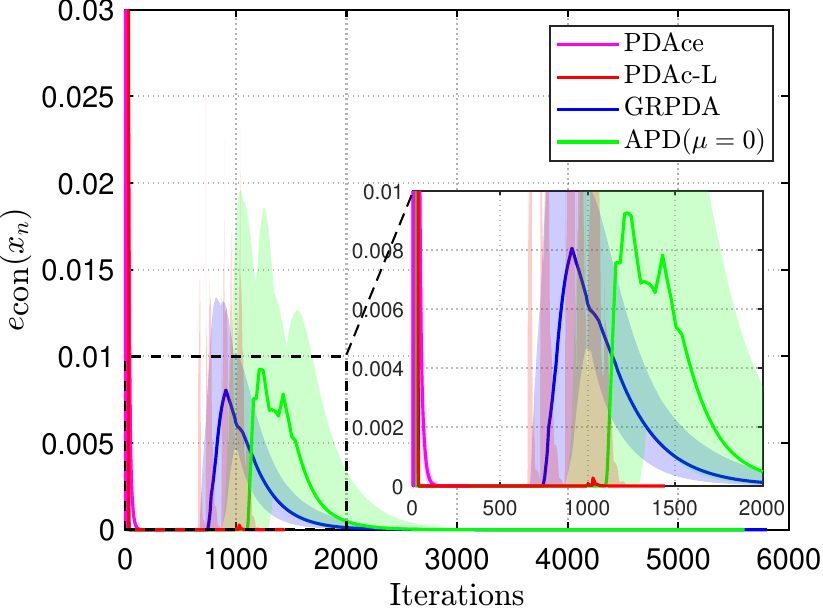}
            \caption{}
		\end{subfigure}
		
		\vspace{1em} 
		
		\begin{subfigure}{0.47\textwidth}
			\centering
			\includegraphics[width=\linewidth]{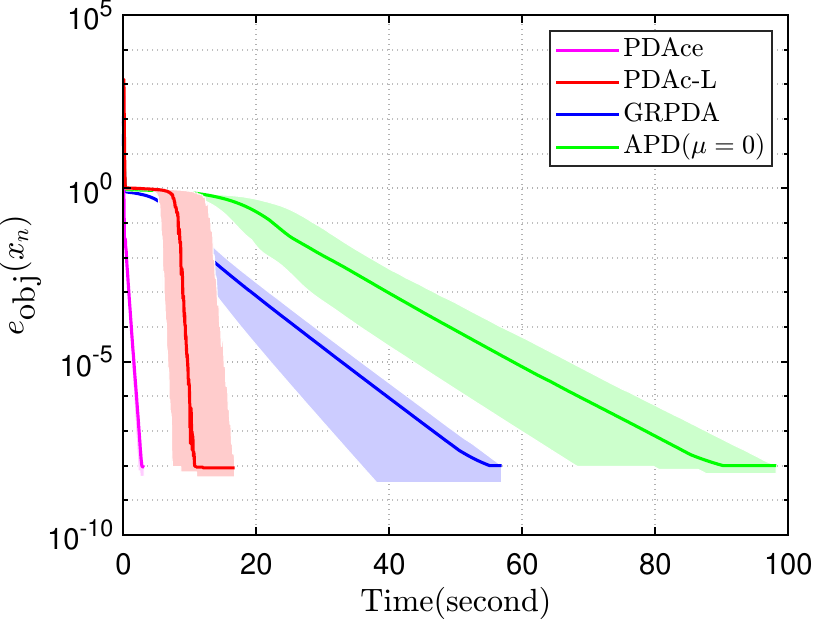}
            \caption{}
		\end{subfigure}
		\hfill
		\begin{subfigure}{0.47\textwidth}
			\centering
			\includegraphics[width=\linewidth]{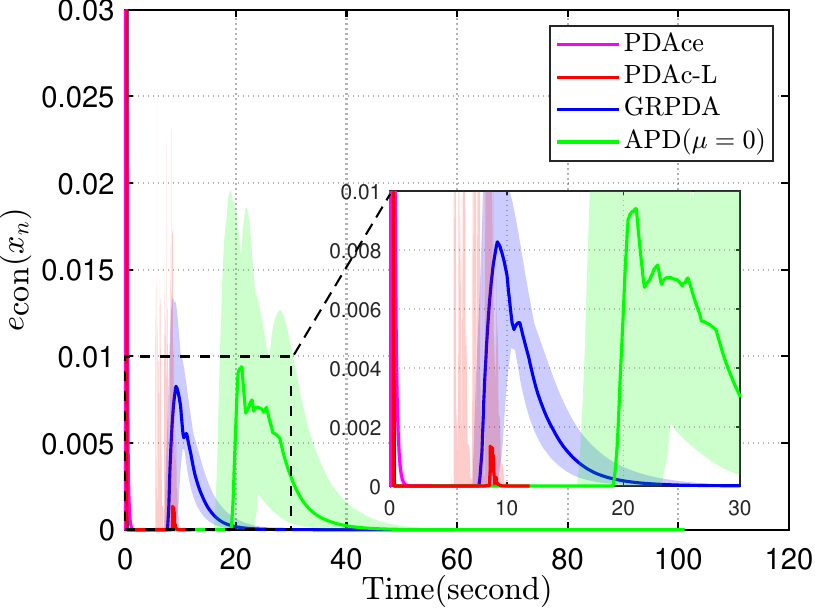}
            \caption{}
		\end{subfigure}
		
		\caption{Performance comparison on 10 random QCQP instances with $n=500$ and $m=20$. The left and right panels report function value residuals and feasibility violations, respectively. Convergence behavior is shown as functions of iteration number (top row) and CPU time (bottom row).}
		\label{fig2}
	\end{figure}
    \textcolor{black}{For the QCQP problems, PDAce exhibits a remarkable numerical advantage over the the other tested algorithms in terms of both iteration counts and CPU time.} Compared with APD, GRPDA and PDAc-L, PDAce has steeper convergence curves and smaller shaded areas (variation ranges) in multiple random experiments, demonstrating excellent numerical stability and robustness.
    
	\subsubsection{Performance of the Accelerated Variant}
	We further test the efficiency of the accelerated algorithm aPDAce (Algorithm \ref{alg:aPDAce}) \textcolor{black}{by comparing it with PDAce, PDAc-L, aPDAc-L, and APD($\mu>0$)} on strongly convex QCQP problems. The strongly convex QCQP problem is the same as the convex case in section \ref{sec5.1}, except that the diagonal entries of $S_0$ are sampled uniformly from $[1, 101]$. With this setting, $S_0 = \tilde{S}_0 + I$ and $A_0 = \tilde{A}_0 + I$, so we can write the objective function as $h(x) = \tilde{h}(x) + \frac{1}{2}\|x\|^2$.
	Thus, the problem becomes a strongly convex-concave minimax problem:
	$
	\min_x \max_y \, f(x) + \tilde{h}(x) + \Phi(x,y) - \iota_+(y),
	$
	where $f(x) = \iota_X(x) + \frac{1}{2}\|x\|^2$ is strongly convex, $\Phi(x,y) = \langle g(x), y \rangle$, $H^*(y) = \iota_+(y)$, 
    and $x \in \mathcal{X} = [-10,10]^n$, $y \in \mathcal{Y} = \mathbb{R}^m_{+}$.
    We evaluate the tested algorithms on 10 randomly generated strongly convex QCQP instances with $n=500$ and $m=20$. Let $\lambda_{\min}(A_0)$ be the smallest eigenvalue of $A_0$. The parameter choices for each algorithm in our numerical study are specified below.
    \par aPDAce: $\mu_f = \lambda_{min}(A_0)$, $\varphi = 1.3$, $\xi = 2$, $M = 120$, $\tau_0 = 10\times 10^{-4}$, $\beta_0 = 1$, $\tau_{max} = 1$, $\alpha = 0.76$.
    \par PDAce: $\tau = 54\times 10^{-4}$, $\sigma = 60\times 10^{-4}$, $\alpha = 0.74$.
    \par APD$(\mu>0)$: $\mu_f = \lambda_{min}(A_0)$, $\tau_0 = 4\times 10^{-4}$, $\sigma_0 = 12\times 10^{-4}$, $\theta_0 = 1$.

    The objective error $e_{\text{obj}}(x_n)$ and constraint violation $e_{\text{con}}(x_n)$ are plotted in Figure \ref{fig3} as functions of the iteration count. As shown in Figure \ref{fig3}, aPDAce converges much faster than PDAce and also outperforms the other tested algorithms.
    \begin{figure}[H]
		\centering
		\begin{subfigure}{0.47\textwidth}
			\centering
			\includegraphics[width=\linewidth]{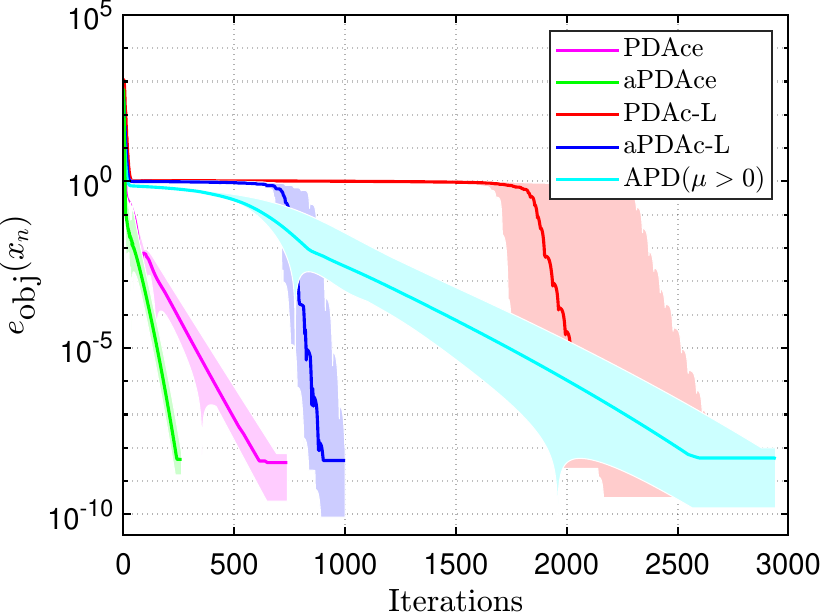}
            \caption{}
		\end{subfigure}
		\hfill
		\begin{subfigure}{0.47\textwidth}
			\centering
			\includegraphics[width=\linewidth]{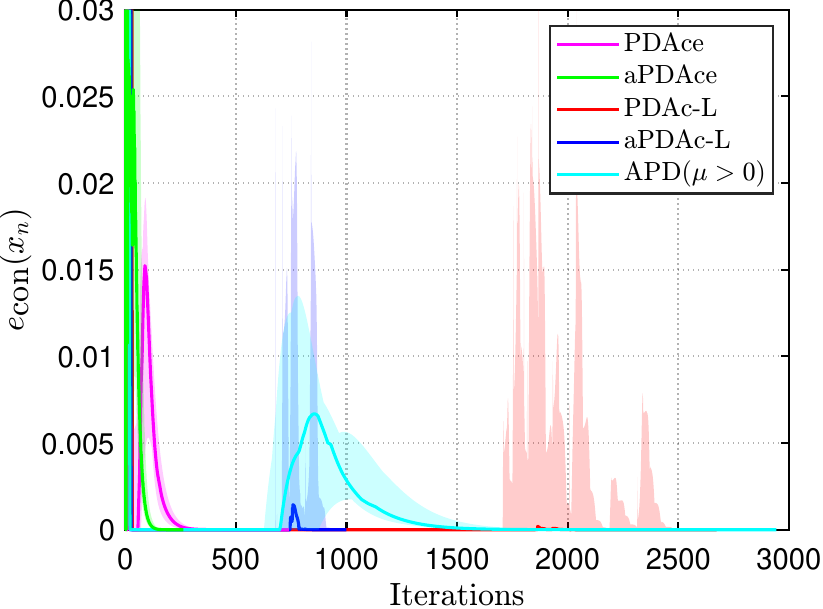}
            \caption{}
		\end{subfigure}
		
		\vspace{1em} 
		
		\begin{subfigure}{0.47\textwidth}
			\centering
			\includegraphics[width=\linewidth]{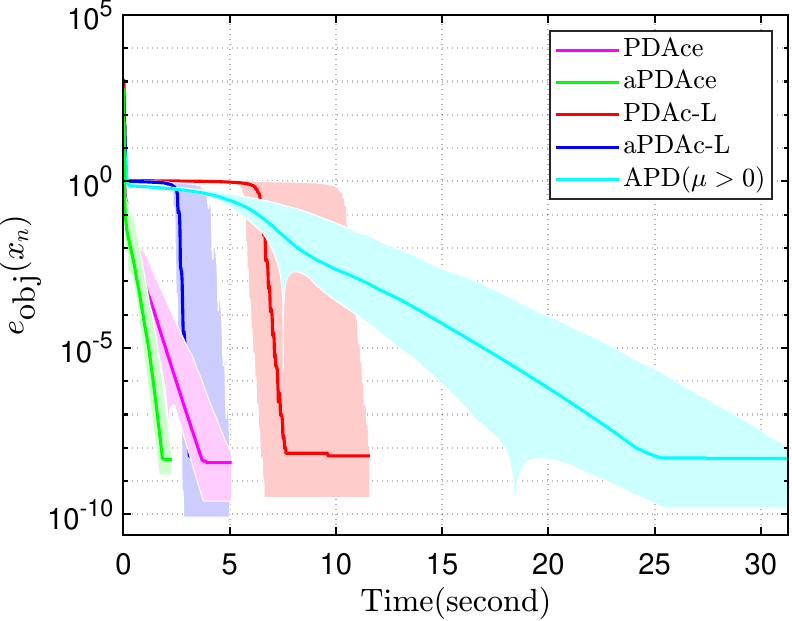}
            \caption{}
		\end{subfigure}
		\hfill
		\begin{subfigure}{0.47\textwidth}
			\centering
			\includegraphics[width=\linewidth]{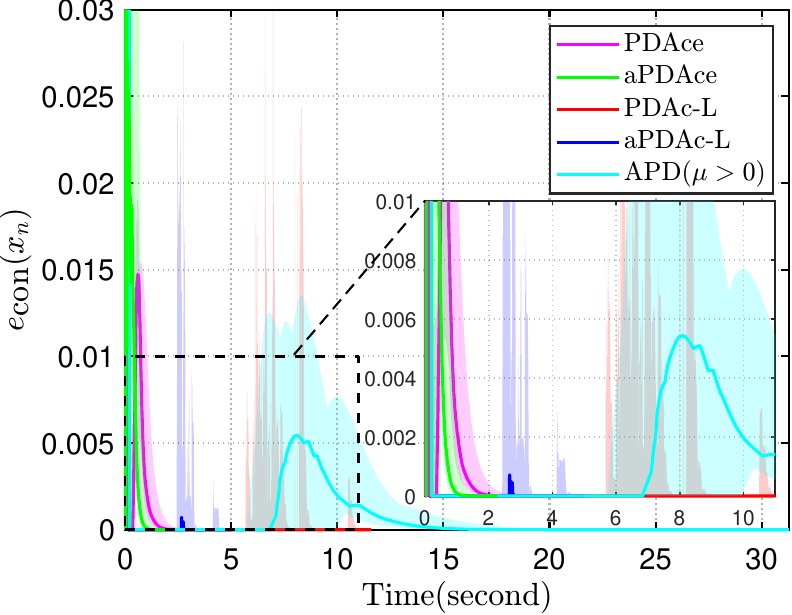}
            \caption{}
		\end{subfigure}
		
		\caption{Performance comparison on 10 random strongly convex QCQP instances with $n=500$ and $m=20$. The left and right columns report function value residuals and feasibility violations, respectively. Convergence  behavior is plotted versus iteration number (top row) and CPU time (bottom row).}
		\label{fig3}
	\end{figure}
    
	\subsection{Numerical comparison on  convex-concave minimax game problem}
    \textcolor{black}{We study a two-player convex-concave minimax game (CCMMG): 
	\begin{equation}\label{eq:minimax_game}
		\min_{x \in \Delta_n} \max_{y \in \Delta_m}
		\mathcal{L}(x,y)
		:= \frac{1}{N} \sum_{j=1}^N \log\bigl(1 + \exp(a_j^\top x)\bigr) + \langle g(x), l(y) \rangle. 
	\end{equation}
	It is not difficult to check that \eqref{eq:minimax_game} is a special case of \eqref{prob1}, where $\Delta_n := \left\{ x \in \mathbb{R}_+^n : \sum_{i=1}^n x_i = 1 \right\}$, $\Delta_m := \left\{ y \in \mathbb{R}_+^m : \sum_{i=1}^m y_i = 1 \right\}$, $f(x)=\iota_{\Delta_n}(x)$,  $H^*(y)=\iota_{\Delta_m}(y)$,
	$h(x)  = \frac{1}{N} \sum_{j=1}^N \log(1 + e^{a_j^\top x})$, $\Phi(x,y) := \langle g(x), l(y) \rangle$.
    Here, $g(x) = (g_1(x),g_2(x),\dots,g_m(x))^\top$ and $l(y) = y$, where $g_j(x) = \frac12 x^\top A_j x + b_j^\top x + c_j$ and $l_j(y) = y_j$ for all $j\in\{1,2,\dots,m\}$. This problem is also a standard test problem in \cite{w10,w11,w13}. }
    
    It is not difficult to verify that a pair $(x^*, y^*)$ is \textcolor{black}{a saddle point of \eqref{eq:minimax_game}} if and only if it satisfies the following fixed-point equations:
	\[
	x^* = \mathcal{P}_{\Delta_n}\big(x^* - \nabla_x \mathcal{L}(x^*, y^*)\big),\quad
	y^* = \mathcal{P}_{\Delta_m}\big(y^* + \nabla_y \mathcal{L}(x^*, y^*)\big),
	\]
    where $\mathcal{P}_{\Delta_n}$ and $\mathcal{P}_{\Delta_m}$ are the projection operators onto the feasible sets $\mathcal{X}$ and $\mathcal{Y}$.
    The residuals used in these numerical experiments are defined as follows:
	\begin{align*}
		r_x(x^k, y^k) &= \left\| x^k - \mathcal{P}_{\Delta_n} \left( x^k - \nabla_x \mathcal{L}(x^k, y^k) \right) \right\|,\\
		r_y(x^k, y^k) &= \left\| y^k - \mathcal{P}_{\Delta_m} \left( y^k + \nabla_y \mathcal{L}(x^k, y^k) \right) \right\|.
	\end{align*}
    The total error of the algorithm is defined as
	\[
	\text{Error}(k) = \max \left\{ r_x(x^k, y^k), r_y(x^k, y^k) \right\}.
	  \]
	These residuals quantify the violation of the fixed-point conditions at the current iterate. The test algorithm  terminates when the stopping condition $$\max\{r_x, r_y\} \le 10^{-8} $$ is satisfied, which means the solution is close enough to the desired fixed point.

    \textcolor{black}{We compare the performance of PDAce with GRPDA, APD($\mu=0$) and PDAc-L in this experiment. The parameters of PDAc-L remain unchanged as in previous experiments, and the parameter configurations of other algorithms are given below.}
    \par GRPDA: $\tau=16\times10^{-4},\ \sigma=1200\times10^{-4},\ \alpha=0.618$.
    \par PDAce: $\tau=80\times10^{-4},\ \sigma=1200\times10^{-4},\ \alpha=0.618$.
    \par APD($\mu=0$): $\tau=22\times10^{-4},\ \sigma=1200\times10^{-4},\ \mu=0,\ \theta=1$.
    
    \textcolor{black}{Numerical results are illustrated in Figure~\ref{fig4} (a, b). PDAce significantly outperforms GRPDA and PDAc-L, and also slightly outperforms APD($\mu=0$). The shaded areas correspond to the result ranges of ten test problems, and solid lines refer to their medians. Clearly, narrower shaded regions indicate lower variability across the tested instances.}
    As clearly observed in Figure~\ref{fig4} (a, b), PDAce has narrower shaded regions than GRPDA, APD($\mu=0$) and PDAc-L, confirming its superior stability.
   
    We further investigate the algorithms' parameter sensitivity via stepsize robustness heatmaps in Figure~\ref{fig4} (c)–(e). The color bar shows iteration counts for the tested algorithms. Blank areas denote divergence, dark blue regions represent convergence within 2000 iterations, and yellow regions indicate slow or non-convergence at the $10^4$ iteration limit.
    As can be easily observed in Figure~\ref{fig4} (c)–(e), compared with GRPDA and APD($\mu=0$), PDAce exhibits a significantly wider dark blue (effective convergence) region and remains stable, particularly for large stepsizes. These results suggest that PDAce admits a wider empirically observed range of effective stepsizes on the tested instances.

    \begin{figure}[H]
		\centering
		\begin{subfigure}{0.47\textwidth}
			\centering
			\includegraphics[width=\linewidth]{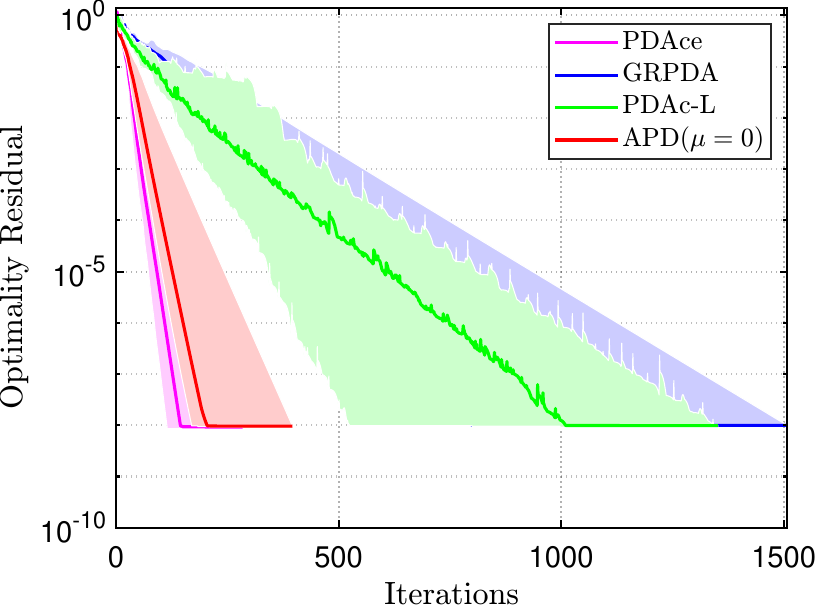}
            \caption{}
		\end{subfigure}
		\hfill
		\begin{subfigure}{0.47\textwidth}
			\centering
			\includegraphics[width=\linewidth]{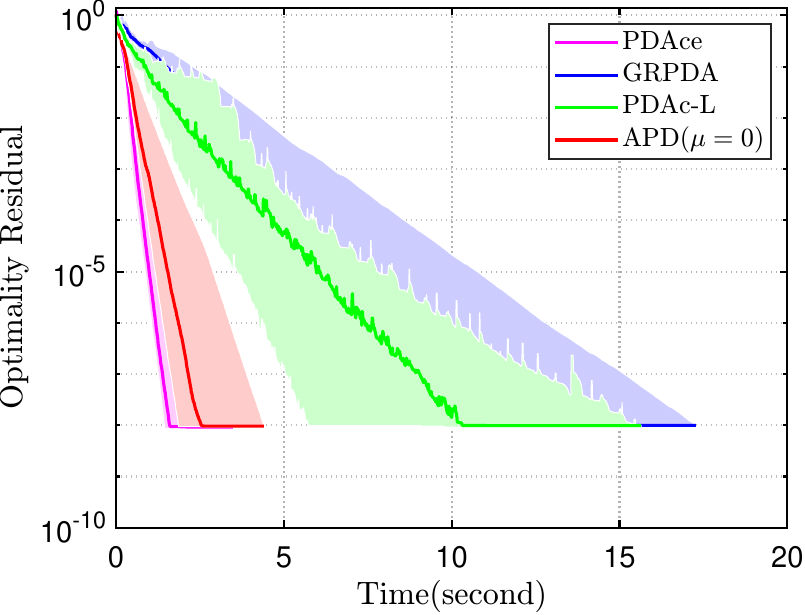}
            \caption{}
		\end{subfigure}
		
		\vspace{1em} 
		
		\begin{subfigure}{0.31\textwidth}
			\centering
			\includegraphics[width=\linewidth]{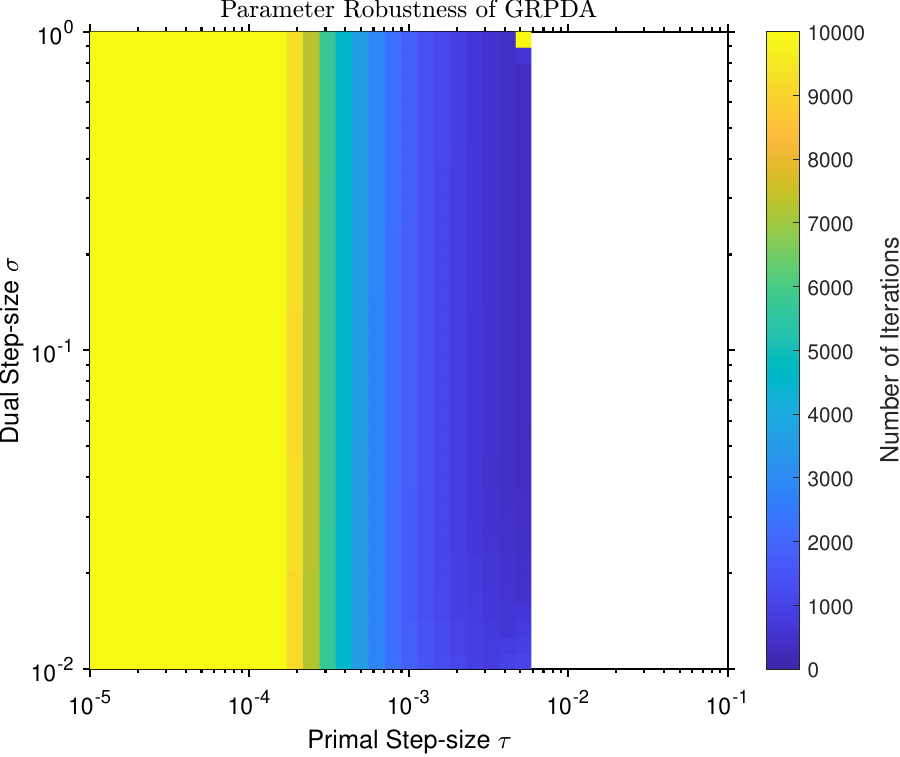}
            \caption{}
		\end{subfigure}
		\hfill
		\begin{subfigure}{0.31\textwidth}
			\centering
			\includegraphics[width=\linewidth]{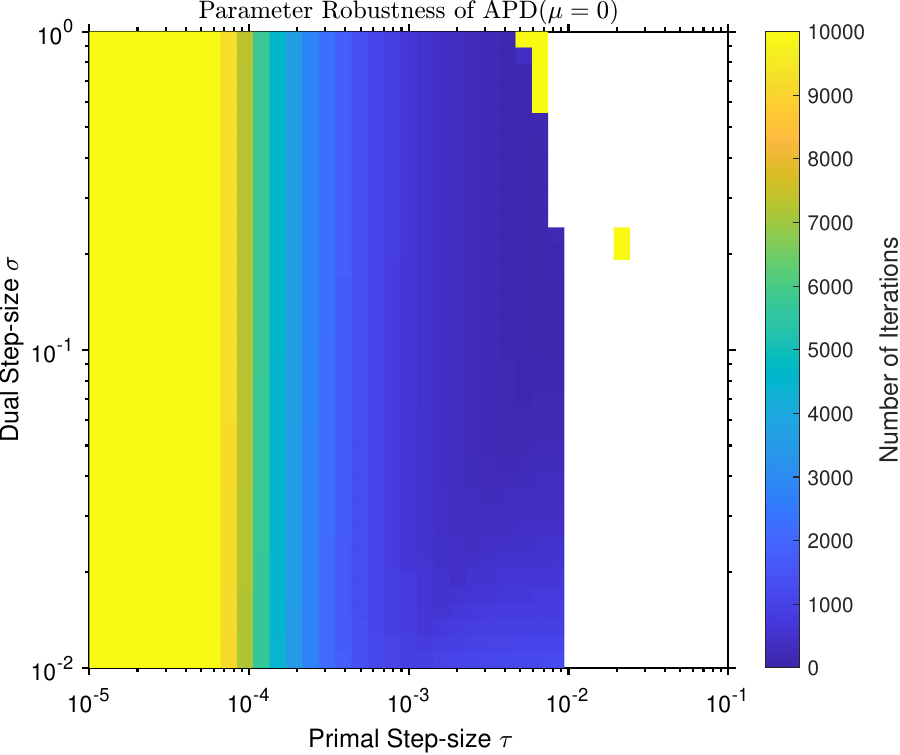}
            \caption{}
		\end{subfigure}
		\hfill
		\begin{subfigure}{0.31\textwidth}
			\centering
			\includegraphics[width=\linewidth]{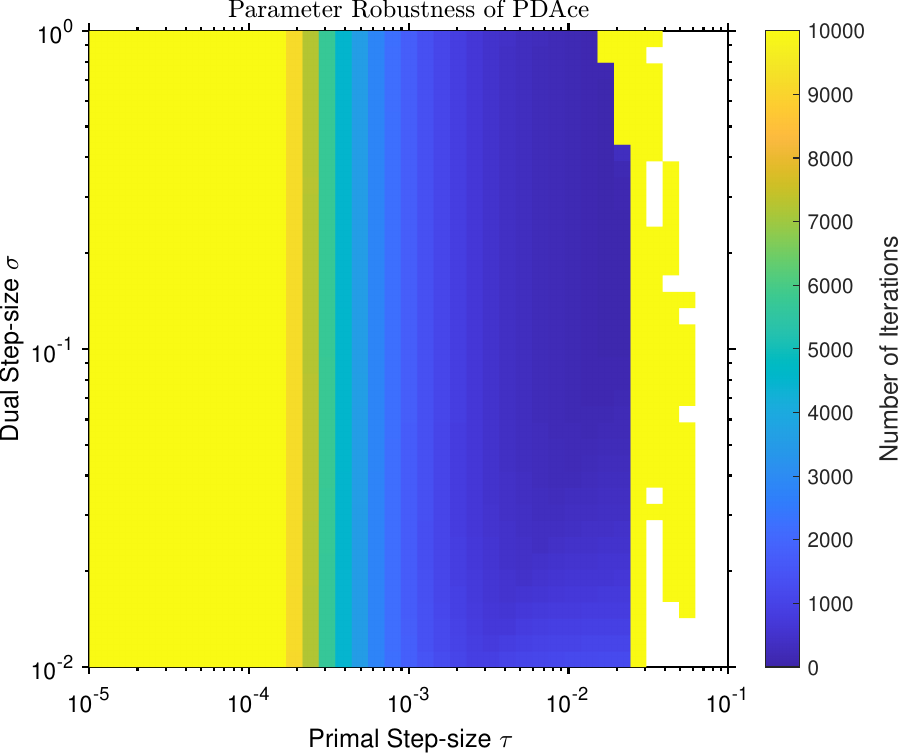}
            \caption{}
		\end{subfigure}
        
		\caption{ Comparison results in terms of convergence behavior versus iteration counts and CPU time (first row) and empirical stepsize robustness tests (second row) on 10 random CCMMG instances with $n = 100$ and $m = 300$.}
		\label{fig4}
	\end{figure}
    
	\subsection{Numerical comparison on structured convex-concave minimax problem}
    To further evaluate the numerical performance of PDAce, the following structured saddle point problem (SPP) with exponential coupling \textcolor{black}{term} is considered \cite{w9}:
    \begin{equation}
        \min_{x\in\mathbb{R}^n}\max_{y\in\mathbb{R}_+^m}\mathcal{L}(x,y):=\|Ax-a\|^2+\langle y,e^{Bx}-b\rangle-\|Cy-c\|^2,
        \label{eq:exp_coupling_spp}
    \end{equation}
    where $A\in\mathbb{R}^{n\times n}$, $B\in\mathbb{R}^{m\times n}$, $C\in\mathbb{R}^{m\times m}$, $a\in\mathbb{R}^n$, $b\in\mathbb{R}^m$, and $c\in\mathbb{R}^m$. The exponential function is applied componentwise. Such problems arise in applications including chemical reaction control, battery charging control, and population dynamics control.
    
    For this problem, the nonlinear coupling mapping is $g(x)=e^{Bx}-b$, with Jacobian $J_g(x)=\operatorname{diag}(e^{Bx})B$. Although $J_g$ is not necessarily globally bounded or globally Lipschitz continuous, it is bounded and Lipschitz continuous on every bounded primal region. Thus, the exponential coupling provides a representative example of the local regularity conditions considered in the convergence analysis.
    
    The numerical performance of PDAce is compared with APD$(\mu=0)$, GRPDA, and PDAc-L. The parameter settings are given below.
    \par PDAce: $\tau=0.017,\ \sigma=0.05,\ \alpha=0.618$.
    \par GRPDA: $\tau=0.01,\ \sigma=0.05,\ \alpha=0.618$.
    \par APD$(\mu=0)$: $\tau=0.005,\ \sigma=0.005,\ \mu=0,\ \theta=1$.
    
    Following \cite{w9}, the distance measure is defined by
    \[
    D_{k+1}:=\frac{1}{2}\left(\|x^{k+1}-x^*\|+\|y^{k+1}-y^*\|\right),
    \]
    where $(x^*,y^*)$ denotes a reference saddle point. This quantity measures the average distance of the primal and dual iterates from the reference solution.
    
    \begin{figure}[H]
        \centering
        \includegraphics[width=0.47\linewidth]{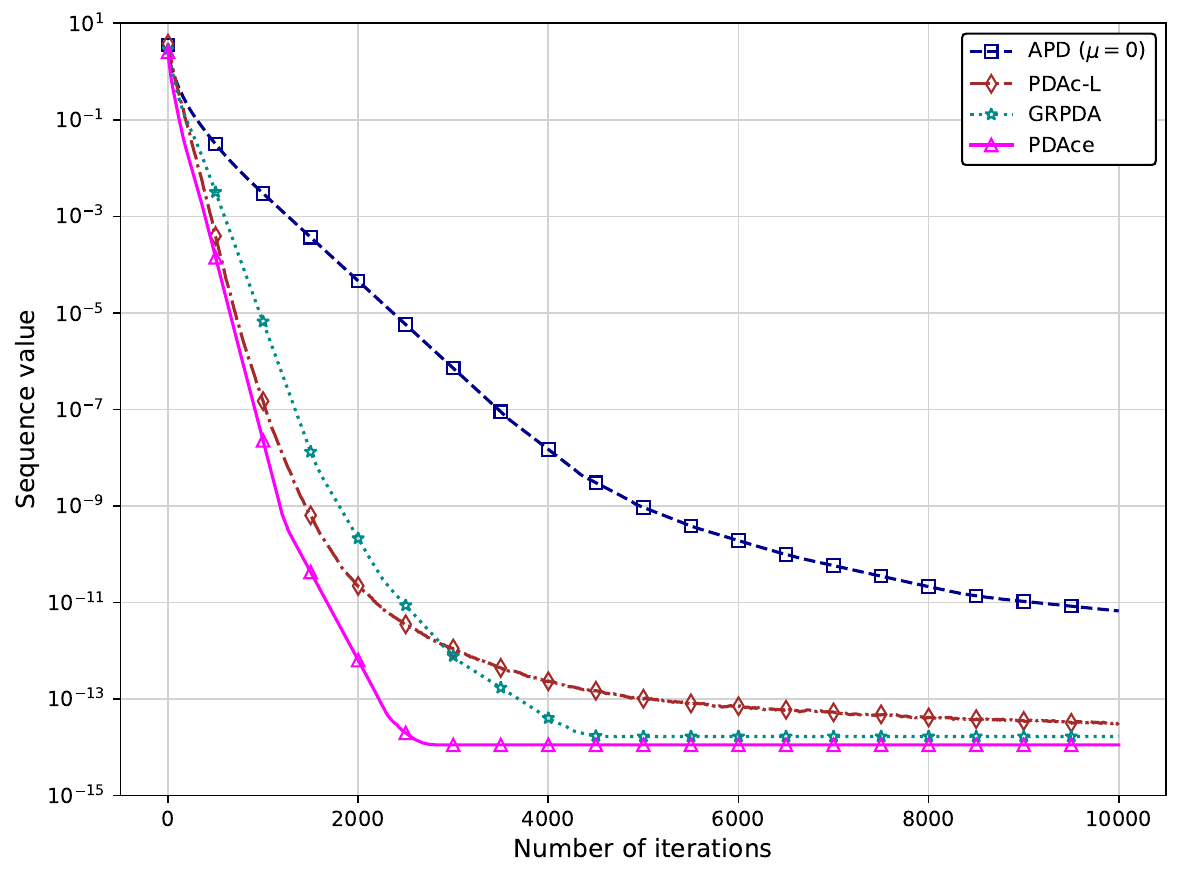}
        \caption{Convergence behavior of PDAce and competing first-order methods on the saddle point problem with exponential coupling. The sequence value $D_{k+1}$ is plotted against the iteration count.}
        \label{fig5}
    \end{figure}
    
    The numerical results are presented in Figure~\ref{fig5}. PDAce exhibits a faster decrease in the sequence value than APD$(\mu=0)$, GRPDA, and PDAc-L, indicating faster convergence toward the reference saddle point in the tested problem. These results suggest that the proposed convex combination and dual extrapolation strategies are effective for saddle point problems with nonlinear exponential coupling.
	
	\section{Conclusion and discussion}\label{section6}
    This paper proposes a new primal-dual algorithm (PDAce) that integrates convex combination and dual extrapolation for convex-concave saddle point problems with nonlinear coupling. Under local Lipschitz continuity assumptions, a Lyapunov potential function is constructed to establish the boundedness and convergence of the generated sequence, together with an $\mathcal{O}(1/N)$ ergodic convergence rate. When both $f$ and $H^*$ are strongly convex, linear convergence is established. Under the strong convexity of $f$, an accelerated variant (aPDAce) is developed, achieving an $\mathcal{O}(1/N^2)$ ergodic convergence rate. Numerical experiments demonstrate the effectiveness and competitiveness of the proposed algorithms compared with existing primal-dual methods.
    
    Another possible direction is to investigate the use of the convex combination point in the linearization of the smooth function $h$, together with its convergence properties and numerical performance. Although this modification is not adopted in the current PDAce algorithm, its convergence analysis would require a suitable Lyapunov potential function to accommodate the change in the linearization point. This observation provides a possible direction for further improving the numerical performance of primal-dual algorithms.
    
    An important issue for practical implementations is the selection of admissible stepsizes under local Lipschitz conditions. Although the convergence analysis does not require a global bound on the Jacobian of the coupling mapping, the accelerated algorithm relies on a valid Jacobian norm bound $M$ over a bounded primal region. Such a bound can be explicitly estimated when the primal feasible set is known and bounded, as in the QCQP experiments. For more general problems, identifying a suitable bounded region and estimating the corresponding local constants may be difficult. Developing adaptive stepsize rules or backtracking linesearch strategies that avoid requiring these bounds in advance is therefore a meaningful direction for future research.
    
    The Lyapunov framework also provides a possible basis for investigating primal-dual methods beyond the convex-concave setting. Extending the analysis to nonconvex-concave saddle point problems, while maintaining effective control of the nonlinear coupling residuals, remains an interesting topic for further investigation.\\

    \noindent\textbf{Declaration of competing interest}\\
    	
    \hspace{1em}The authors declare that they have no known competing financial interests or personal relationships that could have appeared to influence the work reported in this paper.\\
    	
    \noindent\textbf{Acknowledgements}\\
    	
    \hspace{1em} 
    This research is supported by the National Natural Science Foundation of China (Nos.~12261019, 12571329, 12571320).\\
    	
    \noindent\textbf{Data availability}\\
    	
    \hspace{1em}We do not analyse or generate any datasets, because our work proceeds within a theoretical and mathematical approach.\\
	
	\noindent\textbf{CRediT authorship contribution statement}\\
	
	\textbf{Jialong Li:} Writing - original draft, Investigation, Formal analysis, Conceptualization, Visualization. 
	\textbf{Zexian Liu:} Writing - review \& editing, Conceptualization, Validation, Software, Methodology, Funding acquisition, Resources, Supervision. \textbf{Xiaokai Chang:} Writing - review \& editing, Methodology, Supervision, Validation.

    \appendix
    \setcounter{equation}{0}
    \setcounter{figure}{0}
    \setcounter{table}{0}
    \begin{appendices}
    \section{Proof of Lemma \ref{lem9}}\label{app_lem9}
    \begin{proof}
		By Lemma \ref{lem1}, the update rules in Algorithm \ref{alg:aPDAce}, and the strong convexity of $f$, we obtain:
		\begin{align}
			f(x^{k+1}) - f(x) &\le \left\langle \nabla h(x^k) + J_g(\bar{x}^k)^\top y^k,\ x - x^{k+1} \right\rangle \nonumber \\
			&\quad + \frac{1}{\tau_k} \left\langle x^{k+1} - \bar{x}^k,\ x - x^{k+1} \right\rangle - \frac{\mu_f}{2} \| x^{k+1} - x \|^2, \label{eq47}\\[6pt]
			f(x^k) - f(x^{k+1}) &\le \left\langle \nabla h(x^{k-1}) + J_g(\bar{x}^{k-1})^\top y^{k-1},\ x^{k+1} - x^k \right\rangle \nonumber \\
			&\quad + \frac{1}{\tau_{k-1}} \left\langle x^k - \bar{x}^{k-1},\ x^{k+1} - x^k \right\rangle - \frac{\mu_f}{2} \| x^{k+1} - x^k \|^2, \label{eq48}\\[6pt]
			H^*(y^k) - H^*(y) &\le \left\langle g(\bar{x}^{k}) - \alpha\left(g(\bar{x}^{k-1}) - g(x^k)\right),\ y^k - y \right\rangle \nonumber \\
            &\quad + \frac{1}{\beta_k \tau_k} \left\langle y^k - y^{k-1},\ y - y^k \right\rangle.\label{eq49}
		\end{align}
		Since $x^k - \bar{x}^{k-1} = \frac{1}{\alpha}(x^k - \bar{x}^k)$, substituting into (\ref{eq48}) and rearranging gives
		\begin{align}
			f(x^k) - f(x^{k+1}) &\le \left\langle \nabla h(x^{k-1}) + J_g(\bar{x}^{k-1})^\top y^{k-1},\ x^{k+1} - x^k \right\rangle \nonumber \\
			&\quad + \frac{\delta_k}{\alpha \tau_k} \left\langle x^k - \bar{x}^k,\ x^{k+1} - x^k \right\rangle - \frac{\mu_f}{2} \| x^{k+1} - x^k \|^2. \label{eq50}
		\end{align}
		By the convexity of $h(x)$, the desired gradient inequality follows.
		\begin{align}
			h(x^k) - h(x) \le \left\langle \nabla h(x^k), x^k - x \right\rangle.\label{eq51}
		\end{align}
		Summing inequalities \eqref{eq47} and \eqref{eq49}--\eqref{eq51} and using the definition of the primal-dual residual $Q_k(x,y)$ yields
		\begin{align}
			& \qquad Q_k(x,y) \nonumber\\
			& \le \left\langle \nabla h(x^k) + J_g(\bar{x}^k)^\top y^k,\ x - x^{k+1} \right\rangle
			+ \left\langle \nabla h(x^{k-1}) + J_g(\bar{x}^{k-1})^\top y^{k-1},\ x^{k+1} - x^k \right\rangle \nonumber\\
			& \quad + \frac{1}{\tau_k} \left\langle x^{k+1} - \bar{x}^k,\ x - x^{k+1} \right\rangle 
            - \frac{\mu_f}{2} \| x^{k+1} - x \|^2
			- \frac{\mu_f}{2} \| x^{k+1} - x^k \|^2\nonumber\\
			& \quad + \frac{\delta_k}{\alpha \tau_k} \left\langle x^k - \bar{x}^k,\ x^{k+1} - x^k \right\rangle
			+ \frac{1}{\beta_k \tau_k} \left\langle y^k - y^{k-1},\ y - y^k \right\rangle \label{eq52}\\
			& \quad + \left\langle \nabla h(x^k),\ x^k - x \right\rangle
			  + \left\langle g(x^k),\ y \right\rangle - \left\langle g(x),\ y^k \right\rangle \nonumber\\
			& \quad + \left\langle g(\bar{x}^{k}) - \alpha\left(g(\bar{x}^{k-1}) - g(x^k)\right),\ y^k - y \right\rangle.\nonumber
		\end{align}
		Following an argument similar to that used to derive \eqref{eq22} in Lemma \ref{lem4}, (\ref{eq52}) can be bounded as
		\begin{equation}\label{eq521}
			\begin{aligned}
				\tau_k Q_k(x,y)
                &\le \left\langle x^{k+1}-\bar{x}^k,\ x-x^{k+1} \right\rangle
                + \frac{\delta_k}{\alpha}\left\langle x^k-\bar{x}^k,\ x^{k+1}-x^k \right\rangle \\
                &\quad + \frac{1}{\beta_k}\left\langle y^k-y^{k-1},\ y-y^k \right\rangle
                + \tau_k A_k(x,y) + \tau_k B_k(x,y) \\
                &\quad - \frac{\mu_f\tau_k}{2}\|x^{k+1}-x\|^2
                - \frac{\mu_f\tau_k}{2}\|x^{k+1}-x^k\|^2,
			\end{aligned}
		\end{equation}
        where $A_k(x,y)$ and $B_k(x,y)$ are defined by
        \[
        \begin{aligned}
            A_k(x,y) &= \left\langle \nabla_x \Phi(\bar{x}^k,y^k), x-x^{k+1} \right\rangle + \left\langle \nabla_x \Phi(\bar{x}^{k-1},y^{k-1}), x^{k+1}-x^k \right\rangle \\
            &\quad + \left\langle g(\bar{x}^k) -\alpha \bigl(g(\bar x^{k-1})-g(x^k)\bigr), y^k-y \right\rangle + \Phi(x^k,y)-\Phi(x,y^k),\\
            B_k(x,y) &= \left\langle \nabla h(x^k), x-x^{k+1} \right\rangle + \left\langle \nabla h(x^{k-1}), x^{k+1}-x^k \right\rangle + \left\langle \nabla h(x^k), x^k-x \right\rangle.
        \end{aligned}
        \]
        In fact, the definitions of $A_k(x,y)$ and $B_k(x,y)$ are consistent with those in Lemma \ref{lem4}. Using the decomposition established in Lemma \ref{lem4} and applying Young's inequality with a different scaling parameter, the following estimates are obtained.
        \[
        \begin{aligned}
            A_k(x,y) &\le \frac12\left(L\left(\alpha+\frac{(1-\alpha)^2}{\alpha^2}\right)+\frac{5\xi M}{\tau_{k-1}}+\ell\|y\|\frac{1-\alpha}{\alpha}\right)\|x^k-\bar{x}^k\|^2 \\
            &\quad + \frac{M\tau_{k-1}}{2\xi}\|y^k-y^{k-1}\|^2
            + \frac12\left(L+\frac{5\xi M}{\tau_{k-1}}\right)\|x^{k+1}-x^k\|^2,\\
            B_k(x,y) &\le \frac{L_h}{2}\|x^k-x^{k-1}\|^2 + \frac{L_h}{2}\|x^{k+1}-x^k\|^2.
        \end{aligned}
        \]
		where $\frac{\tau_{k-1}}{\xi}$ denotes a scaling weight parameter derived from Young's inequality and $L=\ell B$ in Lemma \ref{lem4}. Substituting the estimates of $A_k(x,y)$ and $B_k(x,y)$ back into the inequality obtained by multiplying \eqref{eq521} by 2 and rearranging terms of (\ref{eq521}), we obtain
		\begin{equation}\label{eq53}
			\begin{aligned}
				& 2\tau_k Q_k(x,y) + \frac{1+\mu_f\tau_k}{1-\alpha}\|\bar{x}^{k+1}-x\|^2 + \frac{1}{\beta_k}\|y^k-y\|^2 \\
                \le & \frac{1+\alpha\mu_f\tau_k}{1-\alpha}\|\bar{x}^k-x\|^2 + \frac{1}{\beta_k}\|y^{k-1}-y\|^2 + L_h\tau_k\|x^k-x^{k-1}\|^2 \\
                & - \left(\frac{1}{\beta_k}-\frac{M\tau_k\tau_{k-1}}{\xi}\right)\|y^k-y^{k-1}\|^2 \\
                & - \left(1+\alpha+\alpha\mu_f\tau_k-\frac{\delta_k}{\alpha}\right)\|x^{k+1}-\bar{x}^k\|^2 \\
                & - \left(\delta_k\left(\frac1\alpha-5\xi M\right)-\tau_k L\left(\alpha+\frac{(1-\alpha)^2}{\alpha^2}\right)-\tau_k\ell\|y\|\frac{1-\alpha}{\alpha}\right)\|x^k-\bar{x}^k\|^2 \\
                & - \left(\delta_k\left(\frac1\alpha-5\xi M\right)+\mu_f\tau_k-\tau_k\left(L+L_h\right)\right)\|x^{k+1}-x^k\|^2.
			\end{aligned}
		\end{equation}
        
		Since $\frac{1+\mu_f\tau_k}{1-\alpha} = \frac{1+\mu_f\tau_k}{1+\mu_f\alpha\tau_{k+1}}\frac{1+\mu_f\alpha\tau_{k+1}}{1-\alpha}$, and 
		\[
		\frac{1+\mu_f\tau_k}{1+\mu_f\alpha\tau_{k+1}}\geq \frac{1+\mu_f\tau_k}{1+\mu_f\alpha\varphi\tau_k} = 1 + \frac{1 - \varphi\alpha}{1 + \mu_f \alpha \varphi \tau_k}\mu_f\tau_k = 1 + \rho_{k+1}\mu_f\tau_k,
		\]
		where the inequality follows from $\tau_{k+1} \le \varphi \tau_k$. By definition of $\beta_{k+1}$, it follows that
		\begin{equation}\label{eq54}
			\frac{1+\mu_f\tau_k}{1-\alpha}\geq (1 + \rho_{k+1}\mu_f\tau_k)\frac{1+\mu_f\alpha\tau_{k+1}}{1-\alpha} = \frac{\beta_{k+1}}{\beta_k}\frac{1+\mu_f\alpha\tau_{k+1}}{1-\alpha}.
		\end{equation}
		using \eqref{eq54} in \eqref{eq53} and multiplying both sides of the inequality by $\beta_k$ yields
		\begin{equation}\label{eq542}
			\begin{aligned}
				& 2\beta_k\tau_k Q_k(x,y)
                + \beta_{k+1}\frac{1+\alpha\mu_f\tau_{k+1}}{1-\alpha}\|\bar{x}^{k+1}-x\|^2
                + \|y^k-y\|^2\\
                \le & \beta_k\frac{1+\alpha\mu_f\tau_k}{1-\alpha}\|\bar{x}^k-x\|^2
                + \|y^{k-1}-y\|^2
                + L_h\beta_k\tau_k\|x^k-x^{k-1}\|^2 \\
                & - \beta_k\left(\frac{1}{\beta_k}-\frac{M\tau_k\tau_{k-1}}{\xi}\right)\|y^k-y^{k-1}\|^2 \\
                & - \beta_k\left(1+\alpha+\alpha\mu_f\tau_k-\frac{\delta_k}{\alpha}\right)\|x^{k+1}-\bar{x}^k\|^2 \\
                & - \beta_k\left(\delta_k\left(\frac1\alpha-5\xi M\right)-\tau_k L\left(\alpha+\frac{(1-\alpha)^2}{\alpha^2}\right)-\tau_k\ell\|y\|\frac{1-\alpha}{\alpha}\right)\|x^k-\bar{x}^k\|^2 \\
                & - \beta_k\left(\delta_k\left(\frac1\alpha-5\xi M\right)+\mu_f\tau_k-\tau_k\left(L+L_h\right)\right)\|x^{k+1}-x^k\|^2.
			\end{aligned}
		\end{equation}
		Adding $L_h\tau_{k+1} \beta_{k+1} \|x^{k+1} - x^k\|^2$ to both sides of inequality (\ref{eq542}), and using the definition of $E_k(x,y)$ gives the desired inequality.
		
	\end{proof}

    \section{Proof of Lemma \ref{lem10}}\label{app_lem10}
    \begin{proof}
        By the nonexpansiveness of the proximal mapping in Lemma~\ref{lem1} and the saddle point optimality condition,
        \[
        \|x^1-x^*\|\leq d_0+\widehat{\tau}c_0,\qquad
        \|\bar{x}^1-x^*\|\leq d_0+(1-\alpha)\widehat{\tau}c_0,
        \]
        and
        \[
        \|x^1-x^0\|\leq 2d_0+\widehat{\tau}c_0.
        \]
        Since $\beta_1/\beta_0<1/(\alpha\varphi)$ and $\tau_1\leq\varphi\widehat{\tau}$, the definition of $\overline E$ gives
        \[
        E_1(x^*,y^*)\leq\overline E.
        \]
        Consequently,
        \[
        \|\bar{x}^1-x^*\|
        \leq\sqrt{\frac{(1-\alpha)\overline E}{\beta_1}}
        \leq R_0,\qquad
        \|y^0-y^*\|\leq\sqrt{\overline E}\leq R_Y.
        \]
        Together with $\bar{x}^0=x^0$ and the convex combination identity, this yields
        \[
        \|x^1-x^*\|
        \leq\frac{\|\bar{x}^1-x^*\|+\alpha\|x^0-x^*\|}{1-\alpha}
        \leq\frac{1+\alpha}{1-\alpha}R_0
        \leq R_X.
        \]
        Hence, $x^0,\bar{x}^0,x^1,\bar{x}^1\in\mathcal X$ and $y^0\in\mathcal Y$.
        
        Suppose that the preceding iterates belong to the candidate regions and $E_k(x^*,y^*)\leq\overline E$ for some $k\geq1$. By the induction hypothesis, $\|y^{k-1}\|\leq B$ and by the saddle point property, $Q_k(x^*,y^*)\geq0$. 
        By Lemma~\ref{lem8} and the stepsize update,
        \[
        \alpha\leq\delta_k\leq\varphi,\qquad
        M\beta_k\tau_k\tau_{k-1}\leq\xi,\qquad
        \tau_{k+1}\frac{\beta_{k+1}}{\beta_k}\leq\frac{\tau_k}{\alpha}.
        \]
        The coefficients of $\|y^k-y^{k-1}\|^2$ and $\|x^{k+1}-\bar{x}^k\|^2$ in Lemma~\ref{lem9} are therefore nonnegative. The coefficient of $\beta_k\|x^k-\bar{x}^k\|^2$ is bounded below by
        \[
        \begin{aligned}
            \delta_k \left(\frac1\alpha-5\xi M-\tau_k L\left(\alpha+\frac{(1-\alpha)^2}{\alpha^2}\right)\right)
            -\tau_k\ell\|y^*\|\frac{1-\alpha}{\alpha} \geq 1-5\alpha\xi M-\tau_{\max}C_1=d>0.
        \end{aligned}
        \]
        Similarly, the remaining descent coefficient of $\beta_k\|x^{k+1}-x^k\|^2$ is bounded below by
        \[
        \begin{aligned}
            &\delta_k\left(\frac1\alpha-5\xi M\right)
            +\mu_f\tau_k-\tau_k(L+L_h)
            -L_h\tau_{k+1}\frac{\beta_{k+1}}{\beta_k}
            \ge \alpha\left(\frac1\alpha-5\xi M\right)
            +\mu_f\tau_k-\tau_k(L+L_h)
            -\frac{L_h\tau_k}{\alpha}\\
            &=1-5\alpha\xi M
            -\tau_k\left(L+\frac{1+\alpha}{\alpha}L_h-\mu_f\right)\ge1-5\alpha\xi M
            -\tau_{\max}\left(L+\frac{1+\alpha}{\alpha}L_h-\mu_f\right)_+\\
            &=e>0.
        \end{aligned}
        \]
        The saddle point inequality and Lemma~\ref{lem9} consequently imply
        \begin{equation}\label{apdace_key_descent}
            E_{k+1}(x^*,y^*)+d\beta_k\|x^k-\bar{x}^k\|^2
            \leq E_k(x^*,y^*).
        \end{equation}
        In particular,
        \[
        E_{k+1}(x^*,y^*)\leq E_k(x^*,y^*)\leq\overline E.
        \]
        The definition of $E_{k+1}$ and $\beta_{k+1}\geq\beta_0$ yields
        \[
        \|\bar{x}^{k+1}-x^*\|
        \leq\sqrt{\frac{(1-\alpha)\overline E}{\beta_{k+1}}}
        \leq R_0,\qquad
        \|y^k-y^*\|\leq\sqrt{\overline E}\leq R_Y.
        \]
        Moreover,
        \[
        \begin{aligned}
            \|x^{k+1}-x^*\|
            \leq\frac{\|\bar{x}^{k+1}-x^*\|+\alpha\|\bar{x}^k-x^*\|}{1-\alpha}\leq\frac{1+\alpha}{1-\alpha}R_0
            \leq R_X.
        \end{aligned}
        \]
        Thus $x^{k+1},\bar{x}^{k+1}\in\mathcal X$ and $y^k\in\mathcal Y$. The invariant-region property and \eqref{apdace_descent} follow by induction. Summing \eqref{apdace_key_descent} over $k=1,\ldots,N$ gives
        \[
        d\sum_{k=1}^{N}\beta_k\|x^k-\bar{x}^k\|^2
        \leq E_1(x^*,y^*)-E_{N+1}(x^*,y^*)
        \leq E_1^*.
        \]
        Letting $N\to\infty$ establishes \eqref{apdace_summability}.

    \end{proof}

    \section{Proof of Lemma \ref{lem11}}\label{app_lem11}    
    \begin{proof}
        By Lemma \ref{lem9}, the coefficient of $\|x^k-\bar{x}^k\|^2$ is
        \[
        \delta_k\left(\frac1\alpha-5\xi M\right)-\tau_k L\left(\alpha+\frac{(1-\alpha)^2}{\alpha^2}\right)-\tau_k\ell\|y\|\frac{1-\alpha}{\alpha}.
        \]
        For arbitrary $y \in \operatorname{dom}(H^*)$,
        \[
        \delta_k\left(\frac1\alpha-5\xi M\right)-\tau_k L\left(\alpha+\frac{(1-\alpha)^2}{\alpha^2}\right)-\tau_k\ell\|y\|\frac{1-\alpha}{\alpha}
        \ge d-\tau_k\ell\frac{1-\alpha}{\alpha}\big(\|y\|-B\big).
        \]
        If $\|y\|\le B$, the right-hand side is at least $d$.
        If $\|y\|>B$, the coefficient is bounded below by
        \[
        d-\tau_k\ell\frac{1-\alpha}{\alpha}\big(\|y\|-B\big).
        \]
        The two cases are unified as
        \[
        \delta_k\left(\frac1\alpha-5\xi M\right)-\tau_k L\left(\alpha+\frac{(1-\alpha)^2}{\alpha^2}\right)-\tau_k\ell\|y\|\frac{1-\alpha}{\alpha}
        \ge d-\tau_k\ell\frac{1-\alpha}{\alpha}\big(\|y\|-B\big)_+.
        \]
        Therefore,
        \[
        \begin{aligned}
            &-\beta_k\left(\delta_k\left(\frac1\alpha-5\xi M\right)-\tau_k L\left(\alpha+\frac{(1-\alpha)^2}{\alpha^2}\right)-\tau_k\ell\|y\|\frac{1-\alpha}{\alpha}\right)\|x^k-\bar{x}^k\|^2 \\
            \le & -d\beta_k\|x^k-\bar{x}^k\|^2+\beta_k\tau_k\ell\frac{1-\alpha}{\alpha}\big(\|y\|-B\big)_+\|x^k-\bar{x}^k\|^2 \\
            \le & \beta_k\tau_k\ell\frac{1-\alpha}{\alpha}\big(\|y\|-B\big)_+\|x^k-\bar{x}^k\|^2.
        \end{aligned}
        \]
        Since the coefficients of the remaining three terms in Lemma \ref{lem9} are non-negative, those terms can be discarded. This completes the proof.
        
    \end{proof}
    
    \section{Proof of Theorem \ref{thm4}}\label{app_thm4}
    \begin{proof}
        Summing Lemma \ref{lem11} from $k=1$ to $N$ gives
        \[
        E_{N+1}(x,y)+2\sum_{k=1}^N\beta_k\tau_k Q_k(x,y)
        \le E_1(x,y)+\ell\frac{1-\alpha}{\alpha}\big(\|y\|-B\big)_+\sum_{k=1}^N\beta_k\tau_k\|x^k-\bar{x}^k\|^2.
        \]
        Using $\tau_k\le\tau_{\max}$, we obtain
        \[
        \sum_{k=1}^N\beta_k\tau_k\|x^k-\bar{x}^k\|^2
        \le \tau_{\max}\sum_{k=1}^N\beta_k\|x^k-\bar{x}^k\|^2
        \le \frac{\tau_{\max}E_1^*}{d}.
        \]
        Since $E_{N+1}(x,y)\ge 0$, it follows that
        \[
        2\sum_{k=1}^N\beta_k\tau_k Q_k(x,y)
        \le E_1(x,y) + \frac{\tau_{\max}\ell(1-\alpha)}{\alpha d}\big(\|y\|-B\big)_+E_1^*.
        \]
        By the convexity of $\mathcal{L}(\cdot,y)$ and concavity of $\mathcal{L}(x,\cdot)$,
        \[
        \mathcal{L}(\bar{X}_N,y)=\mathcal{L}\left(\frac{1}{W_N}\sum_{k=1}^N\beta_k\tau_k x^k,y\right)
        \le \frac{1}{W_N}\sum_{k=1}^N\beta_k\tau_k\mathcal{L}(x^k,y),
        \]
        \[
        \mathcal{L}(x,\bar{Y}_N)=\mathcal{L}\left(x,\frac{1}{W_N}\sum_{k=1}^N\beta_k\tau_k y^k\right)
        \ge \frac{1}{W_N}\sum_{k=1}^N\beta_k\tau_k\mathcal{L}(x,y^k).
        \]
        Subtracting the two inequalities yields
        \[
        \begin{aligned}
        \mathcal{L}(\bar{X}_N,y)-\mathcal{L}(x,\bar{Y}_N)
        &\le \frac{1}{W_N}\sum_{k=1}^N \beta_k\tau_k \big(\mathcal{L}(x^k,y)-\mathcal{L}(x,y^k)\big)
        =\frac{1}{W_N}\sum_{k=1}^N \beta_k\tau_k Q_k(x,y).
        \end{aligned}
        \]
        Combined with the preceding estimate,
        \[
        \mathcal{L}(\bar{X}_N,y)-\mathcal{L}(x,\bar{Y}_N)
        \le \frac{1}{2W_N}\left(E_1(x,y)+\frac{\tau_{\max}\ell(1-\alpha)}{\alpha d}\big(\|y\|-B\big)_+E_1^*\right).
        \]
        Taking the supremum $(x,y)\in \mathcal{S}$ and using the definition of $G_{\mathcal{S}}(x,y)$, we obtain
        \begin{equation}\label{eq450}
            G_{\mathcal{S}}(\bar{X}_N,\bar{Y}_N) \le \frac{1}{2W_N}\left( \sup_{(x,y)\in \mathcal{S}}E_1(x,y)+\frac{\tau_{\max}\ell(1-\alpha)}{\alpha d}\sup_{(x,y)\in \mathcal{S}}\big(\|y\|-B\big)_+ E_1^* \right).
        \end{equation}
        Moreover, $\sup_{(x,y)\in \mathcal{S}}(\|y\|-B)_+ = (\sup_{(x,y)\in \mathcal{S}}\|y\|-B)_+$. Since $S$ is bounded, we have
        \[
        \sup_{(x,y)\in \mathcal{S}}E_1(x,y)<\infty,\quad \sup_{(x,y)\in \mathcal{S}}\|y\|<\infty.
        \]
        Therefore, the right-hand side of \eqref{eq450} is a finite constant independent of $N$ multiplied by $1/W_N$.
        Since $W_N=\Omega(N^2)$, the proof of \eqref{m105} is completed.
        
        If $\sup_{(x,y)\in \mathcal{S}}\|y\|\le B$, then 
        $\big(\sup_{(x,y)\in \mathcal{S}}\|y\|-B\big)_+=0$, which yields \eqref{m106}.
    
        Lemma \ref{lem10} gives $E_{N+1}(x^*,y^*)\le E_1^*$. Hence,  $\frac{\beta_{N+1}(1+\alpha\mu_f\tau_{N+1})}{1-\alpha}\big\|\bar{x}^{N+1}-x^*\big\|^2\le E_1^*$. It follows that $\big\|\bar{x}^{N+1}-x^*\big\|^2 \le \frac{(1-\alpha)E_1^*}{\beta_{N+1}(1+\alpha\mu_f\tau_{N+1})}$. By Lemma \ref{lem8}, $\beta_{N+1}\ge c_2^2(N+1)^2$. Therefore,
        \begin{equation}
            \big\|\bar{x}^{N+1}-x^*\big\|^2 \le \frac{(1-\alpha)E_1^*}{c_2^2(N+1)^2}.
        \end{equation}
        This completes the proof of the theorem.
        
    \end{proof}
    \end{appendices}

\end{document}